\documentclass[oneside,english]{amsart}
\usepackage[latin9]{inputenc}
\usepackage{geometry}
\usepackage{color}
\usepackage{units}
\usepackage{amsthm}
\usepackage{amsbsy}
\usepackage{amssymb}
\usepackage{graphicx}
\usepackage{esint}

\makeatletter
\numberwithin{equation}{section}
\numberwithin{figure}{section}
\theoremstyle{plain}
\newtheorem{thm}{\protect\theoremname}
  \theoremstyle{plain}
  \newtheorem{cor}[thm]{\protect\corollaryname}
  \theoremstyle{plain}
  \newtheorem{lem}[thm]{\protect\lemmaname}
  \theoremstyle{remark}
  \newtheorem{rem}[thm]{\protect\remarkname}

\usepackage{etoolbox}
\patchcmd{\thmhead}{(#3)}{#3}{}{}

\makeatother

\usepackage{babel}
  \providecommand{\corollaryname}{Corollary}
  \providecommand{\lemmaname}{Lemma}
  \providecommand{\remarkname}{Remark}
\providecommand{\theoremname}{Theorem}

\begin{document}

\title{The complexity of spherical $p$-spin models - a second moment approach}

\author{Eliran Subag}
\begin{abstract}
Recently, Auffinger, Ben Arous, and {\v{C}}ern{\'y} initiated the
study of critical points of the Hamiltonian in the spherical pure
$p$-spin spin glass model, and established connections between those
and several notions from the physics literature. Denoting the number
of critical values less than $Nu$ by $\mbox{Crt}_{N}\left(u\right)$,
they computed the asymptotics of $\frac{1}{N}\log\left(\mathbb{E}\mbox{Crt}_{N}\left(u\right)\right)$,
as $N$, the dimension of the sphere, goes to $\infty$. We compute
the asymptotics of the corresponding second moment and show that,
for $p\geq3$ and sufficiently negative $u$, it matches the first
moment: 
\[
\mathbb{E}\left\{ \left(\mbox{Crt}_{N}\left(u\right)\right)^{2}\right\} /\left(\vphantom{\left(\mbox{Crt}_{N}\left(u\right)\right)^{2}}\mathbb{E}\left\{ \mbox{Crt}_{N}\left(u\right)\right\} \right)^{2}\to1.
\]
As an immediate consequence we obtain that $\mbox{Crt}_{N}\left(u\right)/\mathbb{E}\left\{ \mbox{Crt}_{N}\left(u\right)\right\} \to1$,
in $L^{2}$ and thus in probability. For any $u$ for which $\mathbb{E}\mbox{Crt}_{N}\left(u\right)$
does not tend to $0$ we prove that the moments match on an exponential
scale.
\end{abstract}

\maketitle

\section{Introduction}

The Hamiltonian of the spherical \emph{pure }$p$-spin spin glass
model is given by 
\begin{equation}
H_{N}\left(\boldsymbol{\sigma}\right):=H_{N,p}\left(\boldsymbol{\sigma}\right)=\frac{1}{N^{\left(p-1\right)/2}}\sum_{i_{1},...,i_{p}=1}^{N}J_{i_{1},...,i_{p}}\sigma_{i_{1}}\cdots\sigma_{i_{p}},\quad\boldsymbol{\sigma}\in\mathbb{S}^{N-1}\left(\sqrt{N}\right),\label{eq:Hamiltonian}
\end{equation}
where $\boldsymbol{\sigma}=\left(\sigma_{1},...,\sigma_{N}\right)$
, $\mathbb{S}^{N-1}\left(\sqrt{N}\right)\triangleq\left\{ \boldsymbol{\sigma}\in\mathbb{R}^{N}:\,\left\Vert \boldsymbol{\sigma}\right\Vert _{2}=\sqrt{N}\right\} $,
and $J_{i_{1},...,i_{p}}$ are i.i.d standard normal variables. Everywhere
in the paper we shall assume that $p\geq3$.%
\footnote{In the case $p=2$ the critical points of $H_{N}\left(\boldsymbol{\sigma}\right)$
are exactly the points $\boldsymbol{\sigma}\in\mathbb{S}^{N-1}\left(\sqrt{N}\right)$
which are eigenvectors of the matrix $\left(J_{i_{1},i_{2}}+J_{i_{2},i_{1}}\right)_{i_{1},i_{2}=1}^{N}$.
In particular, there are exactly $2N$ such points almost surely.%
} The model was introduced by Crisanti and Sommers \cite{pSPSG} as
a variant of the Ising $p$-spin spin glass model. Unlike the Ising
$p$-spin model, defined on the hypercube, the spherical $p$-spin
model is defined on a continuous space - a property they expected
to yield a model amenable to different methods of analysis, while
retaining the main features of the original model. A generalization
of the model called the spherical \emph{mixed }$p$-spin spin glass
model is obtained by setting the Hamiltonian to be $H_{N}\left(\boldsymbol{\sigma}\right)=\sum_{p\geq2}\beta_{p}H_{N,p}\left(\boldsymbol{\sigma}\right)$,
with $H_{N,p}\left(\boldsymbol{\sigma}\right)$ being independent
pure $p$-spin models and $\beta_{p}\geq0$ (such that the sum is
defined). 

Recently, Auffinger, Ben Arous, and {\v{C}}ern{\'y} \cite{A-BA-C}
suggested to study the critical points of the Hamiltonian of the spherical
pure $p$-spin model in order to understand its landscape. Their work
was later extended \cite{ABA2} to the mixed case. The main results
of \cite{A-BA-C} on the complexity of the Hamiltonian for the pure
$p$-spin model are as follows. Let $\mbox{Crt}_{N}\left(B\right)$
denote the number of critical points of $H_{N}\left(\boldsymbol{\sigma}\right)$
at which $H_{N}\left(\boldsymbol{\sigma}\right)/N$ lies in a Borel
set $B\subset\mathbb{R}$ (cf. (\ref{eq:25})). Use the notation $\mbox{Crt}_{N,k}\left(B\right)$
for the number of such critical points with index $k$. It was shown
in \cite{A-BA-C} that 
\begin{align}
\lim_{N\to\infty}\frac{1}{N}\log\left(\mathbb{E}\left\{ \mbox{Crt}_{N}\left(\left(-\infty,u\right)\right)\right\} \right) & =\Theta_{p}\left(u\right),\label{eq:e1}\\
\lim_{N\to\infty}\frac{1}{N}\log\left(\mathbb{E}\left\{ \mbox{Crt}_{N,k}\left(\left(-\infty,u\right)\right)\right\} \right) & =\Theta_{p,k}\left(u\right),\label{eq:e2}
\end{align}
where $\Theta_{p}\left(u\right)$ and $\Theta_{p,k}\left(u\right)$
are known non-decreasing functions (cf. Theorem \ref{thm:A-BA-C}).
Moreover, with $E_{k}\left(p\right)$ being equal to the unique number
satisfying $\Theta_{p,k}\left(-E_{k}\left(p\right)\right)=0$, 
\[
E_{0}\left(p\right)>E_{1}\left(p\right)>E_{2}\left(p\right)>\cdots,\,\,\mbox{and}\,\,\lim_{k\to\infty}E_{k}\left(p\right)=E_{\infty}\left(p\right)\triangleq2\sqrt{\frac{p-1}{p}},
\]
and for each $k$ and closed set $B\subset\mathbb{R}$ such that $B$
and $\left[-E_{k}\left(p\right),-E_{\infty}\left(p\right)\right]$
are disjoint, $\mathbb{P}\left\{ \mbox{Crt}_{N,k}\left(B\right)>0\right\} $
decays (at least) exponentially in $N$. In addition, they showed
that for $u<-E_{\infty}\left(p\right)$, $\Theta_{p}\left(u\right)=\Theta_{p,0}\left(u\right)$,
which, in particular, implies that for any $\epsilon>0$, with high
probability
\begin{equation}
\mbox{Crt}_{N}\left(\left(-\infty,-E_{0}\left(p\right)-\epsilon\right)\right)=0.\label{eq:88}
\end{equation}

The computation of the means is certainly a significant step in the
investigation of the critical points. However, \emph{by themselves},
the means give very limited information on the probabilistic law of
the corresponding variables. Essentially, they can only be used to
obtain (by appealing to Markov's inequality) the upper bounds on (\ref{eq:88})
stated above. A question that naturally arises is: are the corresponding
variables concentrated around their means? In the general context
of spherical mixed $p$-spin models this is not necessarily the case:
for a subclass of models termed by \cite{ABA2}\emph{ full mixture
}models, there is a range of levels $u$, such that the mean number
of critical points in $\left(-\infty,u\right)$ is exponentially high,
while the probability of having a critical point in $\left(-\infty,u\right)$
goes to zero (see \cite[Corollary 4.1]{ABA2}). 

Focusing on the pure case and on the number of critical points of
general index $\mbox{Crt}_{N}\left(\cdot\right)$, we establish that
the answer to the above is positive. This is done, as suggested in
\cite[p. 2]{A-BA-C}, by computing \textcolor{black}{the second moment
in addition to the already known first moment. }
\begin{thm}
\label{thm:Var-E2}For any $p\geq3$ and $u\in\left(-E_{0}\left(p\right),-E_{\infty}\left(p\right)\right)$,
\begin{equation}
\lim_{N\to\infty}\frac{\mathbb{E}\big\{\left({\rm Crt}_{N}\left(\left(-\infty,u\right)\right)\right)^{2}\big\}}{\big(\mathbb{E}\left\{ {\rm Crt}_{N}\left(\left(-\infty,u\right)\right)\right\} \big)^{2}}=1.\label{eq:intro-1}
\end{equation}

\end{thm}
As an immediate corollary we obtain the following.
\begin{cor}
\label{cor:intro}For any $p\geq3$ and $u\in\left(-E_{0}\left(p\right),-E_{\infty}\left(p\right)\right)$,
\[
\lim_{N\to\infty}\frac{{\rm Crt}_{N}\left(\left(-\infty,u\right)\right)}{\mathbb{E}\left\{ {\rm Crt}_{N}\left(\left(-\infty,u\right)\right)\right\} }=1,
\]
in $L_{2}$, and thus, also in probability.
\end{cor}
\textcolor{black}{The main motivation for the study of the Gaussian
fields }$H_{N,p}\left(\boldsymbol{\sigma}\right)$ \textcolor{black}{is
their importance in the physics literature. Nevertheless, the model
certainly serves as a natural setting to investigate a question of
pure mathematical interest: what is the behavior of the critical points
of an isotropic random function on a high dimensional manifold? To
the best of our knowledge, the corollary above (combined with the
computation of the first moment of \cite{A-BA-C}) is the first concentration
result for the high dimensional limit. }

Computations of moments of the\emph{ }number of critical points were
done in other settings. Closest to our setting are the works of \textcolor{black}{Fyodorov
\cite{Fyo1,Fyo2} which dealt with isotropic fields on the sphere
$\mathbb{S}^{N}$ and on $\mathbb{R}^{N}$ and the first moment of
number of critical points and its large $N$ asymptotics. Further
away, are the works of Nicolaescu \cite{Nic5,Nic4,Nic3,Nic2,Nic1},
Sarnak and Wigman \cite{SrnakWigman}, Cammarota, Marinucci and Wigman
\cite{Cammarota2015,CMWigman}, Douglas, Shiffman, and Zelditch \cite{Zel1,Zel2,Zel3},
Baugher \cite{Baugh}, and Feng and Zelditch \cite{Zel4}. Those concerned
Gaussian fields on a fixed space and asymptotics in parameters of
different nature than the dimension, e.g. ones related to roughness
of the random field by adding functions of higher frequency to a random
expansion. In \cite{Nic5,Nic3,Cammarota2015,CMWigman} concentration
results were also derived by second moment computations. Lastly, we
mention works on nodal domains of Gaussian fields. See for example
Nazarov and Sodin \cite{NazarovSodin1,NazarovSodin2} and references
therein.}

For any $u$ for which $\mathbb{E}\left\{ {\rm Crt}_{N}\left(\left(-\infty,u\right)\right)\right\} $
does not tend to $0$, we show that the moments match on an exponential
scale.
\begin{thm}
\label{thm:Var-E2-log}For any $p\geq3$ and $u\in\left(-E_{0}\left(p\right),\infty\right)$,

\begin{align}
\lim_{N\to\infty}\frac{1}{N}\log\left(\mathbb{E}\left\{ \left({\rm Crt}_{N}\left(\left(-\infty,u\right)\right)\right)^{2}\right\} \right) & =2\lim_{N\to\infty}\frac{1}{N}\log\left(\mathbb{E}\left\{ {\rm Crt}_{N}\left(\left(-\infty,u\right)\right)\right\} \right)=2\Theta_{p}\left(u\right),\label{eq:intro-2}
\end{align}
where $\Theta_{p}\left(u\right)$ is given in (\ref{eq:54}).
\end{thm}
Connections between the critical points and two important notions
from the physics literature were established in \cite{A-BA-C,ABA2}:
the Thouless-Anderson-Palmer (TAP) equations and the free energy.
The TAP approach suggests that `pure states' of the system can be
identified with critical points of the so-called TAP functional \cite{TAP}.
One of the main objects of interest in the analysis using this approach
is the TAP-complexity - that is, the logarithm of the number of solutions
of the TAP equations. The TAP-complexity has been extensively studied
in the physics literature in the context of the Sherrington-Kirkpatrick
model \cite{TAP-SK1,TAP-SK2,TAP-SK3,TAP-SK4,TAP-SK5}, the Ising $p$-spin
spin glass model \cite{TAP-pS-Ising1,TAP-pS-Ising2,TAP-pS-Ising3},
and the spherical $p$-spin spin glass model \cite{TAP-pSPSG1,TAP-pSPSG2,TAP-pSPSG3,TAP-pSPSG4}.
The connection to critical points of the Hamiltonian is based on the
observation of \cite{A-BA-C} (see Section 6 there for more details)
that each critical point of the Hamiltonian corresponds to exactly
two solutions of the TAP equations - meaning that a study of the critical
points is equivalent to a study of the TAP complexity. 

Another interesting link that \cite{A-BA-C,ABA2} found is related
to the ground state 
\begin{equation}
GS^{\infty}=\lim_{N\to\infty}GS^{N}\triangleq\lim_{N\to\infty}\frac{1}{N}\min_{\boldsymbol{\sigma}}H_{N}\left(\boldsymbol{\sigma}\right).\label{eq:gs1}
\end{equation}
The limiting free energy $F\left(\beta\right)$ is known to exist
and is given by the Parisi formula \cite{Parisi,pSPSG}, proved in
\cite{Talag,Chen}. The formula expresses $F\left(\beta\right)$ through
an intricate variational problem, which is greatly simplified when
one-step replica symmetry breaking (1-RSB) is known to occur (see
\cite{Talag2} for a definition of this terminology). In Section 4
of their work, \cite{ABA2} define the class of\emph{ pure-like} spherical
$p$-spin models and prove for it that
\begin{equation}
E_{0}\geq-GS^{\infty}=\lim_{\beta\to\infty}\frac{1}{\beta}F\left(\beta\right)\leq\lim_{\beta\to\infty}\frac{1}{\beta}F^{{\scriptstyle \mbox{\ensuremath{{\scriptstyle 1RSB}}}}}\left(\beta\right)=E_{0},\label{eq:q2}
\end{equation}
where $F^{{\scriptstyle \mbox{\ensuremath{{\scriptstyle 1RSB}}}}}\left(\beta\right)$
is defined to be the free energy obtained from the Parisi formula
under the assumption that 1-RSB occurs.

Therefore, if 1-RSB is exhibited, i.e., the second inequality above
holds as equality, then $GS^{\infty}=-E_{0}$, and the first moment
computation (\ref{eq:e1}) gives the ground state. Using the fact
that pure spherical $p$-spin models are known to exhibit 1-RSB \cite[Proposition 2.2]{Talag},
\cite{A-BA-C} proved that $GS^{\infty}=-E_{0}$. Note that, since
$-E_{0}\leq GS^{\infty}$, in order to prove that $GS^{\infty}=-E_{0}$
only a corresponding reversed inequality is needed. In particular,
proving that w.h.p ${\rm Crt}_{N}\left(\left(-\infty,-E_{0}+\epsilon\right)\right)\geq1$,
for any $\epsilon>0$, is sufficient. Corollary \ref{cor:intro} implies
this, and in fact since $H_{N}\left(\boldsymbol{\sigma}\right)$ is
a Gaussian field, using concentration inequalities even Theorem \ref{thm:Var-E2-log}
is sufficient; see Appendix IV. This gives an alternative derivation
of the result of \cite{A-BA-C} without going through Parisi's formula.

Generally, mixed spherical $p$-spin models do not necessarily exhibit
1-RSB. But, if we are able to compute second moments and prove (\ref{eq:intro-2})
for some mixture, then it would follow that $GS^{\infty}=-E_{0}$
and, by (\ref{eq:q2}), that ``1-RSB in the zero-temperature limit''
occurs. This will be explored in future work, where we shall consider
part of the mixed case regime.

We finish with a remark about two recent works which build on the
concentration result for the critical points which we prove in the
current paper. In the first, Zeitouni and the author \cite{pspinext}
investigate the extremal point process of critical points - that is,
the point process constructed from critical values in the vicinity
of the global minimum of $H_{N}(\boldsymbol{\sigma})$ - and establish
its convergence to a Poisson point process of exponential density.
As a corollary they also obtain that the global minimum (without normalization,
in contrary to (\ref{eq:gs1})) converges to minus a Gumbel variable.
In the second work, the author \cite{geometryGibbs} relates the Gibbs
measure at low temperature to the critical points and shows that the
measure is supported on spherical `bands' around the deepest minima
of $H_{N}(\boldsymbol{\sigma})$, i.e. those of which the extremal
process consists. This allows one to derive interesting consequences,
for example the absence of temperature chaos and precise asymptotics
of the free energy.

In the next section we introduce notation. In Section \ref{sec:results}
we outline the proofs of Theorems \ref{thm:Var-E2} and \ref{thm:Var-E2-log}
and state several related auxiliary results. The rest of the paper
is devoted to proofs of the theorems stated above and those auxiliary
results. When stating each of the latter we will also point out where
its proof is given. The proof of Theorem \ref{thm:Var-E2-log} is
given is Section \ref{sec:Proofs}. Theorem \ref{thm:Var-E2} is proved
in Section \ref{sec:Finer-Asymptotics}.

\subsection*{Acknowledgments}

I am grateful to my adviser Ofer Zeitouni for introducing me to the
problem of computing the second moment and for his help through all
stages of the work. I would also like to thank G\'{e}rard Ben Arous
for helpful discussions. This work is supported by the Adams Fellowship
Program of the Israel Academy of Sciences and Humanities.

\section{\label{sec:Notation}Notation}

For any two points $\boldsymbol{\sigma}$, $\boldsymbol{\sigma}'$
on the sphere, define the overlap function 
\begin{equation}
R\left(\boldsymbol{\sigma},\boldsymbol{\sigma}'\right)\triangleq\frac{\left\langle \boldsymbol{\sigma},\boldsymbol{\sigma}'\right\rangle }{\left\Vert \boldsymbol{\sigma}\right\Vert _{2}\left\Vert \boldsymbol{\sigma}'\right\Vert _{2}}=\frac{\sum_{i=1}^{N}\sigma_{i}\sigma_{i}^{\prime}}{N}.\label{eq:overlap}
\end{equation}
Adopting the notation of \cite{A-BA-C}, for any Borel set $B\subset\mathbb{R}$,
let $\mbox{Crt}_{N}\left(B\right)$ denote the number of critical
points of $H_{N}$, at which it attains a value in $NB=\left\{ Nx:\, x\in B\right\} $:
\begin{equation}
\mbox{Crt}_{N}\left(B\right)\triangleq\#\left\{ \left.\boldsymbol{\sigma}\in\mathbb{S}^{N-1}\left(\sqrt{N}\right)\,\right|\,\nabla H_{N}\left(\boldsymbol{\sigma}\right)=0,\, H_{N}\left(\boldsymbol{\sigma}\right)\in NB\right\} ,\label{eq:25}
\end{equation}
where $\nabla H_{N}\left(\boldsymbol{\sigma}\right)$ denotes the
gradient of $H_{N}\left(\boldsymbol{\sigma}\right)$ (relative to
the standard differential structure on the sphere). We will also be
concerned with the number of ordered pairs of points $(\boldsymbol{\sigma},\boldsymbol{\sigma}')\in\left(\mbox{Crt}_{N}\left(B\right)\right)^{2}$
with overlap in some range. For any subset $I_{R}\subset\left[-1,1\right]$,
we define

\[
\left[\mbox{Crt}_{N}\left(B,I_{R}\right)\right]_{2}\triangleq\#\left\{ \left.\left(\boldsymbol{\sigma},\boldsymbol{\sigma}'\right)\in\left(\mbox{Crt}_{N}\left(B\right)\right)^{2}\,\right|\, R\left(\boldsymbol{\sigma},\boldsymbol{\sigma}'\right)\in I_{R}\right\} .
\]
Note that $\mathbb{E}\left[\mbox{Crt}_{N}\left(B,I_{R}\right)\right]_{2}$
is the `contribution' of pairs with $R\left(\boldsymbol{\sigma},\boldsymbol{\sigma}'\right)\in I_{R}$
to the second moment of $\mbox{Crt}_{N}\left(B\right)$ (and that,
in particular, when $I_{R}=\left[-1,1\right]$, the full range of
the overlap, it is equal to the second moment). In the sequel we shall
assume that each of $B$ and $I_{R}$ is a finite union of non-degenerate
open intervals in $\mathbb{R}$. In this case we shall say that $B$
(or $I_{R}$) is `nice'.

A random matrix $\mathbf{X}_{N}$ from the (normalized) $N\times N$
Gaussian orthogonal ensemble, or an $N\times N$ GOE matrix, for short,
is a real, symmetric matrix such that all elements are centered Gaussian
variables which, up to symmetry, are independent with variance given
by
\[
\mathbb{E}\left\{ \mathbf{X}_{N,ij}^{2}\right\} =\begin{cases}
1/N, & \, i\neq j\\
2/N, & \, i=j.
\end{cases}
\]

Denote the surface area of the $N-1$-dimensional unit sphere by 
\[
\omega_{N}=\frac{2\pi^{N/2}}{\Gamma\left(N/2\right)}.
\]

Let $\mu^{*}$ denote the semicircle measure, the density of which
with respect to Lebesgue measure is 
\begin{equation}
\frac{d\mu^{*}}{dx}=\frac{1}{2\pi}\sqrt{4-x^{2}}\mathbf{1}_{\left|x\right|\leq2},\label{eq:semicirc}
\end{equation}
and define the function (see, e.g., \cite[Proposition II.1.2]{logpotential})
\begin{align}
\Omega(x) & \triangleq\int_{\mathbb{R}}\log\left|\lambda-x\right|d\mu^{*}\left(\lambda\right)\label{eq:Omega}\\
 & =\begin{cases}
\frac{x^{2}}{4}-\frac{1}{2} & \mbox{ if }0\leq\left|x\right|\leq2,\\
\frac{x^{2}}{4}-\frac{1}{2}-\left[\frac{\left|x\right|}{4}\sqrt{x^{2}-4}-\log\left(\sqrt{\frac{x^{2}}{4}-1}+\frac{\left|x\right|}{2}\right)\right] & \mbox{ if }\left|x\right|>2.
\end{cases}\nonumber 
\end{align}

Lastly, set 
\begin{align}
\Psi_{p}\left(r,u_{1},u_{2}\right) & \triangleq1+\log\left(p-1\right)+\frac{1}{2}\log\left(\frac{1-r^{2}}{1-r^{2p-2}}\right)\label{eq:Psi}\\
 & -\frac{1}{2}\left(u_{1},u_{2}\right)\left(\Sigma_{U}\left(r\right)\right)^{-1}\left(\begin{array}{c}
u_{1}\\
u_{2}
\end{array}\right)+\Omega\left(\sqrt{\frac{p}{p-1}}u_{1}\right)+\Omega\left(\sqrt{\frac{p}{p-1}}u_{2}\right),\nonumber 
\end{align}
where $\Sigma_{U}\left(r\right)$ is defined in (\ref{eq:26}).

\section{\label{sec:results}Outline of proofs and Auxiliary results}

As in the calculation of the first moment \cite{A-BA-C}, or in fact
any of the moment calculations for critical points mentioned below
Corollary \ref{cor:intro}, the starting point of our analysis is
an application of (a variant of) the Kac-Rice formula (henceforth,
K-R formula). The formula expresses the expectation of $\left[\mbox{Crt}_{N}\left(B,I_{R}\right)\right]_{2}$
as an integral over $I_{R}$ and combined with a study of certain
conditional laws, in particular those of the Hessians of the Hamiltonian
at two different points $\boldsymbol{\sigma}$ and $\boldsymbol{\sigma}'$,
yields the following lemma, proved in Section \ref{sec:ExactFormula:KR+Hessians}.
\begin{lem}
\label{lem:2ndKR-refinedFrom}Let $\left(U_{1}\left(r\right),U_{2}\left(r\right)\right)\sim N\left(0,\Sigma_{U}\left(r\right)\right)$
(cf. (\ref{eq:26})) be a Gaussian vector independent of $\hat{\mathbf{M}}_{N-1}^{\left(i\right)}\left(r\right)$,
$i=1,2$, defined in Lemma \ref{lem:Hess_struct_2}. Let $\mathcal{\mathbf{M}}_{N-1}^{\left(i\right)}\left(r,U_{1}\left(r\right),U_{2}\left(r\right)\right)$
be defined by (\ref{eq:100}). Then for any nice $B\subset\mathbb{R}$
and $I_{R}\subset\left(-1,1\right)$, 
\end{lem}
\begin{align}
\mathbb{E}\left\{ \left[{\rm Crt}_{N}\left(B,I_{R}\right)\right]_{2}\right\}  & =C_{N}\int_{I_{R}}dr\cdot\left(\mathcal{G}\left(r\right)\right)^{N}\mathcal{F}\left(r\right)\nonumber \\
 & \times\mathbb{E}\left\{ \prod_{i=1,2}\left|\det\left(\mathcal{\mathbf{M}}_{N-1}^{\left(i\right)}\left(r,U_{1}\left(r\right),U_{2}\left(r\right)\right)\right)\right|\cdot\mathbf{1}\Big\{ U_{1}\left(r\right),U_{2}\left(r\right)\in\sqrt{N}B\Big\}\right\} ,\label{eq:2ndmom_refinedform}
\end{align}
where
\begin{align}
C_{N} & =\omega_{N}\omega_{N-1}\left(\frac{\left(N-1\right)\left(p-1\right)}{2\pi}\right)^{N-1},\,\,\,\mathcal{G}\left(r\right)=\left(\frac{1-r^{2}}{1-r^{2p-2}}\right)^{\frac{1}{2}},\label{eq:cgf}\\
\mathcal{F}\left(r\right) & =\left(\mathcal{G}\left(r\right)\right)^{-3}\left(1-r^{2p-2}\right)^{-\frac{1}{2}}\left(1-\left(pr^{p}-\left(p-1\right)r^{p-2}\right)^{2}\right)^{-\frac{1}{2}}.\nonumber 
\end{align}
The analysis of the ratio of the second to first moment squared splits
into two parts - analysis of the asymptotics on the exponential scale
and a refinement to $O(1)$ scale. We shall now discuss the first
part. Lemma \ref{lem:Hess_struct_2} implies that the (correlated)
random matrices $\mathcal{\mathbf{M}}_{N-1}^{\left(i\right)}\left(r,U_{1}\left(r\right),U_{2}\left(r\right)\right)$
satisfy, in distribution, 
\begin{equation}
\left(\begin{array}{c}
\vphantom{\sqrt{\frac{1}{N-1}\frac{p}{p-1}}}\mathbf{M}_{N-1}^{\left(1\right)}\left(r,U_{1}\left(r\right),U_{2}\left(r\right)\right)\\
\vphantom{\sqrt{\frac{1}{N-1}\frac{p}{p-1}}}\mathbf{M}_{N-1}^{\left(2\right)}\left(r,U_{1}\left(r\right),U_{2}\left(r\right)\right)
\end{array}\right)=\left(\begin{array}{c}
\mathbf{X}_{N-1}^{\left(1\right)}(r)-\sqrt{\frac{1}{N-1}\frac{p}{p-1}}U_{1}\left(r\right)I+\mathbf{E}_{N-1}^{\left(1\right)}(r)\\
\mathbf{X}_{N-1}^{\left(2\right)}(r)-\sqrt{\frac{1}{N-1}\frac{p}{p-1}}U_{2}\left(r\right)I+\mathbf{E}_{N-1}^{\left(2\right)}(r)
\end{array}\right),\label{eq:Ms}
\end{equation}
where $\mathbf{X}_{N-1}^{\left(i\right)}(r)$ are correlated GOE matrices
independent of $(U_{1}\left(r\right),U_{2}\left(r\right))$ and $\mathbf{E}_{N-1}^{\left(i\right)}(r)$
are random matrices of rank $2$ viewed as perturbations. On the exponential
level the rank $2$ perturbations are easily dealt with by upper bounding
their Hilbert-Schmidt norm (see Lemmas \ref{cor:r2pert} and \ref{lem:W bound}).
We remark that in parallel to the above, in the computation of the
first moment of \cite{A-BA-C} the determinant of a single shifted
GOE matrix appears in the corresponding K-R formula. There, a certain
algebraic identity related to the density of the eigenvalues of a
GOE matrix, together with Selberg's integral formula, is key to the
analysis. In our situation explicit computations such as Selberg's
formula cannot be used because of the presence of two correlated GOE
matrices. Instead, the main tool we use to upper bound the product
of determinants is the large deviation principle (LDP) satisfied by
the empirical measure of eigenvalues proved in \cite[Theorem 2.1.1]{BAG97}
(see Theorem \ref{lem:GOELD}). Of course, $\frac{1}{N}\log$ of the
absolute value of the determinant is a linear statistic of the eigenvalues
$\lambda_{i}$, namely, it is equal to $\frac{1}{N}\sum\log|\lambda_{i}|$.
Combining this with the LDP, Varadhan's integral lemma \cite[Theorem 4.3.1, Exercise 4.3.11]{LDbook},
and a truncation argument (to control extremely large or close to
$0$ eigenvalues), we derive the following theorem in Section \ref{sec:Logarithmic-asymptotics}.
We stress that the fact that the LDP is at speed $N^{2}$ in contrast
to all other quantities involved in the problem, which decay or grow
exponentially with $N$, is crucial to the proof. 
\begin{thm}
\label{thm:2ndmomUBBK-1}For any nice $B\subset\mathbb{R}$ and nice
$I_{R}\subset\left(-1,1\right)$,

\begin{equation}
\limsup_{N\to\infty}\frac{1}{N}\log\left(\mathbb{E}\left\{ \left[{\rm Crt}_{N}\left(B,I_{R}\right)\right]_{2}\right\} \right)\leq\sup_{r\in I_{R}}\sup_{u_{i}\in B}\Psi_{p}\left(r,u_{1},u_{2}\right).\label{eq:2ndUB_BK-1-1}
\end{equation}

\end{thm}
Note that the terms involving $\Omega$ in the definition of $\Psi_{p}\left(r,u_{1},u_{2}\right)$
can be identified as the contribution from $\frac{1}{N}\log$ of the
absolute value of the determinants, whose asymptotic behavior is expressed
in terms of the semicircle law, and that the quadratic form in $u_{1}$
and $u_{2}$ corresponds to the joint Gaussian density of $U_{1}\left(r\right)$
and $U_{2}\left(r\right)$. In order to prove Theorem \ref{thm:Var-E2-log}
we need to identify the points at which the supremum above is attained.
The following lemma, proved in Section \ref{sec:maximalityPsi}, gives
sufficient conditions allowing to restrict attention to points satisfying
$u_{1}=u_{2}$.
\begin{lem}
\label{lem:u1=00003Du2-1}Defining $\Psi_{p}\left(r,u\right)\triangleq\Psi_{p}\left(r,u,u\right)$
we have the following. 
\begin{enumerate}
\item \label{enu:part 1}For nice $B\subset\left(-\infty,-E_{\infty}\left(p\right)\right)$,
for any $r\in\left(-1,1\right)$,
\[
\sup_{u_{i}\in B}\Psi_{p}\left(r,u_{1},u_{2}\right)=\sup_{u\in B}\Psi_{p}\left(r,u\right).
\]

\item \label{enu:part 2}For nice $B$ that intersect $\left(-E_{0}\left(p\right),E_{0}\left(p\right)\right)$,
\[
\limsup_{N\to\infty}\frac{1}{N}\log\left(\mathbb{E}\left\{ \left[{\rm Crt}_{N}\left(B,\left(-1,1\right)\right)\right]_{2}\right\} \right)\leq\sup_{r\in\left(-1,1\right)}\sup_{u\in B}\Psi_{p}\left(r,u\right).
\]

\end{enumerate}
\end{lem}
We complement the above with the following lemma, also proved in Section
\ref{sec:maximalityPsi}, which states for which $r$ the maximum
is attained (in one point of the proof we use computer for the numeric
evaluation of certain expressions, see the paragraph following (\ref{eq:yy})). 
\begin{lem}
\label{lem:r=00003D0}Setting $u_{th}\left(p\right)\triangleq\sqrt{2\frac{p-1}{p-2}\log\left(p-1\right)}>E_{0}\left(p\right)$,
for fixed $u$, $\Psi_{p}^{u}\left(r\right)\triangleq\Psi_{p}\left(r,u,u\right)$
can be extended to a continuous function $\bar{\Psi}_{p}^{u}\left(r\right)$
on $\left[-1,1\right]$, such that:
\begin{enumerate}
\item \label{enu:p1}If $\left|u\right|<u_{th}\left(p\right)$, then $\bar{\Psi}_{p}^{u}\left(r\right)$
attains its maximum on $\left[-1,1\right]$, uniquely, at $r=0$.
\item \label{enu:p2}If $\left|u\right|>u_{th}\left(p\right)$, then $\bar{\Psi}_{p}^{u}\left(r\right)$
is maximal on $\left[-1,1\right]$ at any $r\in\left\{ 1,\left(-1\right)^{p+1}\right\} $
and only there.
\item \label{enu:p3}If $\left|u\right|=u_{th}\left(p\right)$, then $\bar{\Psi}_{p}^{u}\left(r\right)$
is maximal on $\left[-1,1\right]$ at any $r\in\left\{ 0,1,\left(-1\right)^{p+1}\right\} $
and only there.
\end{enumerate}
\end{lem}
Combining Theorem \ref{thm:2ndmomUBBK-1} and Lemmas \ref{lem:u1=00003Du2-1}
and \ref{lem:r=00003D0} (and using Theorem \ref{thm:A-BA-C}, which
provides a lower bound for $\left[\mbox{Crt}_{N}\left(B,\left(-1,1\right)\right)\right]_{2}$),
we prove Theorem \ref{thm:Var-E2-log} as well as the following corollary
in Section \ref{sec:Proofs}.
\begin{cor}
\label{cor:lowoverlap}For any $u\in\left(-E_{0}\left(p\right),-E_{\infty}\left(p\right)\right)$
and $\epsilon>0$, 
\begin{align*}
\lim_{N\to\infty}\frac{1}{N}\log\left(\mathbb{E}\left\{ \left({\rm Crt}_{N}\left(\left(-\infty,u\right)\right)\right)^{2}\right\} \right) & =\lim_{N\to\infty}\frac{1}{N}\log\left(\mathbb{E}\left\{ \left[{\rm Crt}_{N}\left(\left(-\infty,u\right),\left(-1,1\right)\right)\right]_{2}\right\} \right)\\
 & >\limsup_{N\to\infty}\frac{1}{N}\log\left(\mathbb{E}\left\{ \left[{\rm Crt}_{N}\left(\left(-\infty,u\right),\left(-1,1\right)\setminus\left(-\epsilon,\epsilon\right)\right)\right]_{2}\right\} \right).
\end{align*}

\end{cor}
We now move on to discuss the refinement of the asymptotics to $O(1)$
scale - i.e., the proof of Theorem \ref{thm:Var-E2}. Corollary \ref{cor:lowoverlap}
implies that the contribution of overlaps outside $\left(-\epsilon,\epsilon\right)$
to the second moment of ${\rm Crt}_{N}\left(\left(-\infty,u\right)\right)$
is negligible, assuming $u\in\left(-E_{0}\left(p\right),-E_{\infty}\left(p\right)\right)$.
By the fact that $\Theta_{p}\left(u\right)$ (see (\ref{eq:e1}))
is strictly increasing for $u<0$ and the equivalence of moments on
exponential scale (i.e., Theorem \ref{thm:Var-E2-log}), we also have
that the contribution of levels outside $\left(u-\epsilon,u\right)$
to either the first or second moment is negligible. Thus, relying
on the fact that the second moment is larger than the first squared,
in order to prove Theorem \ref{thm:Var-E2} it is enough to show that
(see Lemma \ref{lem:21})
\begin{equation}
\lim_{N\to\infty}\frac{\mathbb{E}\left[\mbox{Crt}_{N}\left(\left(u-\epsilon_{N},u\right),\left(-\rho_{N},\rho_{N}\right)\right)\right]_{2}}{\left(\mathbb{E}\left\{ {\rm Crt}_{N}\left(\left(u-\epsilon_{N},u\right)\right)\right\} \right)^{2}}\leq1,\label{eq:x1-1}
\end{equation}
for any sequences $\epsilon_{N},\,\rho_{N}\to0$. Using the formula
(\ref{eq:2ndmom_refinedform}) and the corresponding formula for the
first moment derived by \cite{A-BA-C}, one finds that proving (\ref{eq:x1-1})
boils down to showing that uniformly in $u_{i}\in\left(u-\epsilon_{N},u\right)$
and $r\in(-\rho_{N},\rho_{N})$, as $N\to\infty$,
\begin{equation}
\frac{\mathbb{E}\left\{ \prod_{i=1}^{2}\left|\det\left(\mathbf{M}_{N-1}^{\left(i\right)}\left(r,\sqrt{N}u_{1},\sqrt{N}u_{2}\right)\right)\right|\right\} }{\prod_{i=1}^{2}\mathbb{E}\left\{ \det\left(\mathbf{X}_{N-1}-\sqrt{\frac{N}{N-1}\frac{p}{p-1}}u_{i}I\right)\right\} }\leq1+o(1),\label{eq:x3}
\end{equation}
where $\mathbf{X}_{N-1}$ is a GOE matrix.

Recall the equality in distribution (\ref{eq:Ms}). As we shall see
(in Lemma \ref{lem:n7}), the perturbations $\mathbf{E}_{N-1}^{\left(i\right)}(r)$
are negligible when computing the expectation above, even on $O(1)$
scale. That is, it is sufficient to prove (\ref{eq:x3}) with its
numerator replaced by 
\begin{equation}
\mathbb{E}\left\{ \prod_{i=1}^{2}\left|\det\left(\mathbf{X}_{N-1}^{\left(i\right)}(r)-\sqrt{\frac{N}{N-1}\frac{p}{p-1}}u_{i}I\right)\right|\right\} ,\label{eq:MX}
\end{equation}
where $\mathbf{X}_{N-1}^{\left(i\right)}(r)$ are the correlated GOE
matrices in (\ref{eq:Ms}). Note that in the setting of Theorem \ref{thm:Var-E2}
we assume that $u$ is strictly less than $-E_{\infty}\left(p\right)$.
This exactly means that the shifts $-\sqrt{\frac{N}{N-1}\frac{p}{p-1}}u_{i}$
are larger than $2$ and therefore the eigenvalues of the shifted
GOE matrices in (\ref{eq:MX}) are bounded away from $0$ with high
probability. This will allow us to apply concentration inequalities
of linear statistics of the eigenvalues to $\frac{1}{N}\log$ of the
product in (\ref{eq:MX}) (truncated) and its derivative in $u_{i}$.
Using the latter we will relate (\ref{eq:MX}) to 
\[
w_{u}(r)=\mathbb{E}\left\{ \prod_{i=1}^{2}\det\left(\mathbf{X}_{N-1}^{\left(i\right)}(r)-\sqrt{\frac{N}{N-1}\frac{p}{p-1}}uI\right)\right\} .
\]
We note that with $r=0$, $\mathbf{X}_{N-1}^{\left(1\right)}(0)$
and $\mathbf{X}_{N-1}^{\left(2\right)}(0)$ are i.i.d, so that $w_{u}(0)$
coincides with the denominator of (\ref{eq:x3}) with $u_{i}=u$.
Combining the above, at this point what we will need to show in order
to conclude (\ref{eq:x3}) is that $w_{u}(r)=(1+o(1))w_{u}(0)$ as
$N\to\infty$, uniformly in $r\in(-\rho_{N},\rho_{N})$. The key to
proving this will be to show that $w_{u}(r)$ is convex in a power
of $r$ and bound the ratio $|w_{u}(1)/w_{u}(0)|$ by a constant independent
of $N$ (see Lemma \ref{lem:n4}).%
\footnote{To be precise, $w_{u}(r)$ is convex in a power of $r$ only on $[0,1]$,
and for negative $r$ we will use a certain relation between $w_{u}(r)$
and $w_{u}(-r)$.%
}

We finish with two remarks about generalizations. First, we note that
parts of the current work generalize to the case of general mixed
models. Specifically, by the same method, and a somewhat more tedious
algebra, one can obtain an equivalent of Theorem \ref{thm:2ndmomUBBK-1}.
In the general case however, the function that replaces $\Psi_{p}$
is more complicated (mainly due to changes in the conditional law
of the Hessians of the Hamiltonian) and its analysis, albeit just
`a matter of calculus', seems to be substantially more difficult.
(Moreover, from the remark made in the introduction, we know that
the second moment cannot match the first squared for full mixture
models, which implies that for certain mixed models the function $\Psi_{p}$
achieve its maximum in the interior of the interval $[0,1]$. We do
not have a characterization of the mixtures that allow one to carry
out the analysis we performed in the pure $p$-spin case.) 

In another direction, the authors of \cite{A-BA-C,ABA2} treat the
case of critical points of any given index. To complete the analysis
of the corresponding second moment, note that the effect of introducing
a restriction on the index in (\ref{eq:2ndmom_refinedform}) is simply
adding there the indicator of the corresponding event. By a similar
method to that used in the proof of Theorem \ref{thm:2ndmomUBBK-1},
this would result in an addition to $\Psi_{p}\left(r,u_{1},u_{2}\right)$
of the term
\[
\limsup_{N\to\infty}\frac{1}{N}\log\left(\mathbb{P}\left\{ \left(\mathbf{M}_{N-1}^{\left(i\right)}\left(r,\sqrt{N}u_{1},\sqrt{N}u_{2}\right)\right)_{i=1,2}\mbox{\,\ are of index}\, k\,\right\} \right),
\]
and would require both analyzing the probability above and the modified
function $\Psi_{p}\left(r,u_{1},u_{2}\right)$ in order to obtain
an upper bound on the logarithmic asymptotics of the second moment
of the number of critical points of index $k$. We have not attempted
to complete this computation. We remark, however, that for the study
of the Gibbs measure at low enough temperature it is sufficient to
understand the critical points with no restriction on the index; see
\cite{geometryGibbs}. In fact, only the critical points close to
$-NE_{0}(p)$ play a role in \cite{geometryGibbs} and those are typically
local minima (e.g., as follows from bounds on critical points of positive
index proved in \cite{A-BA-C}).

Lastly, we state two results of \cite{A-BA-C} that will be needed
later.

\subsection*{An integral formula and the logarithmic asymptotics of the first
moment}

We shall need the following two results borrowed from \cite{A-BA-C}. 
\begin{lem}
\cite[Lemmas 3.1, 3.2]{A-BA-C}\label{lem:1stmom} For all $p\geq3$,
\begin{equation}
\mathbb{E}\left\{ {\rm Crt}_{N}\left(\left(-\infty,u\right)\right)\right\} =\omega_{N}\left(\frac{p-1}{2\pi}(N-1)\right)^{\frac{N-1}{2}}\mathbb{E}\left\{ \left|\det\left(\mathbf{M}_{N-1}-\sqrt{\frac{p}{p-1}\frac{1}{N-1}}UI\right)\right|\mathbf{1}\Big\{ U<\sqrt{N}u\Big\}\right\} ,\label{eq:1stmom}
\end{equation}
where $\mathbf{M}_{N-1}$ is a GOE matrix of dimension $N-1\times N-1$
independent of $U\sim N\left(0,1\right)$.\end{lem}
\begin{thm}
\cite[Theorem 2.8]{A-BA-C}\label{thm:A-BA-C} For all $p\geq3$,
\begin{equation}
\lim_{N\to\infty}\frac{1}{N}\log\left(\mathbb{E}\left\{ {\rm Crt}_{N}\left(\left(-\infty,u\right)\right)\right\} \right)=\Theta_{p}\left(u\right)=\begin{cases}
\frac{1}{2}+\frac{1}{2}\log\left(p-1\right)-\frac{u^{2}}{2}+\Omega\left(\sqrt{\frac{p}{p-1}}u\right) & \mbox{ if }u<0,\\
\frac{1}{2}\log\left(p-1\right) & \mbox{ if }u\geq0.
\end{cases}\label{eq:54}
\end{equation}

\end{thm}

\section{\label{sec:ExactFormula:KR+Hessians}proof of Lemma \ref{lem:2ndKR-refinedFrom}}

This section is devoted to the proof of Lemma \ref{lem:2ndKR-refinedFrom}.
Let $f_{N}\left(\boldsymbol{\sigma}\right)$ be equal to $H_{N}\left(\boldsymbol{\sigma}\right)$
reparametrized and normalized to be a Gaussian field on 
\[
\mathbb{S}=\mathbb{S}^{N-1}=\left\{ \boldsymbol{\sigma}\in\mathbb{R}^{N}:\,\left\Vert \boldsymbol{\sigma}\right\Vert _{2}=1\right\} 
\]
with constant variance $1$, 
\begin{equation}
f_{N}\left(\boldsymbol{\sigma}\right)=f_{N,p}\left(\boldsymbol{\sigma}\right)=\frac{1}{\sqrt{N}}H_{N,p}\left(\sqrt{N}\boldsymbol{\sigma}\right).\label{eq:f}
\end{equation}
The covariance of $f_{N}\left(\boldsymbol{\sigma}\right)$ is given
by
\[
\mathbb{E}\left\{ f_{N}\left(\boldsymbol{\sigma}\right),f_{N}\left(\boldsymbol{\sigma}'\right)\right\} =\left\langle \boldsymbol{\sigma},\boldsymbol{\sigma}'\right\rangle ^{p},
\]
where $\left\langle \boldsymbol{\sigma},\boldsymbol{\sigma}'\right\rangle =\sum_{i=1}^{N}\sigma_{i}\sigma_{i}^{\prime}$
is the usual inner product.

Note that
\begin{align}
\mbox{Crt}_{N}\left(B\right) & =\mbox{Crt}_{N}^{f}\left(B\right)\triangleq\#\left\{ \left.\boldsymbol{\sigma}\in\mathbb{S}^{N-1}\,\right|\,\nabla f_{N}\left(\boldsymbol{\sigma}\right)=0,\, f_{N}\left(\boldsymbol{\sigma}\right)\in\sqrt{N}B\right\} ,\nonumber \\
\left[\mbox{Crt}_{N}\left(B,I_{R}\right)\right]_{2} & =\left[\mbox{Crt}_{N}^{f}\left(B,I_{R}\right)\right]_{2}\triangleq\#\left\{ \left.\left(\boldsymbol{\sigma},\boldsymbol{\sigma}'\right)\in\left(\mathbb{S}^{N-1}\right)^{2}\,\right|\,\left\langle \boldsymbol{\sigma},\boldsymbol{\sigma}'\right\rangle \in I_{R},...\right.\label{eq:31}\\
 & \left.\nabla f_{N}\left(\boldsymbol{\sigma}\right)=\nabla f_{N}\left(\boldsymbol{\sigma}'\right)=0,\, f_{N}\left(\boldsymbol{\sigma}\right)\in\sqrt{N}B,\, f_{N}\left(\boldsymbol{\sigma}'\right)\in\sqrt{N}B\right\} .\nonumber 
\end{align}

Endow the sphere $\mathbb{S}^{N-1}$ with the standard Riemannian
structure, induced by the Euclidean Riemannian metric on $\mathbb{R}^{N}$.
Given a (piecewise) smooth orthonormal frame field $E=\left(E_{i}\right)_{i=1}^{N-1}$
on $\mathbb{S}^{N-1}$ we define 
\begin{equation}
\nabla f_{N}\left(\boldsymbol{\sigma}\right)=\left(E_{i}f_{N}\left(\boldsymbol{\sigma}\right)\right)_{i=1}^{N-1},\,\,\nabla^{2}f_{N}\left(\boldsymbol{\sigma}\right)=\left(E_{i}E_{j}f_{N}\left(\boldsymbol{\sigma}\right)\right)_{i,j=1}^{N-1}.\label{eq:grad_Hess}
\end{equation}

\begin{lem}
\label{lem:FirstKR}Let $E=\left(E_{i}\right)_{i=1}^{N-1}$ be an
arbitrary (piecewise) smooth orthonormal frame field on $\mathbb{S}^{N-1}$
and use the notation (\ref{eq:grad_Hess}). For any nice $B\subset\mathbb{R}$
and nice $I_{R}\subset\left(-1,1\right)$,
\begin{align}
\mathbb{E}\left\{ \left[{\rm Crt}_{N}\left(B,I_{R}\right)\right]_{2}\right\} = & \omega_{N}\omega_{N-1}\left(\left(N-1\right)p\left(p-1\right)\right)^{N-1}\int_{I_{R}}dr\cdot\left(1-r^{2}\right)^{\frac{N-3}{2}}\varphi_{\nabla f\left(\mathbf{n}\right),\nabla f\left(\boldsymbol{\sigma}\left(r\right)\right)}\left(0,0\right)\label{eq:35}\\
 & \times\mathbb{E}\Bigg\{\left|\det\left(\frac{\nabla^{2}f\left(\mathbf{n}\right)}{\sqrt{\left(N-1\right)p\left(p-1\right)}}\right)\right|\cdot\left|\det\left(\frac{\nabla^{2}f\left(\boldsymbol{\sigma}\left(r\right)\right)}{\sqrt{\left(N-1\right)p\left(p-1\right)}}\right)\right|\nonumber \\
 & \mathbf{1}\Big\{ f\left(\mathbf{n}\right),\, f\left(\boldsymbol{\sigma}\left(r\right)\right)\in\sqrt{N}B\Big\}\,\Bigg|\,\nabla f\left(\mathbf{n}\right)=\nabla f\left(\boldsymbol{\sigma}\left(r\right)\right)=0\Bigg\},\nonumber 
\end{align}
where $\varphi_{\nabla f\left(\boldsymbol{\sigma}\right),\nabla f\left(\boldsymbol{\sigma}'\right)}$
is the joint density of the gradients $\nabla f\left(\boldsymbol{\sigma}\right)$
and $\nabla f\left(\boldsymbol{\sigma}'\right)$, and where 
\begin{equation}
\boldsymbol{\sigma}\left(r\right)=\left(0,...,0,\sqrt{1-r^{2}},r\right).\label{eq:sig_r}
\end{equation}

\end{lem}
The proof of Lemma \ref{lem:FirstKR} is deferred to the end of the
section. Clearly, the left-hand side of (\ref{eq:35}) is independent
of the choice of the orthonormal frame $E$. Thus, as a corresponding
continuous Radon-Nikodym derivative, the integrand in the right-hand
side is also independent of $E$. Therefore, Lemma \ref{lem:2ndKR-refinedFrom}
follows from Lemma \ref{lem:FirstKR}, combined with Lemmas \ref{lem:dens}
and \ref{lem:Hess_struct_2} given below. Their computationally heavy
proof is given in Appendix II.
\begin{lem}
\label{lem:dens}(the density of the gradients and the conditional
law of $\left(f\left(\mathbf{n}\right),f\left(\boldsymbol{\sigma}\left(r\right)\right)\right)$)
For any $r\in\left(-1,1\right)$ there exists a choice of $E=\left(E_{i}\right)_{i=1}^{N-1}$
such that the following holds. The density of $\left(\nabla f\left(\mathbf{n}\right),\nabla f\left(\boldsymbol{\sigma}\left(r\right)\right)\right)$
at $\left(0,0\right)\in\mathbb{R}^{N-1}\times\mathbb{R}^{N-1}$ is
\begin{align}
 & \varphi_{\nabla f\left(\mathbf{n}\right),\nabla f\left(\boldsymbol{\sigma}\left(r\right)\right)}\left(0,0\right)\nonumber \\
 & =\left(2\pi p\right)^{-\left(N-1\right)}\left[1-r^{2p-2}\right]^{-\frac{N-2}{2}}\left[1-\left(pr^{p}-\left(p-1\right)r^{p-2}\right)^{2}\right]^{-\frac{1}{2}},\label{eq:grad_dens}
\end{align}
and conditional on $\left(\nabla f\left(\mathbf{n}\right),\nabla f\left(\boldsymbol{\sigma}\left(r\right)\right)\right)=\left(0,0\right)$,
the vector $\left(f\left(\mathbf{n}\right),f\left(\boldsymbol{\sigma}\left(r\right)\right)\right)$
is a centered Gaussian vector with covariance matrix $\Sigma_{U}\left(r\right)$
(cf. (\ref{eq:26})).
\end{lem}

\begin{lem}
\label{lem:Hess_struct_2}(the conditional law of the Hessians) For
any $r\in\left(-1,1\right)$, with the same choice of $E=\left(E_{i}\right)_{i=1}^{N-1}$
as in Lemma \ref{lem:dens}, the following holds. Conditional on $f\left(\mathbf{n}\right)=u_{1},\, f\left(\boldsymbol{\sigma}\left(r\right)\right)=u_{2}$,
$\nabla f\left(\mathbf{n}\right)=\nabla f\left(\boldsymbol{\sigma}\left(r\right)\right)=0$,
the random variable 
\[
\left(\frac{\nabla^{2}f\left(\mathbf{n}\right)}{\sqrt{\left(N-1\right)p\left(p-1\right)}},\,\frac{\nabla^{2}f\left(\boldsymbol{\sigma}\left(r\right)\right)}{\sqrt{\left(N-1\right)p\left(p-1\right)}}\right)
\]
 has the same law as 
\[
\left(\mathbf{M}_{N-1}^{\left(1\right)}\left(r,u_{1},u_{2}\right),\,\mathbf{M}_{N-1}^{\left(2\right)}\left(r,u_{1},u_{2}\right)\right),
\]
 where 
\begin{equation}
\mathbf{M}_{N-1}^{\left(i\right)}\left(r,u_{1},u_{2}\right)=\hat{\mathbf{M}}_{N-1}^{\left(i\right)}\left(r\right)-\sqrt{\frac{1}{N-1}\frac{p}{p-1}}u_{i}I+\frac{m_{i}\left(r,u_{1},u_{2}\right)}{\sqrt{\left(N-1\right)p\left(p-1\right)}}e_{N-1,N-1},\label{eq:100}
\end{equation}
$e_{N-1,N-1}$ is an $N-1\times N-1$ matrix whose $N-1,N-1$ entry
is equal to $1$ and all other entries are $0$, $m_{i}$ is given
in (\ref{eq:m_i}), and $\hat{\mathbf{M}}_{N-1}^{\left(1\right)}\left(r\right)$
and $\hat{\mathbf{M}}_{N-1}^{\left(2\right)}\left(r\right)$ are $N-1\times N-1$
Gaussian random matrices with block structure
\begin{align}
\hat{\mathbf{M}}_{N-1}^{\left(i\right)}\left(r\right) & =\left(\begin{array}{cc}
\hat{\mathbf{G}}_{N-2}^{\left(i\right)}\left(r\right) & Z^{\left(i\right)}\left(r\right)\\
\left(Z^{\left(i\right)}\left(r\right)\right)^{T} & Q^{\left(i\right)}\left(r\right)
\end{array}\right),\label{eq:ghat}
\end{align}
satisfying the following: 
\begin{enumerate}
\item The random elements $\left(\hat{\mathbf{G}}_{N-2}^{\left(1\right)}\left(r\right),\hat{\mathbf{G}}_{N-2}^{\left(2\right)}\left(r\right)\right)$,
$\left(Z^{\left(1\right)}\left(r\right),Z^{\left(2\right)}\left(r\right)\right)$,
and $\left(Q^{\left(1\right)}\left(r\right),Q^{\left(2\right)}\left(r\right)\right)$
are independent. 
\item The matrices $\hat{\mathbf{G}}^{\left(i\right)}\left(r\right)=\hat{\mathbf{G}}_{N-2}^{\left(i\right)}\left(r\right)$
are $N-2\times N-2$ random matrices such that $\sqrt{\frac{N-1}{N-2}}\hat{\mathbf{G}}^{\left(i\right)}\left(r\right)$
is a GOE matrix and, in distribution,
\[
\left(\begin{array}{c}
\vphantom{\left\{ \left\{ \left\{ \right\} ^{2}\right\} ^{2}\right\} ^{2}}\hat{\mathbf{G}}^{\left(1\right)}\left(r\right)\\
\vphantom{\left\{ \left\{ \left\{ \right\} ^{2}\right\} ^{2}\right\} ^{2}}\hat{\mathbf{G}}^{\left(2\right)}\left(r\right)
\end{array}\right)=\left(\begin{array}{c}
\vphantom{\left\{ \left\{ \left\{ \right\} ^{2}\right\} ^{2}\right\} ^{2}}\sqrt{1-\left|r\right|^{p-2}}\bar{\mathbf{G}}^{\left(1\right)}+\left({\rm sgn}\left(r\right)\right)^{p}\sqrt{\left|r\right|^{p-2}}\bar{\mathbf{G}}\\
\vphantom{\left\{ \left\{ \left\{ \right\} ^{2}\right\} ^{2}\right\} ^{2}}\sqrt{1-\left|r\right|^{p-2}}\bar{\mathbf{G}}^{\left(2\right)}+\sqrt{\left|r\right|^{p-2}}\bar{\mathbf{G}}
\end{array}\right),
\]
where $\bar{\mathbf{G}}=\bar{\mathbf{G}}_{N-2}$, $\bar{\mathbf{G}}^{\left(1\right)}=\bar{\mathbf{G}}_{N-2}^{\left(1\right)}$,
and $\bar{\mathbf{G}}^{\left(2\right)}=\bar{\mathcal{\mathbf{G}}}_{N-2}^{\left(2\right)}$
are independent and have the same law as $\hat{\mathbf{G}}^{\left(i\right)}\left(r\right)$. 
\item The column vectors $Z^{\left(i\right)}\left(r\right)=\left(Z_{j}^{\left(i\right)}\left(r\right)\right)_{j=1}^{N-2}$
are Gaussian such that for any $j\leq N-2$, $\left(Z_{j}^{\left(1\right)}\left(r\right),\, Z_{j}^{\left(2\right)}\left(r\right)\right)$
is independent of all the other elements of the two vectors and 
\[
\left(Z_{j}^{\left(1\right)}\left(r\right),\, Z_{j}^{\left(2\right)}\left(r\right)\right)\sim N\left(0,\,\left(\left(N-1\right)p\left(p-1\right)\right)^{-1}\cdot\Sigma_{Z}\left(r\right)\right),
\]
where $\Sigma_{Z}\left(r\right)$ is given in (\ref{eq:86}).
\item Lastly, $Q^{\left(i\right)}\left(r\right)$ are Gaussian random variables
with 
\[
\left(Q^{\left(1\right)}\left(r\right),\, Q^{\left(2\right)}\left(r\right)\right)\sim N\left(0,\,\left(\left(N-1\right)p\left(p-1\right)\right)^{-1}\cdot\Sigma_{Q}\left(r\right)\right),
\]
where $\Sigma_{Q}\left(r\right)$ is given in (\ref{eq:86}).
\end{enumerate}
\end{lem}

\subsection{Proof of Lemma \ref{lem:FirstKR}}

First note that from additivity it is enough to prove the lemma under
the assumption that $I_{R}$ is an open interval. By the monotone
convergence theorem we may also assume that the closure of $I_{R}$
is contained in $(-1,1)$. Defining
\begin{equation}
\mathcal{S}_{N}^{2}\left(I_{R}\right)\triangleq\left\{ \left.\left(\boldsymbol{\sigma},\boldsymbol{\sigma}'\right)\in\left(\mathbb{S}^{N-1}\right)^{2}\,\right|\,\left\langle \boldsymbol{\sigma},\boldsymbol{\sigma}'\right\rangle \in I_{R}\right\} ,\label{eq:S2N}
\end{equation}
we have 
\begin{equation}
\left[\mbox{Crt}_{N}\left(B,I_{R}\right)\right]_{2}=\#\left\{ \left(\boldsymbol{\sigma},\boldsymbol{\sigma}'\right)\in\mathcal{S}_{N}^{2}\left(I_{R}\right)\,\left|\,\nabla f_{N}\left(\boldsymbol{\sigma}\right)=\nabla f_{N}\left(\boldsymbol{\sigma}'\right)=0,\, f_{N}\left(\boldsymbol{\sigma}\right),f_{N}\left(\boldsymbol{\sigma}'\right)\in\sqrt{N}B\right.\right\} .\label{eq:34}
\end{equation}

Consider the ($\mathbb{R}^{2\left(N-1\right)}$-valued) Gaussian field
\begin{equation}
\left(\nabla f_{N}\left(\boldsymbol{\sigma}\right),\nabla f_{N}\left(\boldsymbol{\sigma}'\right)\right),\label{eq:33}
\end{equation}
defined on the ($2\left(N-1\right)$-dimensional) submanifold $\mathcal{S}_{N}^{2}\left(I_{R}\right)$
(with boundary).

We are interested in the mean number of points in $\mathcal{S}_{N}^{2}\left(I_{R}\right)$
for which the field (\ref{eq:33}) satisfies the condition in the
definition of (\ref{eq:34}). This fits the setting of the variant
of the K-R Theorem given in \cite[Theorem 12.1.1]{RFG}. The latter
requires several regularity conditions to hold, which we prove in
Appendix III. From \cite[Theorem 12.1.1]{RFG} and an argument along
the lines of \cite[Section 11.5]{RFG} we have that
\begin{align*}
 & \negthickspace\negthickspace\mathbb{E}\left\{ \left[\mbox{Crt}_{N}\left(B,I_{R}\right)\right]_{2}\right\} =\int_{\mathbb{S}^{N-1}}d\boldsymbol{\sigma}\int_{\left\{ \boldsymbol{\sigma}'\in\mathbb{S}^{N-1}:\,\left\langle \boldsymbol{\sigma},\boldsymbol{\sigma}'\right\rangle \in I_{R}\right\} }d\boldsymbol{\sigma}'\varphi_{\nabla f\left(\boldsymbol{\sigma}\right),\nabla f\left(\boldsymbol{\sigma}'\right)}\left(0,0\right)\\
 & \times\mathbb{E}\Bigg\{\left|\det\nabla^{2}f\left(\boldsymbol{\sigma}\right)\right|\left|\det\nabla^{2}f\left(\boldsymbol{\sigma}'\right)\right|\mathbf{1}\Big\{ f\left(\boldsymbol{\sigma}\right),\, f\left(\boldsymbol{\sigma}'\right)\in\sqrt{N}B\Big\}\,\Bigg|\,\nabla f\left(\boldsymbol{\sigma}\right)=\nabla f\left(\boldsymbol{\sigma}'\right)=0\Bigg\},
\end{align*}
where $d\boldsymbol{\sigma}$ denotes the usual surface area on $\mathbb{S}^{N-1}$. 

Denote the north pole $\mathbf{n}\triangleq\left(0,0,...,0,1\right)\in\mathbb{S}^{N-1}$.
By symmetry, the inner integral is independent of $\boldsymbol{\sigma}$.
Thus, above we can set $\boldsymbol{\sigma}=\mathbf{n}$, remove the
integration over $\boldsymbol{\sigma}$ and multiply by a factor of
$\omega_{N}$. Now, note that with $\boldsymbol{\sigma}=\mathbf{n}$,
the integrand depends on $\boldsymbol{\sigma}'$ only through the
overlap $\rho\left(\boldsymbol{\sigma}'\right)=\left\langle \mathbf{n},\boldsymbol{\sigma}'\right\rangle $.
Thus we can use the co-area formula with the function $\rho\left(\boldsymbol{\sigma}'\right)$
to express the second integral as a one-dimensional integral over
a parameter $r$ (the volume of the inverse-image $\rho^{-1}\left(r\right)$
and the inverse of the Jacobian are given by $\omega_{N-1}\left(1-r^{2}\right)^{\frac{N-2}{2}}$
and $\left(1-r^{2}\right)^{-\frac{1}{2}}$, respectively). Doing so
yields (\ref{eq:35}), and completes the proof. \qed

\section{\label{sec:Logarithmic-asymptotics}Proof of Theorem \ref{thm:2ndmomUBBK-1}}

This section is dedicated to the proof of Theorem \ref{thm:2ndmomUBBK-1}.
For this we shall need the three lemmas below, which are proved in
the following subsections. Throughout the section we use the following
notation. Let 
\begin{equation}
\left(U_{1}\left(r\right),U_{2}\left(r\right)\right)\sim N\left(0,\Sigma_{U}\left(r\right)\right)\label{eq:Ui}
\end{equation}
(cf. (\ref{eq:26})) be a Gaussian vector independent of all other
variables and set 
\begin{equation}
\bar{U}_{i}\left(r\right)=\sqrt{\frac{1}{N-1}\frac{p}{p-1}}U_{i}\left(r\right).\label{eq:ubar}
\end{equation}
Also, let $\mathbf{G}_{N-2}^{\left(i\right)}\left(r\right)$ be the
upper-left $N-2\times N-2$ submatrix of $\mathbf{M}_{N-1}^{\left(i\right)}\left(r\right):=\mathcal{\mathbf{M}}_{N-1}^{\left(i\right)}\left(r,U_{1}\left(r\right),U_{2}\left(r\right)\right)$
(cf. Lemma \ref{lem:Hess_struct_2}). With $\hat{\mathbf{G}}_{N-2}^{\left(i\right)}\left(r\right)$
as defined in (\ref{eq:ghat}) we have 
\begin{equation}
\mathbf{G}_{N-2}^{\left(i\right)}\left(r\right)\triangleq\hat{\mathbf{G}}_{N-2}^{\left(i\right)}\left(r\right)-\bar{U}_{i}\left(r\right)I.\label{eq:M}
\end{equation}
Set 
\begin{equation}
W_{i}\left(r\right)=W_{i,N}\left(r\right)\triangleq\left(2\sum_{j=1}^{N-2}\left(\mathbf{M}_{N-1}^{\left(i\right)}\left(r\right)\right)_{j,N-1}+\left(\mathbf{M}_{N-1}^{\left(i\right)}\left(r\right)\right)_{N-1,N-1}\right)^{1/2}.\label{eq:W}
\end{equation}

For any $\kappa>\epsilon>0$ define
\[
h_{\epsilon}\left(x\right)=\max\left\{ \epsilon,x\right\} ,
\]
and
\begin{equation}
h_{\epsilon}^{\kappa}\left(x\right)=\begin{cases}
\epsilon & \,\,\mbox{if }x<\epsilon,\\
x & \,\,\mbox{if }x\in\left[\epsilon,\kappa\right],\\
1 & \,\,\mbox{if }x>\kappa,
\end{cases}\,\,\,\mbox{and}\,\,\, h_{\kappa}^{\infty}\left(x\right)=\begin{cases}
1 & \mbox{if }x\leq\kappa,\\
x & \mbox{if }x>\kappa,
\end{cases}\label{eq:h_trunc}
\end{equation}
so that $h_{\epsilon}^{\kappa}\left(x\right)h_{\kappa}^{\infty}\left(x\right)=h_{\epsilon}\left(x\right)$.
Lastly, define 
\begin{align}
\log_{\epsilon}^{\kappa}\left(x\right) & =\log\left(h_{\epsilon}^{\kappa}\left(x\right)\right).\label{eq:log trunc}
\end{align}
For a real symmetric matrix $\mathbf{A}$ let $\lambda_{j}\left(\mathbf{A}\right)$
denote the eigenvalues of $\mathbf{A}$.

The following bounds the determinant of $\mathbf{M}_{N-1}^{\left(i\right)}\left(r\right)$
in terms of the eigenvalues of $\mathbf{G}_{N-2}^{\left(i\right)}\left(r\right)$,
up to a multiplicative error term depending only on the last column
and row of $\mathbf{M}_{N-1}^{\left(i\right)}\left(r\right)$.
\begin{lem}
\label{cor:r2pert}Under the notation of Lemma \ref{lem:Hess_struct_2},
for any $\epsilon>0$, $r\in\left(-1,1\right)$, almost surely, 
\begin{align*}
\left|\det\left(\mathcal{\mathbf{M}}_{N-1}^{\left(i\right)}\left(r,U_{1}\left(r\right),U_{2}\left(r\right)\right)\right)\right| & \leq\frac{W_{i}\left(r\right)\left(W_{i}\left(r\right)+\epsilon\right)}{\epsilon}\prod_{j=1}^{N-2}h_{\epsilon}\left(\left|\lambda_{j}\left(\mathbf{G}_{N-2}^{\left(i\right)}\left(r\right)\right)\right|\right).
\end{align*}

\end{lem}
We shall need the following bound on $W_{i}(r)$.
\begin{lem}
\label{lem:W bound}There exists a bounded function $v\left(r\right):\left(-1,1\right)\to\mathbb{R}$
for which 
\[
\lim_{\delta\searrow0}\frac{v\left(1-\delta\right)}{\delta}\mbox{\,\,\ and \,\,}\lim_{\delta\searrow0}\frac{v\left(\delta-1\right)}{\delta}
\]
exist and are finite, such that for any natural $m$, the non-negative
random variables $W_{i}\left(r\right)$ satisfy for large enough $N$
\[
\mathbb{E}\left\{ \left(W_{i}\left(r\right)\right)^{2m}\right\} \leq v^{m}\left(r\right).
\]

\end{lem}
The following bounds, which are uniform in $r$, are the last ingredient
we need for proving Theorem \ref{thm:2ndmomUBBK-1}.
\begin{lem}
\label{lem:bound-h}For any $q>0$ and nice set $B$ the following
hold.
\begin{enumerate}
\item \label{enu:p1 bound-h}For any $\epsilon>0$ and $\kappa>\max\left\{ \epsilon,1\right\} $
there exists a constant $c=c\left(\epsilon,\kappa\right)>0$, such
that for large enough $N$, uniformly in $r\in\left(-1,1\right)$,
\begin{align}
 & \mathbb{E}\left\{ \prod_{i=1,2}\prod_{j=1}^{N-2}\left(h_{\epsilon}^{\kappa}\left(\left|\lambda_{j}\left(\mathbf{G}_{N-2}^{\left(i\right)}\left(r\right)\right)\right|\right)\right)^{q}\cdot\mathbf{1}\left\{ U_{1}(r),\, U_{2}(r)\in\sqrt{N}B\right\} \right\} \leq\exp\left\{ -cN^{2}\right\} \label{eq:40}\\
 & \quad+\mathbb{E}\left\{ \exp\left\{ \sum_{i=1,2}qN\int\log_{\epsilon}^{\kappa}\left(\left|\lambda-\bar{U}_{i}\right|\right)d\mu^{*}+2q\epsilon N\right\} \cdot\mathbf{1}\left\{ U_{1}(r),\, U_{2}(r)\in\sqrt{N}B\right\} \right\} ,\nonumber 
\end{align}
where $\mu^{*}$ is the semicircle law, given in (\ref{eq:semicirc}).
\item \label{enu:p2 bound-h}For large enough $\kappa>0$, uniformly in
$r\in\left(-1,1\right)$,
\begin{equation}
\mathbb{E}\left\{ \prod_{i=1,2}\prod_{j=1}^{N-2}\left(h_{\kappa}^{\infty}\left(\left|\lambda_{j}\left(\mathbf{G}_{N-2}^{\left(i\right)}\left(r\right)\right)\right|\right)\right)^{q}\right\} \leq2.\label{eq:44}
\end{equation}

\end{enumerate}
\end{lem}

\subsection{Proof of Lemma \ref{cor:r2pert}}

Let $\dot{\mathbf{M}}_{N-1}^{\left(i\right)}\left(r\right)$ denote
the matrix obtained from $\mathbf{M}_{N-1}^{\left(i\right)}\left(r\right)$
by replacing all entries in the last row and column by $0$. The eigenvalues
of $\dot{\mathbf{M}}_{N-1}^{\left(i\right)}\left(r\right)$ are the
same as those of $\mathbf{G}_{N-2}^{\left(i\right)}\left(r\right)$,
with an extra eigenvalue equal to $0$. For a general symmetric matrix
$\mathbf{A}$, $\sum_{i,j}\mathbf{A}_{i,j}^{2}=\sum_{j}\lambda_{j}^{2}\left(\mathbf{A}\right)$.
Thus, 
\[
\sum_{j}\lambda_{j}^{2}\left(\mathbf{M}_{N-1}^{\left(i\right)}\left(r\right)-\dot{\mathbf{M}}_{N-1}^{\left(i\right)}\left(r\right)\right)=W_{i}^{2}(r).
\]
Hence, the absolute value of any eigenvalue of $\mathbf{M}_{N-1}^{\left(i\right)}\left(r\right)-\dot{\mathbf{M}}_{N-1}^{\left(i\right)}\left(r\right)$
is bounded by $W_{i}(r)$. Note that $\mathbf{M}_{N-1}^{\left(i\right)}\left(r\right)-\dot{\mathbf{M}}_{N-1}^{\left(i\right)}\left(r\right)$
has rank $2$ at most, ant therefore has at most $2$ non-zero eigenvalues.
By an application of Corollary \ref{cor:Miroslav} we have that, almost
surely,
\[
\left|\det\left(\mathbf{M}_{N-1}^{\left(i\right)}\left(r\right)\right)\right|\leq\frac{W_{i}(r)\left(W_{i}(r)+T_{i}(r)\right)}{T_{i}(r)}\prod_{j=1}^{N-2}\left|\lambda_{j}\left(\mathbf{G}_{N-2}^{\left(i\right)}\left(r\right)\right)\right|,
\]
where $T_{i}(r)$ is the minimal absolute value of an eigenvalue of
$\mathbf{G}_{N-2}^{\left(i\right)}\left(r\right)$. The lemma follows
from this.\qed

\subsection{Proof of Lemma \ref{lem:W bound}}

From symmetry it is enough to prove the lemma with $i=1$. From Lemma
\ref{lem:Hess_struct_2} it follows that the law of $\mathcal{\mathbf{M}}_{N-1}^{\left(1\right)}\left(r\right)$
is the same as the law of 
\begin{equation}
\frac{\nabla^{2}f\left(\mathbf{n}\right)}{\sqrt{\left(N-1\right)p\left(p-1\right)}}\label{eq:42}
\end{equation}
conditional on 
\begin{equation}
\nabla f\left(\mathbf{n}\right)=\nabla f\left(\boldsymbol{\sigma}\left(r\right)\right)=0\label{eq:conditional grad}
\end{equation}
 (where $\boldsymbol{\sigma}\left(r\right)$ is given in (\ref{eq:sig_r})).
We emphasize that here the conditioning is only on the gradient at
the two points and not on the values of the Hamiltonian. The covariance
structure of the Gaussian matrix $\nabla^{2}f\left(\mathbf{n}\right)$,
conditional on (\ref{eq:conditional grad}), is computed in Section
\ref{sub:Proof-of-Lemmas}. In particular, it is given by (\ref{eq:cov1}),
in which $\mbox{Cov}_{\nabla f}$ denotes the conditional covariance.
In particular, we have that $(W_{1}\left(r\right))^{2}$ is identical
in distribution to
\begin{align*}
 & 2\frac{\mbox{Cov}_{\nabla f}\left\{ E_{1}E_{N-1}f\left(\mathbf{n}\right),E_{1}E_{N-1}f\left(\mathbf{n}\right)\right\} }{\left(N-1\right)p\left(p-1\right)}\sum_{i=1}^{N-2}X_{i}^{2}\\
 & +\frac{\mbox{Cov}_{\nabla f}\left\{ E_{N-1}E_{N-1}f\left(\mathbf{n}\right),E_{N-1}E_{N-1}f\left(\mathbf{n}\right)\right\} }{\left(N-1\right)p\left(p-1\right)}X_{N-1}^{2},
\end{align*}
where the covariances are as in (\ref{eq:cov1}) and $X_{i}$ are
i.i.d standard Gaussian variables and where we used the fact that
the conditional variance of $E_{i}E_{N-1}f\left(\mathbf{n}\right)$
is identical for all $i\leq N-2$.

Setting
\begin{equation}
\bar{v}\left(r\right)=2\left(N-1\right)p\left(p-1\right)\cdot\max_{i\in\{1,N-1\}}\left\{ \mbox{Cov}_{\nabla f}\left\{ E_{i}E_{N-1}f\left(\mathbf{n}\right),E_{i}E_{N-1}f\left(\mathbf{n}\right)\right\} \right\} ,\label{eq:43}
\end{equation}
by straightforward algebra, using (\ref{eq:cov1}), we have that
\[
\lim_{\epsilon\searrow0}\frac{\bar{v}\left(1-\epsilon\right)}{\epsilon}\mbox{\,\,\ and \,\,}\lim_{\epsilon\searrow0}\frac{\bar{v}\left(\epsilon-1\right)}{\epsilon}
\]
exist and are finite, and that $\bar{v}\left(r\right)$ is a bounded
function on $(-1,1)$.

Since $W_{1}\left(r\right)$ is stochastically dominated by 
\[
\sqrt{\frac{\bar{v}\left(r\right)}{p\left(p-1\right)}\frac{1}{N-1}\sum_{i=1}^{N-1}X_{i}^{2}},
\]
we conclude that 
\[
\mathbb{E}\left\{ \left(W_{1}\left(r\right)\right)^{2m}\right\} \leq\left(\frac{\bar{v}\left(r\right)}{\left(N-1\right)p\left(p-1\right)}\right)^{m}\mathbb{E}\left\{ \left(\sum_{i=1}^{N-1}X_{i}^{2}\right)^{m}\right\} .
\]
Since $\sum_{i=1}^{N-1}X_{i}^{2}$ is a chi-squared variable of $N-1$
degrees of freedom (cf. \cite[p. 13]{chisq}), 
\[
\mathbb{E}\left\{ \left(\sum_{i=1}^{N-1}X_{i}^{2}\right)^{m}\right\} =\left(N-1\right)\left(N+1\right)\cdots\left(N-3+2m\right).
\]
The lemma follows from this.\qed

\subsection{Proof of Lemma \ref{lem:bound-h}}

Note that
\begin{align}
 & \negthickspace\negthickspace\mathbb{E}\left\{ \prod_{i=1,2}\prod_{j=1}^{N-2}\left(h_{\epsilon}^{\kappa}\left(\left|\lambda_{j}\left(\mathbf{G}_{N-2}^{\left(i\right)}\left(r\right)\right)\right|\right)\right)^{q}\cdot\mathbf{1}\left\{ U_{i}(r)\in\sqrt{N}B\right\} \right\} \label{eq:38}\\
 & =\mathbb{E}\left\{ \prod_{i=1,2}\exp\left\{ q\sum_{j=1}^{N-2}\log_{\epsilon}^{\kappa}\left(\left|\lambda_{j}\left(\hat{\mathbf{G}}_{N-2}^{\left(i\right)}\left(r\right)\right)-\bar{U}_{i}(r)\right|\right)\right\} \cdot\mathbf{1}\left\{ U_{i}(r)\in\sqrt{N}B\right\} \right\} \nonumber \\
 & =\mathbb{E}\left\{ \prod_{i=1,2}\exp\left\{ q\left(N-2\right)\int\log_{\epsilon}^{\kappa}\left(\left|\lambda-\bar{U}_{i}(r)\right|\right)dL_{N-2}^{\left(i\right)}\left(\lambda\right)\right\} \cdot\mathbf{1}\left\{ U_{i}(r)\in\sqrt{N}B\right\} \right\} ,\nonumber 
\end{align}
where $L_{r,N-2}^{\left(i\right)}$ is the empirical measure of eigenvalues
of $\hat{\mathbf{G}}_{N-2}^{\left(i\right)}\left(r\right)$ (cf. (\ref{eq:emp_meas})).

The function $\log_{\epsilon}^{\kappa}\left(\left|\,\cdot\,-x\right|\right)$
is bounded and Lipschitz continuous, with the same bound and Lipschitz
constant for all $x\in\mathbb{R}$. Thus, there exists $c_{\epsilon,\kappa}>0$
such that (cf. Appendix I) 
\begin{equation}
A_{\epsilon}\triangleq\cup_{i=1,2}\cup_{x\in\mathbb{R}}\left\{ \int\log_{\epsilon}^{\kappa}\left(\left|\lambda-x\right|\right)d\left(L_{r,N-2}^{\left(i\right)}-\mu^{*}\right)>\epsilon\right\} \subset\cup_{i=1,2}\left\{ d_{LU}\left(\mu^{*},L_{r,N-2}^{\left(i\right)}\right)>c_{\epsilon,\kappa}\right\} .\label{eq:36}
\end{equation}

Since $\log_{\epsilon}^{\kappa}$ is bounded from above by $\log\left(\kappa\right)$
and since on $A_{\epsilon}^{c}$,
\[
\int\log_{\epsilon}^{\kappa}\left(\left|\lambda-x\right|\right)dL_{r,N-2}^{\left(i\right)}\left(\lambda\right)\leq\int\log_{\epsilon}^{\kappa}\left(\left|\lambda-x\right|\right)d\mu^{*}\left(\lambda\right)+\epsilon,
\]
with 
\begin{align*}
S\left(r,\mu_{1},\mu_{2}\right) & \triangleq\exp\left\{ q\left(N-2\right)\sum_{i=1,2}\int\log_{\epsilon}^{\kappa}\left(\left|\lambda-\bar{U}_{i}(r)\right|\right)d\mu_{i}\right\} ,\\
F_{N}\left(r\right) & \triangleq\left\{ U_{1}\left(r\right),\, U_{2}\left(r\right)\in\sqrt{N}B\right\} ,
\end{align*}
we have 
\begin{align}
 & \negthickspace\negthickspace\mathbb{E}\left\{ S\left(r,L_{r,N-2}^{\left(1\right)},L_{r,N-2}^{\left(2\right)}\right)\mathbf{\mathbf{1}}_{F_{N}\left(r\right)}\right\} \nonumber \\
 & =\mathbb{E}\left\{ S\left(r,L_{r,N-2}^{\left(1\right)},L_{r,N-2}^{\left(2\right)}\right)\cdot\mathbf{\mathbf{1}}_{A_{\epsilon}^{c}}\mathbf{\mathbf{1}}_{F_{N}\left(r\right)}\right\} +\mathbb{E}\left\{ S\left(r,L_{r,N-2}^{\left(1\right)},L_{r,N-2}^{\left(2\right)}\right)\cdot\mathbf{\mathbf{1}}_{A_{\epsilon}}\mathbf{\mathbf{1}}_{F_{N}\left(r\right)}\right\} \nonumber \\
 & \leq\exp\left\{ 2q\epsilon N\right\} \cdot\mathbb{E}\left\{ S\left(r,\mu^{*},\mu^{*}\right)\mathbf{\mathbf{1}}_{F_{N}\left(r\right)}\right\} +\exp\left\{ 2q\log\left(\kappa\right)N\right\} \cdot\mathbb{P}\left\{ A_{\epsilon}\right\} .\label{eq:39}
\end{align}

From Theorem  \ref{lem:GOELD} and (\ref{eq:36}), setting 
\[
c_{\epsilon,\kappa}^{\prime}=\frac{1}{2}\inf_{\mu\in\left(B\left(\mu^{*},c_{\epsilon,\kappa}\epsilon\right)\right)^{c}}J\left(\mu\right)>0,
\]
(where positivity follows from the fact that $J$ is a good rate function
with unique minimizer), one obtains for large enough $N$, 
\begin{equation}
\mathbb{P}\left\{ A_{\epsilon}\right\} \leq2\exp\left\{ -c_{\epsilon,\kappa}^{\prime}N^{2}\right\} .\label{eq:37}
\end{equation}

Combining (\ref{eq:38}), (\ref{eq:39}), and (\ref{eq:37}), we obtain,
for large enough $N$,

\begin{align*}
 & \negthickspace\negthickspace\mathbb{E}\left\{ \prod_{i=1,2}\prod_{j=1}^{N-2}\left(h_{\epsilon}^{\kappa}\left(\left|\lambda_{j}\left(\mathbf{G}_{N-2}^{\left(i\right)}\left(r\right)\right)\right|\right)\right)^{q}\cdot\mathbf{1}\left\{ U_{1}\left(r\right),\, U_{2}\left(r\right)\in\sqrt{N}B\right\} \right\} \\
 & \leq\exp\left\{ 2q\epsilon N\right\} \mathbb{E}\left\{ \prod_{i=1,2}\exp\left\{ qN\int\log_{\epsilon}^{\kappa}\left(\left|\lambda-\bar{U}_{i}\left(r\right)\right|\right)d\mu^{*}\right\} \cdot\mathbf{1}\left\{ U_{i}\left(r\right)\in\sqrt{N}B\right\} \right\} \\
 & +2\exp\left\{ 2q\log\left(\kappa\right)N\right\} \exp\left\{ -c_{\epsilon,\kappa}^{\prime}N^{2}\right\} ,
\end{align*}
from which part (\ref{enu:p1 bound-h}) follows.

Define
\[
\Lambda\left(r\right)=\Lambda_{N}\left(r\right)\triangleq\max_{\substack{i=1,2\\
j\leq N-2
}
}\left|\lambda_{j}\left(\mathbf{G}_{N-2}^{\left(i\right)}\left(r\right)\right)\right|.
\]
From a union bound and (\ref{eq:M}), 
\begin{equation}
\mathbb{P}\left\{ \Lambda\left(r\right)>t\right\} \leq\sum_{i=1,2}\left(\mathbb{P}\left\{ \max_{\substack{j\leq N-2}
}\left|\lambda_{j}\left(\hat{\mathbf{G}}_{N-2}^{\left(i\right)}\left(r\right)\right)\right|>t/2\right\} +\mathbb{P}\left\{ \bar{U}_{i}\left(r\right)>t/2\right\} \right).\label{eq:41}
\end{equation}

It is easy to verify that the variance of $U_{i}\left(r\right)$ is
bounded by $1$, uniformly in $r\in\left(-1,1\right)$. Recall that
$\sqrt{\frac{N-1}{N-2}}\hat{\mathbf{G}}_{N-2}^{\left(i\right)}\left(r\right)$
is a GOE matrix. Thus, from (\ref{eq:41}) and Lemma \ref{lem:maxEigBd},
there exists a constant $\tilde{c}>0$ such that for large enough
$t$ and any $N$,
\[
\mathbb{P}\left\{ \Lambda\left(r\right)>t\right\} \leq\sqrt{\frac{\tilde{c}N}{2\pi}}e^{-\frac{1}{2}\tilde{c}t^{2}N}.
\]

Let $\Lambda_{0}\sim N\left(0,\left(\tilde{c}N\right)^{-1}\right)$.
For large enough $\kappa>0$ and any $N$,

\begin{align}
\mathbb{E}\left\{ \prod_{i=1,2}\prod_{j=1}^{N-2}\left(h_{\kappa}^{\infty}\left(\left|\lambda_{j}\left(\mathbf{G}_{N-2}^{\left(i\right)}\left(r\right)\right)\right|\right)\right)^{q}\right\}  & \leq\mathbb{P}\left\{ \Lambda\left(r\right)\leq\kappa\right\} +\mathbb{E}\left\{ \left(\Lambda\left(r\right)\right)^{2qN}\mathbf{1}\left\{ \Lambda\left(r\right)>\kappa\right\} \right\} \nonumber \\
 & \leq1+\mathbb{E}\left\{ \Lambda_{0}^{2qN}\mathbf{1}\left\{ \Lambda_{0}>\kappa\right\} \right\} .\label{eq:80}
\end{align}

From the Cauchy-Schwarz inequality,
\[
\mathbb{E}\left\{ \Lambda_{0}^{2qN}\mathbf{1}\left\{ \Lambda_{0}>\kappa\right\} \right\} \leq\left[\mathbb{E}\left\{ \Lambda_{0}^{4qN}\right\} \mathbb{P}\left\{ \Lambda_{0}>\kappa\right\} \right]^{1/2}\leq\exp\left\{ -N\left(\frac{\tilde{c}\kappa^{2}}{4}-c_{q}\right)\right\} ,
\]
for some $c_{q}$. Finally, taking $\kappa$ to be large enough, this
together with (\ref{eq:80}) yields (\ref{eq:44}).\qed

\subsection{Proof of Theorem \ref{thm:2ndmomUBBK-1}}

Let $\kappa>\epsilon>0$, let $2\leq m\in\mathbb{N}$ and set $q=q\left(m\right)=m/\left(m-1\right)$.
From Lemma \ref{cor:r2pert}, the fact that $h_{\epsilon}^{\kappa}\left(x\right)h_{\kappa}^{\infty}\left(x\right)=h_{\epsilon}\left(x\right)$,
and H{\"{o}}lder's inequality,
\begin{equation}
\mathbb{E}\left\{ \prod_{i=1,2}\left|\det\left(\mathcal{\mathbf{M}}_{N-1}^{\left(i\right)}\left(r\right)\right)\right|\cdot\mathbf{1}\Big\{ U_{i}\left(r\right)\in\sqrt{N}B\Big\}\right\} \leq\left(\mathcal{E}_{\epsilon,\kappa}^{\left(1\right)}\left(r\right)\right)^{\nicefrac{1}{q}}\left(\mathcal{E}_{\epsilon,\kappa}^{\left(2\right)}\left(r\right)\right)^{\nicefrac{1}{2m}}\left(\mathcal{E}_{\epsilon,\kappa}^{\left(3\right)}\left(r\right)\right)^{\nicefrac{1}{4m}},\label{eq:83}
\end{equation}
where
\begin{align}
\mathcal{E}_{\epsilon,\kappa}^{\left(1\right)}\left(r\right) & =\mathbb{E}\left\{ \prod_{i=1,2}\prod_{j=1}^{N-2}\left(h_{\epsilon}^{\kappa}\left(\left|\lambda_{j}\left(\mathbf{G}_{N-2}^{\left(i\right)}\left(r\right)\right)\right|\right)\right)^{q}\cdot\mathbf{1}\left\{ U_{i}\left(r\right)\in\sqrt{N}B\right\} \right\} ,\nonumber \\
\mathcal{E}_{\epsilon,\kappa}^{\left(2\right)}\left(r\right) & =\mathbb{E}\left\{ \prod_{i=1,2}\prod_{j=1}^{N-2}\left(h_{\kappa}^{\infty}\left(\left|\lambda_{j}\left(\mathbf{G}_{N-2}^{\left(i\right)}\left(r\right)\right)\right|\right)\right)^{2m}\right\} ,\label{eq:84}\\
\mathcal{E}_{\epsilon,\kappa}^{\left(3\right)}\left(r\right) & =\mathbb{E}\left\{ \left(\frac{W_{1}\left(r\right)\left(W_{1}\left(r\right)+\epsilon\right)}{\epsilon}\right)^{4m}\right\} \mathbb{E}\left\{ \left(\frac{W_{2}\left(r\right)\left(W_{2}\left(r\right)+\epsilon\right)}{\epsilon}\right)^{4m}\right\} .\nonumber 
\end{align}

Substituting this in (\ref{eq:2ndmom_refinedform}) and using H{\"{o}}lder's
inequality yields
\begin{align*}
\mathbb{E}\left\{ \left[\mbox{Crt}_{N}\left(B,I_{R}\right)\right]_{2}\right\}  & \leq C_{N}\left[\int_{I_{R}}\left(\mathcal{G}\left(r\right)\right)^{qN}\mathcal{E}_{\epsilon,\kappa}^{\left(1\right)}\left(r\right)dr\right]^{1/q}\left[\int_{I_{R}}\left(\mathcal{F}\left(r\right)\right)^{m}\mathcal{E}_{\epsilon,\kappa}^{\left(2\right)}\left(r\right)\left(\mathcal{E}_{\epsilon,\kappa}^{\left(3\right)}\left(r\right)\right)^{\nicefrac{1}{2}}dr\right]^{1/m},
\end{align*}
where $C_{N}$, $\mathcal{F}\left(r\right)$, and $\mathcal{G}\left(r\right)$
are given in (\ref{eq:cgf}). 

Therefore,
\begin{align}
\limsup_{N\to\infty}\frac{1}{N}\log\left(\mathbb{E}\left\{ \left[\mbox{Crt}_{N}\left(B,I_{R}\right)\right]_{2}\right\} \right) & \leq\limsup_{N\to\infty}\frac{1}{N}\log\left(C_{N}\right)\label{eq:45}\\
 & +\limsup_{N\to\infty}\frac{1}{qN}\log\left(\int_{I_{R}}\left(\mathcal{G}\left(r\right)\right)^{qN}\mathcal{E}_{\epsilon,\kappa}^{\left(1\right)}\left(r\right)dr\right)\nonumber \\
 & +\limsup_{N\to\infty}\frac{1}{mN}\log\left(\int_{I_{R}}\left(\mathcal{F}\left(r\right)\right)^{m}\mathcal{E}_{\epsilon,\kappa}^{\left(2\right)}\left(r\right)\left(\mathcal{E}_{\epsilon,\kappa}^{\left(3\right)}\left(r\right)\right)^{\nicefrac{1}{2}}dr\right).\nonumber 
\end{align}
The first summand is equal to 
\[
1+\log\left(p-1\right).
\]

One has that $\mathcal{F}\left(r\right)$ is bounded on any interval
$\left(-r_{0},r_{0}\right)$ with $0<r_{0}<1$, and that the limits
\[
\lim_{\delta\searrow0}\delta\mathcal{F}\left(1-\delta\right)\mbox{\,\,\ and \,\,}\lim_{\delta\searrow0}\delta\mathcal{F}\left(\delta-1\right)
\]
 exist and are finite. Using Lemma \ref{lem:W bound}, we therefore
have that
\[
\left(\mathcal{F}\left(r\right)\right)^{m}\left(\mathcal{E}_{\epsilon,\kappa}^{\left(3\right)}\left(r\right)\right)^{\nicefrac{1}{2}}
\]
is a bounded function of $r$ on $\left(-1,1\right)$. Thus, from
part (\ref{enu:p2 bound-h}) of Lemma \ref{lem:bound-h}, for $\kappa$
large enough , the third summand of (\ref{eq:45}) is equal to $0$.

Lastly, we need to analyze the second summand. To do so, we use part
(\ref{enu:p1 bound-h}) of Lemma \ref{lem:bound-h} and Varadhan's
integral lemma \cite[Theorem 4.3.1, Exercise 4.3.11]{LDbook}. Define
\begin{align*}
\Omega_{\epsilon}^{\kappa}(x) & \triangleq\int_{\mathbb{R}}\log_{\epsilon}^{\kappa}\left(\left|\lambda-x\right|\right)d\mu^{*}\left(\lambda\right),\\
\gamma_{p} & \triangleq\sqrt{\frac{p}{p-1}}.
\end{align*}
Note that, for $\left(\tilde{U}_{1},\tilde{U}_{2}\right)\sim N\left(0,I_{2\times2}\right)$,
\begin{align*}
\left(U_{1}\left(r\right),U_{2}\left(r\right)\right) & \overset{d}{=}\left(\tilde{U}_{1},\tilde{U}_{2}\right)\cdot\left(\Sigma_{U}\left(r\right)\right)^{1/2}.
\end{align*}
Let $e_{i}$, $i=1,2$, denote the standard basis of $\mathbb{R}^{2}$,
taken as $2\times1$ column vectors; so that $\left(t_{1},t_{2}\right)e_{i}=t_{i}$.
Lastly, define
\[
T\left(B\right)\triangleq\left\{ \left(r,\tilde{u}_{1},\tilde{u}_{2}\right):\, r\in\left(-r_{0},r_{0}\right),\,\left(\tilde{u}_{1},\tilde{u}_{2}\right)\cdot\left(\Sigma_{U}\left(r\right)\right)^{1/2}\in B\times B\right\} .
\]

Using part (\ref{enu:p1 bound-h}) of Lemma \ref{lem:bound-h}, we
obtain that, for large $N$, assuming $\kappa>1$, for some constant
$c>0$, 
\begin{align}
 & \negthickspace\negthickspace\int_{-r_{0}}^{r_{0}}\left(\mathcal{G}\left(r\right)\right)^{qN}\mathcal{E}_{\epsilon,\kappa}^{\left(1\right)}\left(r\right)dr-\exp\left\{ -cN^{2}\right\} \label{eq:48}\\
 & \leq e^{2q\epsilon N}\int_{-r_{0}}^{r_{0}}\left(\mathcal{G}\left(r\right)\right)^{qN}\mathbb{E}\left\{ \prod_{i=1,2}\exp\left\{ qN\Omega_{\epsilon}^{\kappa}\left(\bar{U}_{i}\left(r\right)\right)\right\} \cdot\mathbf{1}\left\{ \frac{U_{i}\left(r\right)}{\sqrt{N}}\in B\right\} \right\} dr\nonumber \\
 & =2r_{0}e^{2q\epsilon N}\mathbb{E}\left\{ \exp\left\{ qN\cdot\phi_{\epsilon}^{\kappa}\left(R,\frac{\tilde{U}_{1}}{\sqrt{N}},\frac{\tilde{U}_{2}}{\sqrt{N}}\right)\right\} \cdot\mathbf{1}\left\{ \left(R,\frac{\tilde{U}_{1}}{\sqrt{N}},\frac{\tilde{U}_{2}}{\sqrt{N}}\right)\in T\left(B\right)\right\} \right\} \triangleq2r_{0}e^{2q\epsilon N}\zeta_{\epsilon,\kappa,N},\nonumber 
\end{align}
where $R$ is independent of $\tilde{U}_{1},\,\tilde{U}_{1}$ and
is uniformly distributed in $\left(-r_{0},r_{0}\right)$, and where
\[
\phi_{\epsilon}^{\kappa}\left(r,\bar{u}_{1},\bar{u}_{2}\right)\triangleq\log\left(\mathcal{G}\left(r\right)\right)+\sum_{i=1,2}\Omega_{\epsilon}^{\kappa}\left(\gamma_{p}\left(\tilde{u}_{1},\tilde{u}_{2}\right)\cdot\left(\Sigma_{U}\left(R\right)\right)^{1/2}\cdot e_{i}\right).
\]

Note that $\phi_{\epsilon}^{\kappa}$ is a continuous function on
$\left(-1,1\right)\times\mathbb{R}\times\mathbb{R}$. Since $\mathcal{G}\left(r\right)\in\left(0,1\right)$
and $\Omega_{\epsilon}^{\kappa}$ is bounded from above by $\log\kappa$,
for any $q'>0$,
\[
\limsup_{N\to\infty}\frac{1}{N}\log\left(\mathbb{E}\left\{ \exp\left\{ q'N\cdot\phi_{\epsilon}^{\kappa}\left(R,\frac{\tilde{U}_{1}}{\sqrt{N}},\frac{\tilde{U}_{2}}{\sqrt{N}}\right)\right\} \right\} \right)\leq2q'\log\kappa.
\]

The random variable $\left(R,\frac{\tilde{U}_{1}}{\sqrt{N}},\frac{\tilde{U}_{2}}{\sqrt{N}}\right)$
satisfies the LDP with the good rate function
\[
J_{0}\left(r,\tilde{u}_{1},\tilde{u}_{2}\right)=\frac{\tilde{u}_{1}^{2}}{2}+\frac{\tilde{u}_{2}^{2}}{2}.
\]

Therefore, from Varadhan's integral lemma \cite[Theorem 4.3.1, Exercise 4.3.11]{LDbook}
combined with (\ref{eq:48}),
\begin{align*}
\limsup_{N\to\infty}\frac{1}{N}\log\left(\int_{-r_{0}}^{r_{0}}\left(\mathcal{G}\left(r\right)\right)^{qN}\mathcal{E}_{\epsilon,\kappa}^{\left(1\right)}\left(r\right)dr\right) & \leq\limsup_{N\to\infty}\frac{1}{N}\log\left(2r_{0}e^{2q\epsilon N}\zeta_{\epsilon,\kappa,N}\right)\\
 & \leq2q\epsilon+\sup_{\left(r,\tilde{u}_{1},\tilde{u}_{2}\right)\in T\left(B\right)}\left\{ q\phi_{\epsilon}^{\kappa}\left(r,\tilde{u}_{1},\tilde{u}_{2}\right)-\frac{\tilde{u}_{1}^{2}}{2}-\frac{\tilde{u}_{2}^{2}}{2}\right\} .
\end{align*}

Together with our analysis of the two other summands in (\ref{eq:45}),
this yields, for large enough $\kappa$,
\begin{equation}
\limsup_{N\to\infty}\frac{1}{N}\log\left(\mathbb{E}\left\{ \left[\mbox{Crt}_{N}\left(B\right)\right]_{2}^{r_{0}}\right\} \right)\leq1+\log\left(p-1\right)+2\epsilon+\frac{1}{q}\sup_{\left(r,\tilde{u}_{1},\tilde{u}_{2}\right)\in T\left(B\right)}\left\{ q\phi_{\epsilon}^{\kappa}\left(r,\tilde{u}_{1},\tilde{u}_{2}\right)-\frac{\tilde{u}_{1}^{2}}{2}-\frac{\tilde{u}_{2}^{2}}{2}\right\} .\label{eq:50}
\end{equation}
Letting $m\to\infty$, which implies that $q=q\left(m\right)\to1$,
we obtain (\ref{eq:50}) with $q=1$.

By a change of variables,
\begin{align*}
 & \negthickspace\negthickspace\negthickspace\negthickspace\sup_{\left(r,\tilde{u}_{1},\tilde{u}_{2}\right)\in T\left(B\right)}\left\{ \phi_{\epsilon}^{\kappa}\left(r,\tilde{u}_{1},\tilde{u}_{2}\right)-\frac{\tilde{u}_{1}^{2}}{2}-\frac{\tilde{u}_{2}^{2}}{2}\right\} \\
 & =\sup_{r\in\left(-r_{0},r_{0}\right)}\sup_{u_{1},u_{2}\in B}\left\{ \log\left(\mathcal{G}\left(r\right)\right)+\sum_{i=1,2}\Omega_{\epsilon}^{\kappa}\left(\gamma_{p}u_{i}\right)-\frac{1}{2}\left(u_{1},u_{1}\right)\left(\Sigma_{U}\left(r\right)\right)^{-1}\left(u_{1},u_{1}\right)^{T}\right\} .
\end{align*}

Letting $\kappa\to\infty$ and then $\epsilon\to0$ completes the
proof. \qed

\section{\label{sec:maximalityPsi}Proofs of Lemmas \ref{lem:u1=00003Du2-1}
and \ref{lem:r=00003D0}}

The bound of Theorem \ref{thm:2ndmomUBBK-1} is given in terms of
the supremum of $\Psi_{p}\left(r,u_{1},u_{2}\right)$ on the region
$I_{R}\times B\times B.$ In order to complete the proof of Theorem
\ref{thm:Var-E2-log}, we need to identify the points at which the
supremum is attained. This is the content of Lemmas \ref{lem:u1=00003Du2-1}
and \ref{lem:r=00003D0}, which we prove in this section. The following
simple remark is related to the proof of Lemma \ref{lem:u1=00003Du2-1},
and will also be used in the sequel.
\begin{rem}
\label{rem:r=00003D1}The bound of Theorem \ref{thm:2ndmomUBBK-1}
holds for any nice $I_{R}\subset\left(-1,1\right)$. We are particularly
interested in the case where $I_{R}=\left[-1,1\right]$,
\[
\left[\mbox{Crt}_{N}\left(B,\left[-1,1\right]\right)\right]_{2}=\left(\mbox{Crt}_{N}\left(B\right)\right)^{2}.
\]

The difference 
\[
\left[\mbox{Crt}_{N}\left(B,\left[-1,1\right]\right)\right]_{2}-\left[\mbox{Crt}_{N}\left(B,\left(-1,1\right)\right)\right]_{2}
\]
is simply the number of ordered pairs of points $\boldsymbol{\sigma}=\pm\boldsymbol{\sigma}'$
with $H_{N}\left(\boldsymbol{\sigma}\right),\, H_{N}\left(\boldsymbol{\sigma}'\right)\in NB$.
Thus, it is bounded from above by $2\mbox{Crt}_{N}\left(B\right)$.

Therefore, assuming $\lim_{N\to\infty}\mathbb{E}\mbox{Crt}_{N}\left(B\right)=\infty$,
\begin{equation}
\frac{\mathbb{E}\left\{ \left[\mbox{Crt}_{N}\left(B,\left(-1,1\right)\right)\right]_{2}\right\} }{\mathbb{E}\left\{ \left(\mbox{Crt}_{N}\left(B\right)\right)^{2}\right\} }\overset{N\to\infty}{\longrightarrow}1.\label{eq:n1}
\end{equation}

\end{rem}

\subsection{Proof of Lemma \ref{lem:u1=00003Du2-1}}

We begin with part (\ref{enu:part 1}). Fix $r\in\left(-1,1\right)$.
Note that $\log\left(x\right)$ is a concave function on $\left(0,\infty\right)$
and thus $\Omega\left(x\right)$ (defined in (\ref{eq:Omega})) is
concave on $\left(-\infty,-2\right)$. Since $\Sigma_{U}^{-1}\left(r\right)$
is positive definite for any $r\in\left(-1,1\right)$, we conclude
that, for $u_{1},u_{2}<-2\sqrt{\frac{p-1}{p}}=-E_{\infty}\left(p\right)$,
the function 
\begin{equation}
\left(u_{1},u_{2}\right)\mapsto-\frac{1}{2}\left(u_{1},u_{2}\right)\left(\Sigma_{U}\left(r\right)\right)^{-1}\left(\begin{array}{c}
u_{1}\\
u_{2}
\end{array}\right)+\Omega\left(\sqrt{\frac{p}{p-1}}u_{1}\right)+\Omega\left(\sqrt{\frac{p}{p-1}}u_{2}\right)\label{eq:60}
\end{equation}
 is concave. 

Let $u\in\mathbb{R}$ and define 
\begin{align*}
\Psi_{u}^{*}\left(v\right) & =\Psi_{p}\left(r,u+v,u-v\right)\\
 & =\tau_{p,r}-\frac{1}{2}\left(u+v,u-v\right)\left(\Sigma_{U}\left(r\right)\right)^{-1}\left(\begin{array}{c}
u+v\\
u-v
\end{array}\right)\\
 & +\Omega\left(\sqrt{\frac{p}{p-1}}\left(u+v\right)\right)+\Omega\left(\sqrt{\frac{p}{p-1}}\left(u-v\right)\right),
\end{align*}
where $\tau_{p,r}$ is a constant depending on $p$, $r$. 

If $u\in\left(-\infty,-E_{\infty}\left(p\right)\right)$, then for
\[
v\in\left(E_{\infty}\left(p\right)+u,-E_{\infty}\left(p\right)-u\right)\triangleq D\left(u\right),
\]
the function $\Psi_{u}^{*}\left(v\right)$ is concave in $v$ (as
a restriction of (\ref{eq:60}) to a line in $\mathbb{R}^{2}$, up
to adding the constant $\tau_{p,r}$). Moreover, by symmetry, 
\[
\frac{\partial}{\partial v}\Psi_{u}^{*}\left(0\right)=0,
\]
and therefore 
\[
\sup_{v\in D\left(u\right)}\Psi_{u}^{*}\left(v\right)=\Psi_{u}^{*}\left(0\right)=\Psi_{p}\left(r,u,u\right).
\]

Hence, for nice $B\subset\left(-\infty,-E_{\infty}\left(p\right)\right)$,
since 
\[
B\times B\subset\left\{ \left(u+v,u-v\right):\, u\in B,\, v\in D\left(u\right)\right\} ,
\]
we conclude that 
\[
\sup_{u_{i}\in B}\Psi_{p}\left(r,u_{1},u_{2}\right)\leq\sup_{u\in B}\sup_{v\in D\left(u\right)}\Psi_{u}^{*}\left(v\right)=\sup_{u\in B}\Psi_{p}\left(r,u,u\right).
\]
This completes the proof of part (\ref{enu:part 1}) of Lemma \ref{lem:u1=00003Du2-1}.

Now, assume that $B\subset\mathbb{R}$ is nice. Let $B_{1}$ and $B_{2}$
be nice disjoint sets whose union is $B$. Note that, since $x^{2}+y^{2}\geq2xy$,
for any $x,\, y\in\mathbb{R}$,
\begin{align*}
\left[\mbox{Crt}_{N}\left(B,\left(-1,1\right)\right)\right]_{2} & \leq\left(\mbox{Crt}_{N}\left(B_{1}\right)+\mbox{Crt}_{N}\left(B_{2}\right)\right)^{2}\\
 & \leq2\left(\left(\mbox{Crt}_{N}\left(B_{1}\right)\right)^{2}+\left(\mbox{Crt}_{N}\left(B_{2}\right)\right)^{2}\right).
\end{align*}

Note that (see Remark \ref{rem:r=00003D1})

\begin{align*}
\left(\mbox{Crt}_{N}\left(B_{i}\right)\right)^{2} & =\left[\mbox{Crt}_{N}\left(B_{i},\left[-1,1\right]\right)\right]_{2}\leq\left[\mbox{Crt}_{N}\left(B_{i},\left(-1,1\right)\right)\right]_{2}+2\mbox{Crt}_{N}\left(B_{i}\right).
\end{align*}
 Thus, by Theorem \ref{thm:2ndmomUBBK-1},
\begin{align}
 & \negthickspace\negthickspace\limsup_{N\to\infty}\frac{1}{N}\log\left(\mathbb{E}\left\{ \left[\mbox{Crt}_{N}\left(B,\left(-1,1\right)\right)\right]_{2}\right\} \right)\nonumber \\
 & \leq\max_{i=1,2}\left\{ \sup_{r\in\left(-1,1\right)}\sup_{u_{1},u_{2}\in B_{i}}\Psi_{p}\left(r,u_{1},u_{2}\right)\right\} \vee\limsup_{N\to\infty}\frac{1}{N}\log\left(\mathbb{E}\left\{ \mbox{Crt}_{N}\left(B\right)\right\} \right),\label{eq:61}
\end{align}
where $x\vee y=\max\left\{ x,y\right\} $, for any two numbers $x$,
$y$.

By applying the same argument iteratively, we obtain that if $B_{i}$,
$i=1,...,n$, is an $N$-independent partition of $B$ to nice sets,
then (\ref{eq:61}) holds with the maximum taken over all $i\leq n$. 

Let $\epsilon>0$ and choose a partition $B_{1},...,B_{n+1},B_{n+2}$
of $B$ such that $B_{1},...,B_{n}$ are intervals that form a partition
of $B'=B\cap\left[-E_{0}\left(p\right),E_{0}\left(p\right)\right]$
such that the diameter of $B_{i}$ is less then $\epsilon$ and such
that 
\begin{align*}
B_{n+1} & =B\cap(-\infty,-E_{0}\left(p\right)),\\
B_{n+2} & =B\cap(E_{0}\left(p\right),\infty).
\end{align*}
Then, 
\begin{align}
 & \negthickspace\negthickspace\limsup_{N\to\infty}\frac{1}{N}\log\left(\mathbb{E}\left\{ \left[\mbox{Crt}_{N}\left(B,\left(-1,1\right)\right)\right]_{2}\right\} \right)\nonumber \\
 & \leq\limsup_{N\to\infty}\frac{1}{N}\log\left(\mathbb{E}\left\{ \mbox{Crt}_{N}\left(B\right)\right\} \right)\vee\sup_{r\in\left(-1,1\right)}\sup_{\substack{u_{1},u_{2}\in B'\\
\left|u_{1}-u_{2}\right|<\epsilon
}
}\Psi_{p}\left(r,u_{1},u_{2}\right)\label{eq:62}\\
 & \vee\sup_{r\in\left(-1,1\right)}\sup_{u_{1},u_{2}\in B_{n+1}}\Psi_{p}\left(r,u_{1},u_{2}\right)\vee\sup_{r\in\left(-1,1\right)}\sup_{u_{1},u_{2}\in B_{n+2}}\Psi_{p}\left(r,u_{1},u_{2}\right).\nonumber 
\end{align}

Since $B_{n+1}\subset\left(-\infty,-E_{\infty}\left(p\right)\right)$,
by the first part of the lemma, 
\begin{equation}
\sup_{u_{1},u_{2}\in B_{n+1}}\Psi_{p}\left(r,u_{1},u_{2}\right)=\sup_{u\in B_{n+1}}\Psi_{p}\left(r,u,u\right).\label{eq:101}
\end{equation}
By symmetry of $\Psi_{p}\left(r,u_{1},u_{2}\right)$ in $\left(u_{1},u_{2}\right)$,
the same holds with $B_{n+2}$.

By concavity considerations similar to those used in the proof of
part (\ref{enu:part 1}), for any $u_{1},u_{2}\in\mathbb{R}$, setting
$u=\left(u_{1}+u_{2}\right)/2$,
\[
-\frac{1}{2}\left(u_{1},u_{2}\right)\left(\Sigma_{U}\left(r\right)\right)^{-1}\left(\begin{array}{c}
u_{1}\\
u_{2}
\end{array}\right)\leq-\frac{1}{2}\left(u,u\right)\left(\Sigma_{U}\left(r\right)\right)^{-1}\left(\begin{array}{c}
u\\
u
\end{array}\right).
\]
Therefore,
\[
\Psi_{p}\left(r,u_{1},u_{2}\right)\leq\Psi_{p}\left(r,u,u\right)+\left|2\Omega\left(\sqrt{\frac{p}{p-1}}u\right)-\Omega\left(\sqrt{\frac{p}{p-1}}u_{1}\right)-\Omega\left(\sqrt{\frac{p}{p-1}}u_{2}\right)\right|.
\]

The function $\Omega$ is uniformly continuous on $\left[-E_{0}\left(p\right),E_{0}\left(p\right)\right]$.
Therefore, for any $u_{1}$, $u_{2}$ such that $\left|u_{1}-u_{2}\right|<\epsilon$,
\[
\Psi_{p}\left(r,u_{1},u_{2}\right)\leq\Psi_{p}\left(r,u,u\right)+O\left(\epsilon\right),\mbox{\,\,\,\ as }\epsilon\to0.
\]
Therefore,
\[
\sup_{r\in\left(-1,1\right)}\sup_{\substack{u_{1},u_{2}\in B'\\
\left|u_{1}-u_{2}\right|<\epsilon
}
}\Psi_{p}\left(r,u_{1},u_{2}\right)\leq\sup_{r\in\left(-1,1\right)}\sup_{u\in B'}\Psi_{p}\left(r,u,u\right)+O\left(\epsilon\right).
\]

By letting $\epsilon\to0$, combining the above with (\ref{eq:101})
and the similar equality for $B_{n+2}$, we obtain from (\ref{eq:62}),
\begin{align}
 & \negthickspace\negthickspace\limsup_{N\to\infty}\frac{1}{N}\log\left(\mathbb{E}\left\{ \left[\mbox{Crt}_{N}\left(B,\left(-1,1\right)\right)\right]_{2}\right\} \right)\nonumber \\
 & \leq\limsup_{N\to\infty}\frac{1}{N}\log\left(\mathbb{E}\left\{ \mbox{Crt}_{N}\left(B\right)\right\} \right)\vee\sup_{r\in\left(-1,1\right)}\sup_{u\in B}\Psi_{p}\left(r,u,u\right).\label{eq:63}
\end{align}

Now, assume that $B$ intersects $\left(-E_{0}\left(p\right),E_{0}\left(p\right)\right)$.
Since it is nice, the intersection contains an open interval and by
Theorem \ref{thm:A-BA-C},
\[
\lim_{N\to\infty}\frac{1}{N}\log\left(\mathbb{E}\left\{ \mbox{Crt}_{N}\left(B\right)\right\} \right)>0.
\]

By Remark \ref{rem:r=00003D1}, it follows that
\begin{equation}
\limsup_{N\to\infty}\frac{1}{N}\log\left(\mathbb{E}\left\{ \left[\mbox{Crt}_{N}\left(B,\left(-1,1\right)\right)\right]_{2}\right\} \right)>\lim_{N\to\infty}\frac{1}{N}\log\left(\mathbb{E}\left\{ \left(\mbox{Crt}_{N}\left(B\right)\right)\right\} \right),\label{eq:66}
\end{equation}
meaning that (\ref{eq:63}) is equal to $\sup_{r\in\left(-1,1\right)}\sup_{u\in B}\Psi_{p}\left(r,u,u\right)$.
This completes the proof of part (\ref{enu:part 2}).\qed

\subsection{Proof of Lemma \ref{lem:r=00003D0} }

By straightforward algebra, 
\begin{equation}
\Psi_{p}^{u}\left(r\right)=\zeta_{p,u}+\frac{1}{2}\log\left(\frac{1-r^{2}}{1-r^{2p-2}}\right)-u^{2}\frac{1-r^{p}+(p-1)r^{p-2}(1-r^{2})}{1-r^{2p-2}+(p-1)r^{p-2}(1-r^{2})},\label{eq:65}
\end{equation}
where $\zeta_{p,u}$ depends only on $p$ and $u$. 

Note that 
\begin{equation}
\frac{1-r^{p}+(p-1)r^{p-2}(1-r^{2})}{1-r^{2p-2}+(p-1)r^{p-2}(1-r^{2})}=1-\frac{r^{p}-r^{2p-2}}{1-r^{2p-2}+(p-1)r^{p-2}(1-r^{2})}.\label{eq:20dec1}
\end{equation}
and
\begin{equation}
1-r^{2p-2}+(p-1)r^{p-2}(1-r^{2})=\left(1-r^{2}\right)\left(p-1\right)\left(\frac{1+r^{2}+\cdots+r^{2p-4}}{p-1}+r^{p-2}\right).\label{eq:11}
\end{equation}

For any $r\in\left(-1,1\right)$, 
\begin{equation}
\frac{1+r^{2}+\cdots+r^{2p-4}}{p-1}>\left|r^{p-2}\right|,\;\mbox{ and thus}\;\frac{1+r^{2}+\cdots+r^{2p-4}}{p-1}+r^{p-2}>0,\label{eq:bd1}
\end{equation}
since these are the arithmetic and geometric means of the same non-degenerate,
non-negative sequence. 

That is, the denominator in (\ref{eq:65}) above is positive for $r\in\left(-1,1\right)$.
Hence, in order to see that $\Psi_{p}^{u}\left(r\right)$ can be continuously
extended to $\left[-1,1\right]$ all that is need is to check that
the limits at $r=\pm1$ exist. This can be verified using L'H{\^{o}}pital's
rule.

Moreover, for odd $p$, (\ref{eq:20dec1}) is less then $1$ for $r\in\left(0,1\right)$
and is greater then $1$ for $r\in\left(-1,0\right)$. For even $p$,
of course, the expression is symmetric in $r$. Thus, the maximum
of $\bar{\Psi}_{p}^{u}\left(r\right)$ is achieved on $\left[0,1\right]$,
and if and only if $p$ is even, then the maximum can be attained
at some $r^{*}<0$. In that case it is also attained at $-r^{*}$.

Set, for $r\in\left[0,1\right)$,
\begin{equation}
Q_{p}^{u}\left(r\right)\triangleq\frac{1}{2}\log\left(\frac{1-r^{2}}{1-r^{2p-2}}\right)+u^{2}\frac{r^{p}-r^{2p-2}}{1-r^{2p-2}+(p-1)r^{p-2}(1-r^{2})},\label{eq:64}
\end{equation}
and 
\[
Q_{p}^{u}\left(1\right)\triangleq\lim_{r\nearrow1}Q_{p}^{u}\left(r\right)=\frac{1}{2}\log\left(\frac{1}{p-1}\right)+u^{2}\frac{p-2}{4(p-1)}.
\]

We conclude that in order to prove the lemma, it is enough to prove
it with $\Psi_{p}\left(r,u\right)$ replaced by $Q_{p}^{u}\left(r\right)$,
with $\left[-1,1\right]$ replaced by $\left[0,1\right]$, and with
the term $\left(-1\right)^{p+1}$ removed. 

Setting, for $r\in\left[0,1\right)$,
\begin{equation}
g_{0}\left(r\right)\triangleq\frac{r^{p}-r^{2p-2}}{1-r^{2p-2}+(p-1)r^{p-2}(1-r^{2})},\label{eq:74}
\end{equation}
and
\[
g_{0}\left(1\right)\triangleq\lim_{r\nearrow1}g_{0}\left(r\right)=\frac{p-2}{4(p-1)},
\]
we have, for $r\in\left(0,1\right)$,
\begin{equation}
\frac{d}{dr}g_{0}\left(r\right)=\frac{pr^{p-1}+\left[p\left(p-2\right)\right]r^{3p-3}-(p-1)(p-2)r^{3p-5}}{\left(1-r^{2p-2}+(p-1)r^{p-2}(1-r^{2})\right)^{2}}>0,\label{eq:x1}
\end{equation}
That is, $g_{0}\left(r\right)$ is strictly increasing in $r$.

We now show that if part (\ref{enu:p3}) of the lemma holds, the other
two follow. Assume that part (\ref{enu:p3}) holds. Let $u\in\mathbb{R}$
such that $\left|u\right|<u_{th}\left(p\right)$. For any $r\in\left(0,1\right]$,
$g_{0}\left(r\right)>0$ and
\[
Q_{p}^{u}\left(r\right)<Q_{p}^{u_{th}\left(p\right)}\left(r\right)\leq Q_{p}^{u_{th}\left(p\right)}\left(0\right)=Q_{p}^{u}\left(0\right).
\]

Similarly, let $u\in\mathbb{R}$ such that $\left|u\right|>u_{th}\left(p\right)$.
For any $r\in\left[0,1\right)$, 
\begin{align*}
Q_{p}^{u}\left(1\right) & =Q_{p}^{u_{th}\left(p\right)}\left(1\right)+\left(u^{2}-u_{th}^{2}\left(p\right)\right)g_{0}\left(1\right)\geq Q_{p}^{u_{th}\left(p\right)}\left(r\right)+\left(u^{2}-u_{th}^{2}\left(p\right)\right)g_{0}\left(1\right)\\
 & =Q_{p}^{u}\left(r\right)+\left(u^{2}-u_{th}^{2}\left(p\right)\right)\left(g_{0}\left(1\right)-g_{0}\left(r\right)\right)>Q_{p}^{u}\left(r\right).
\end{align*}

All that remains is to prove part (\ref{enu:p3}). First, we note
that
\begin{equation}
Q_{p}^{u_{th}\left(p\right)}\left(1\right)=\frac{1}{2}\log\left(\frac{1}{p-1}\right)+2\frac{p-1}{p-2}\log\left(p-1\right)\frac{p-2}{4(p-1)}=0=Q_{p}^{u_{th}\left(p\right)}\left(0\right).\label{eq:yy}
\end{equation}

We need to show that for any $r\in\left(0,1\right)$, $Q_{p}^{u_{th}\left(p\right)}\left(r\right)<0$.
First we assume that $p\leq10$. We have that $\frac{d}{dr}Q_{p}^{u_{th}\left(p\right)}\left(0\right)=0$
and $\frac{d}{dr}Q_{p}^{u_{th}\left(p\right)}\left(1\right),\,-\frac{d^{2}}{dr^{2}}Q_{p}^{u_{th}\left(p\right)}\left(0\right)>c_{0}$
for some $c_{0}>0$ ($c_{0}$ and $t_{0}$, $\epsilon_{0}$, to be
defined soon, can be computed explicitly). By a Taylor expansion combined
with bounds on higher order derivatives, for some $t_{0}>0$, for
any $r\in(0,t_{0})\cup(1-t_{0},1)$, $Q_{p}^{u_{th}\left(p\right)}\left(r\right)<0$.
By bounding the absolute value of the derivative $\frac{d}{dr}Q_{p}^{u_{th}\left(p\right)}\left(r\right)$
on the interval $(t_{0},1-t_{0})$, we have that for some $\epsilon_{0}>0$,
in order to prove that $Q_{p}^{u_{th}\left(p\right)}\left(r\right)<0$
for any $r\in(t_{0},1-t_{0})$ it is enough to verify the same only
for a finite mesh $t_{0}=r_{1}<\cdots<r_{k}=1-t_{0}$, with differences
$r_{i+1}-r_{i}$ that are bounded from above by $\epsilon_{0}$. We
verified the latter numerically using computer (see also Figure \ref{fig:Qp}). 

\begin{figure}[h]
\includegraphics[scale=0.2]{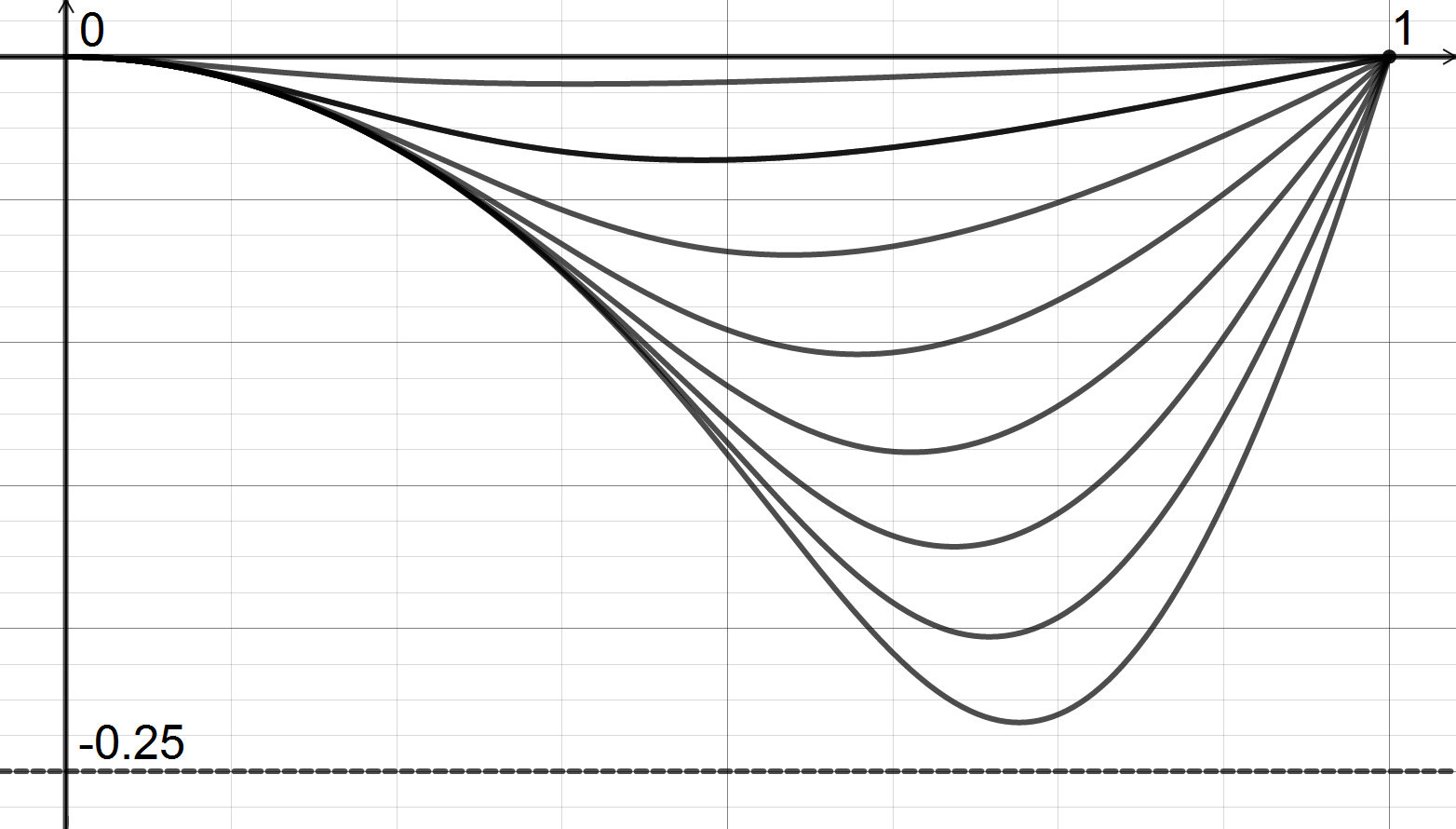}

\protect\caption{\label{fig:Qp}The functions $Q_{p}^{u_{th}\left(p\right)}\left(r\right)$
in the interval $\left[0,1\right]$, for $3\leq p\leq10$. For any
$r$, $Q_{p}^{u_{th}\left(p\right)}\left(r\right)$ decreases in $p$:
$Q_{p}^{u_{th}\left(p\right)}\left(r\right)\geq Q_{p+1}^{u_{th}\left(p+1\right)}\left(r\right)$.}
\end{figure}

We now assume that $p>10$. First, suppose also that $r\in\left(0,0.65\right]$.
By (\ref{eq:bd1}), 
\[
\frac{1-r^{2p-2}}{1-r^{2}}=1+\left(p-2\right)r^{2}\frac{1+r^{2}+\cdots+r^{2p-6}}{p-2}\geq1+(p-2)r^{p-1}.
\]

From the inequality $\log\left(1+x\right)\geq\frac{x}{1+x}$, valid
for $x>0$, we then have, for $r\in(0,0.65]$, $p\geq10$,
\[
\log\left(\frac{1-r^{2p-2}}{1-r^{2}}\right)\geq\frac{(p-2)r^{p-1}}{1+(p-2)r^{p-1}}\geq\frac{(p-2)r^{p-1}}{1+8\cdot0.65^{\left(-9\right)}},
\]
where the last inequality follows since $(p-2)\cdot0.65^{p-1}$ is
decreasing in $p$, for $p\geq10$. In addition, for $r\in(0,1)$,
\[
\frac{r^{p}-r^{2p-2}}{1-r^{2p-2}+\left(p-1\right)r^{p-2}\left(1-r^{2}\right)}\leq r^{p}\frac{1-r^{p-2}}{1-r^{2p-2}}\leq r^{p}.
\]

Thus, for $r\in(0,0.65]$, $p\geq10$,
\begin{align*}
Q_{p}^{u_{th}\left(p\right)}\left(r\right) & =\frac{1}{2}\log\left(\frac{1-r^{2}}{1-r^{2p-2}}\right)+\left(u_{th}\left(p\right)\right)^{2}\frac{r^{p}-r^{2p-2}}{1-r^{2p-2}+(p-1)r^{p-2}(1-r^{2})}\\
 & \leq-\frac{1}{2}\frac{(p-2)r^{p-1}}{1+8\cdot0.65^{\left(-9\right)}}+\left(u_{th}\left(p\right)\right)^{2}r^{p}\\
 & \leq r^{p-1}\left\{ 0.65\cdot u_{th}^{2}\left(p\right)-\frac{(p-2)}{2\left(1+8\cdot0.65^{\left(-9\right)}\right)}\right\} \triangleq\tau_{p}r^{p-1}\triangleq\bar{Q}_{p}(r).
\end{align*}

We have that $\tau_{10}<0$ and $\tau_{p}$ decreases in $p$, for
$p\geq10$. Hence, for $r\in(0,0.65]$, $p\geq10$, 
\[
Q_{p}^{u_{th}\left(p\right)}(r)<0=Q_{p}^{u_{th}\left(p\right)}(0).
\]

Now, assume that $r\in\left[0.65,1\right)$. From (\ref{eq:bd1})
and (\ref{eq:11}),
\begin{align*}
Q_{p}^{u_{th}\left(p\right)}(r) & \leq\frac{1}{2}\log\left(\frac{1-r^{2}}{1-r^{2p-2}}\right)+u_{th}^{2}\left(p\right)\frac{r^{p}-r^{2p-2}}{2(p-1)r^{p-2}(1-r^{2})}\\
 & =\frac{1}{2}\log\left(\frac{1-r^{2}}{1-r^{2p-2}}\right)+\frac{\log\left(p-1\right)}{p-2}\frac{1-r^{p-2}}{1-r^{2}}r^{2}\triangleq\widetilde{Q}_{p}\left(r\right).
\end{align*}

The derivative of $\widetilde{Q}_{p}(r)$ by $p$ is given, for $r\in(0,1)$,
by
\begin{align*}
\frac{d}{dp}\widetilde{Q}_{p}(r) & =\frac{r^{2p-2}\log r}{1-r^{2p-2}}+\frac{\frac{p-2}{p-1}-\log(p-1)}{p-2}\cdot\frac{1-r^{p-2}}{\left(1-r^{2}\right)}\cdot r^{2}+\frac{\log(p-1)}{p-2}\cdot\frac{-r^{p}\log r}{\left(1-r^{2}\right)}\\
 & \leq\frac{r^{2}}{\left(p-2\right)\left(1-r^{2}\right)}\left[\left(1-\log(p-1)\right)\left(1-r^{p-2}\right)-\log r\cdot\log(p-1)r^{p-2}\right].
\end{align*}

Therefore, for $r\in(0,1)$, $\frac{d}{dp}\widetilde{Q}_{p}(r)<0$
if
\[
\frac{1-\log(p-1)}{\log(p-1)}\left(1-r^{p-2}\right)-\log r<0.
\]
 Since for any $r\in[0.6,1)$ and any $p\geq10$, $\frac{1-\log(p-1)}{\log(p-1)}$
decreases in $p$, $\left(1-r^{p-2}\right)$ increases in $p$, and
\[
\frac{1-\log(10-1)}{\log(10-1)}\left(1-r^{10-2}\right)-\log r<0,
\]
it follows that $\frac{d}{dp}\widetilde{Q}_{p}(r)<0$, for any $r\in[0.6,1)$
and any $p\geq10$. Thus, if $\widetilde{Q}_{10}\left(r\right)<0$
for all $r\in[0.6,1)$, then the same holds for $Q_{p}^{u_{th}\left(p\right)}(r)$,
for any $p\geq10$. For $\widetilde{Q}_{10}\left(r\right)$ this was
verified numerically using a computer using a similar method to one
described above (see also Figure \ref{fig:Qtilde_10}).\qed

\begin{figure}[h]
\includegraphics[scale=0.2]{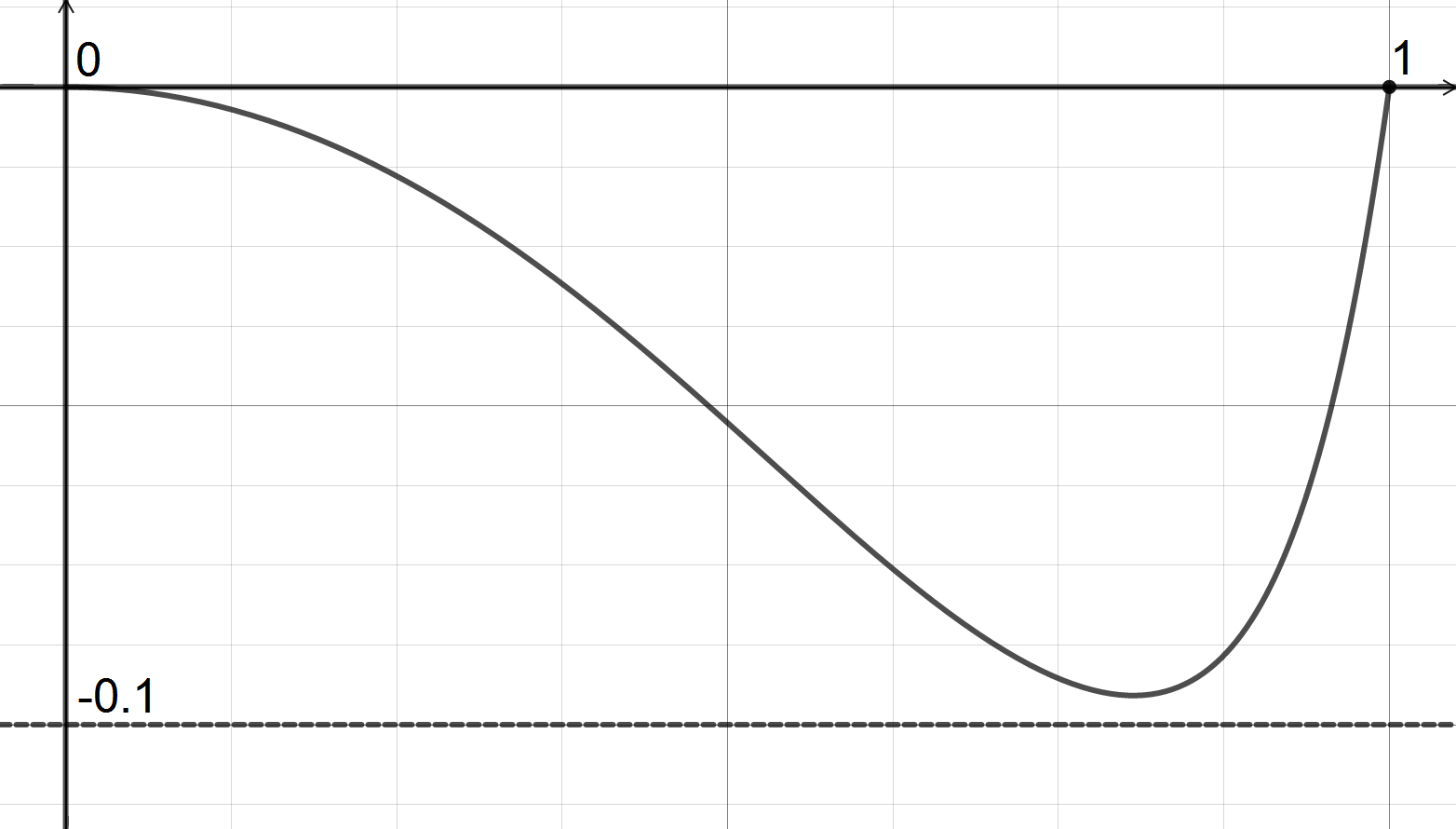}

\protect\caption{\label{fig:Qtilde_10}The function $\widetilde{Q}_{10}\left(r\right)$
in the interval $\left[0,1\right]$.}

\end{figure}

\section{\label{sec:Proofs}Proofs of Theorem \ref{thm:Var-E2-log} and Corollary
\ref{cor:lowoverlap}}

The content of this section is in its title. Our starting point is
the bound of Theorem \ref{thm:Var-E2-log} and the main tools we shall
use are Lemmas \ref{lem:u1=00003Du2-1} and \ref{lem:r=00003D0}.

\subsection{Proof of Theorem \ref{thm:Var-E2-log}}

By Theorem \ref{thm:A-BA-C}, denoting $u_{-}=u\wedge0=\min\left\{ u,0\right\} $,
\[
\frac{1}{N}\log\left(\mathbb{E}\left\{ \left(\mbox{Crt}_{N}\left(\left(-\infty,u\right)\right)\right)^{2}\right\} \right)\geq\frac{1}{N}\log\left(\left(\mathbb{E}\left\{ \mbox{Crt}_{N}\left(\left(-\infty,u\right)\right)\right\} \right)^{2}\right)\overset{N\to\infty}{\longrightarrow}2\Theta_{N}\left(u_{-}\right)=\Psi_{p}\left(0,u_{-},u_{-}\right).
\]
Combining this with (\ref{eq:n1}), it follows that what remains to
show in order to prove the theorem is that 
\begin{equation}
\limsup_{N\to\infty}\frac{1}{N}\log\mathbb{E}\left\{ \left[\mbox{Crt}_{N}\left(\left(-\infty,u\right),\left(-1,1\right)\right)\right]_{2}\right\} \leq\Psi_{p}\left(0,u_{-},u_{-}\right).\label{eq:ub1}
\end{equation}

Theorem \ref{thm:2ndmomUBBK-1}, part (\ref{enu:part 2}) of Lemma
\ref{lem:u1=00003Du2-1}, Lemma \ref{lem:r=00003D0}, and the fact
that $\bar{\Psi}_{p}^{v}\left(0\right)$ is symmetric in $v$, yield
\begin{align}
 & \negthickspace\negthickspace\negthickspace\negthickspace\limsup_{N\to\infty}\frac{1}{N}\log\mathbb{E}\left\{ \left[\mbox{Crt}_{N}\left(\left(-\infty,u\right),\left(-1,1\right)\right)\right]_{2}\right\} \nonumber \\
 & \leq\left(\sup_{v\in\left(-\infty,-u_{th}\left(p\right)\right)}\bar{\Psi}_{p}^{v}\left(1\right)\right)\vee\left(\sup_{v\in\left[-u_{th}\left(p\right),u_{-}\right]}\bar{\Psi}_{p}^{v}\left(0\right)\right).\label{eq:67}
\end{align}

We note that, for $v\leq0$, $\bar{\Psi}_{p}^{v}\left(0\right)=2\Theta_{p}\left(v\right)$
(cf. Theorem \ref{thm:A-BA-C}). Also, the monotonicity of the left-hand
side of (\ref{eq:54}) implies that $\Theta_{p}\left(v\right)$ is
non-decreasing for $v\leq0$. Since $u\in\left(-E_{0},\infty\right)$,
the supremum on the right-hand side of (\ref{eq:67}) is positive.
Hence, (\ref{eq:ub1}) holds if we are able to show that
\begin{equation}
\sup_{v\in\left(-\infty,-u_{th}\left(p\right)\right)}\bar{\Psi}_{p}^{v}\left(1\right)\leq0.\label{eq:68}
\end{equation}

By a straightforward calculation,
\begin{equation}
\frac{\partial}{\partial v}\bar{\Psi}_{p}^{v}\left(1\right)=-\frac{v\left(3p-2\right)}{2\left(p-1\right)}+2\sqrt{\frac{p}{p-1}}\Omega'\left(\sqrt{\frac{p}{p-1}}v\right).\label{eq:70}
\end{equation}
We note that for $x<-2$, 
\begin{equation}
\Omega'(x)=\int_{\left(-2,2\right)}\frac{d}{dx}\log\left(\lambda-x\right)d\mu^{*}\left(\lambda\right)\geq\inf_{\lambda\in\left(-2,2\right)}\frac{1}{x-\lambda}=\frac{1}{x+2}.\label{eq:69}
\end{equation}
From the above one can verify that $\frac{\partial}{\partial v}\bar{\Psi}_{p}^{v}\left(1\right)\geq0$
for $v\in\left(-\infty,-u_{th}\left(p\right)\right)$.

With $v=-u_{th}\left(p\right)<-E_{0}\left(p\right)$, by Lemma \ref{lem:r=00003D0},
\[
\bar{\Psi}_{p}^{v}\left(1\right)=\bar{\Psi}_{p}^{v}\left(0\right)=2\Theta_{p}\left(v\right)<0.
\]
This proves (\ref{eq:68}) and completes the proof. \qed

\subsection{Proof of Corollary \ref{cor:lowoverlap}}

The equality follows from Remark \ref{rem:r=00003D1} and the fact
that $u>-E_{0}\left(p\right)$.

Let $u\in\left(-E_{0}\left(p\right),-E_{\infty}\left(p\right)\right)$,
let $\epsilon>0$ and set $I_{\epsilon}=\left(-1,1\right)\setminus\left(-\epsilon,\epsilon\right)$.
For arbitrary $\tilde{u}\in(-u_{th}\left(p\right),u)$, Theorem \ref{thm:2ndmomUBBK-1},
Lemma \ref{lem:u1=00003Du2-1}, and Lemma \ref{lem:r=00003D0} yield
\begin{align}
 & \negthickspace\negthickspace\negthickspace\negthickspace\limsup_{N\to\infty}\frac{1}{N}\log\mathbb{E}\left\{ \left[\mbox{Crt}_{N}\left(\left(-\infty,u\right),I_{\epsilon}\right)\right]_{2}\right\} \nonumber \\
 & \leq\left(\sup_{r\in\left(-1,1\right)}\sup_{u_{1},u_{2}\in\left(-\infty,\tilde{u}\right)}\Psi_{p}\left(r,u_{1},u_{2}\right)\right)\vee\left(\sup_{r\in I_{\epsilon}}\sup_{u_{1},u_{2}\in\left[\tilde{u},u\right)}\Psi_{p}\left(r,u_{1},u_{2}\right)\right)\nonumber \\
 & \leq\left(\sup_{v\in\left(-\infty,\tilde{u}\right)}\bar{\Psi}_{p}^{v}\left(1\right)\right)\vee\left(\sup_{r\in I_{\epsilon}}\sup_{v\in\left[\tilde{u},u\right)}\Psi_{p}\left(r,v\right)\right).\label{eq:77}
\end{align}

We note that $\Psi_{p}\left(r,v\right)$ is continuous as a function
of $r$ at $\left(0,v\right)$. In the proof of Lemma \ref{lem:r=00003D0}
we saw that $\Psi_{p}\left(\left|r\right|,v\right)\geq\Psi_{p}\left(r,v\right)$,
thus
\[
\sup_{r\in I_{\epsilon}}\Psi_{p}\left(r,v\right)=\sup_{0<r\in I_{\epsilon}}\Psi_{p}\left(r,v\right).
\]
From (\ref{eq:65}), (\ref{eq:20dec1}), with $g_{0}\left(r\right)$
as defined in (\ref{eq:74}), 
\[
\Psi_{p}\left(0,v\right)-\Psi_{p}\left(r,v\right)=T_{r}-g_{0}(r)v^{2},
\]
where $T_{r}$ depends only on $r$. From this and  since $g_{0}\left(r\right)$
strictly increases in $r>0$ (see (\ref{eq:x1})) and $g_{0}\left(0\right)=0$,
we have that, uniformly in $v\in\left[\tilde{u},u\right)$,
\begin{align}
\Psi_{p}\left(0,v\right)-\sup_{0<r\in I_{\epsilon}}\Psi_{p}\left(r,v\right) & =\Psi_{p}\left(0,\tilde{u}\right)-\sup_{0<r\in I_{\epsilon}}\left(\Psi_{p}\left(r,\tilde{u}\right)-\left(\tilde{u}^{2}-v^{2}\right)g_{0}\left(r\right)\right)\nonumber \\
 & \geq\Psi_{p}\left(0,\tilde{u}\right)-\sup_{0<r\in I_{\epsilon}}\Psi_{p}\left(r,\tilde{u}\right)+\left(\tilde{u}^{2}-v^{2}\right)\inf_{r\in I_{R}}g_{0}\left(r\right)\label{eq:aa}\\
 & \geq\Psi_{p}\left(0,\tilde{u}\right)-\sup_{0<r\in I_{\epsilon}}\Psi_{p}\left(r,\tilde{u}\right)\triangleq c_{\epsilon}>0,\nonumber 
\end{align}
where the last inequality follows from Lemma \ref{lem:r=00003D0}.

Therefore, 
\begin{equation}
\sup_{r\in I_{\epsilon}}\sup_{v\in\left[\tilde{u},u\right)}\Psi_{p}\left(r,v\right)\leq\sup_{v\in\left[\tilde{u},u\right)}\Psi_{p}\left(0,v\right)-c_{\epsilon}<\Psi_{p}\left(0,u\right)=2\Theta_{p}\left(u\right).\label{eq:75}
\end{equation}

Recall that (\ref{eq:68}) holds. Thus, since $\Theta_{p}\left(u\right)>0$
and $\bar{\Psi}_{p}^{v}\left(1\right)$ is continuous in $v$, assuming
$\tilde{u}$ is close enough to $-u_{th}\left(p\right)$, 
\begin{equation}
\sup_{v\in\left(-\infty,\tilde{u}\right)}\bar{\Psi}_{p}^{v}\left(1\right)<2\Theta_{p}\left(u\right).\label{eq:76}
\end{equation}

Equations (\ref{eq:77}), (\ref{eq:75}), and (\ref{eq:76}) give
\begin{align*}
\limsup_{N\to\infty}\frac{1}{N}\log\mathbb{E}\left\{ \left[\mbox{Crt}_{N}\left(\left(-\infty,u\right),I_{\epsilon}\right)\right]_{2}\right\}  & <2\Theta_{p}\left(u\right)\\
 & =\lim_{N\to\infty}\frac{1}{N}\log\mathbb{E}\left\{ \left[\mbox{Crt}_{N}\left(\left(-\infty,u\right),\left(-1,1\right)\right)\right]_{2}\right\} ,
\end{align*}
where the equality follows from Theorems \ref{thm:Var-E2-log} and
\ref{thm:A-BA-C}.\qed

\section{\label{sec:Finer-Asymptotics}Proof of Theorem \ref{thm:Var-E2}}

The following notation will be used throughout the section. With $\mathbf{X}:=\mathbf{X}_{N-1}$
being a GOE matrix of dimension $N-1$, setting $\bar{u}:=\bar{u}_{N}=\sqrt{\frac{1}{N-1}\frac{p}{p-1}}u$,
we define for any $u<-E_{\infty}(p)$,
\begin{align}
\mathfrak{S}(u) & =\int\frac{1}{\sqrt{\frac{p-1}{p}}\lambda-u}d\mu^{*}(\lambda),\label{eq:S(u)}\\
\mathfrak{C}_{N}(u) & =\omega_{N}\left(\frac{p-1}{2\pi}(N-1)\right)^{\frac{N-1}{2}}\sqrt{\frac{N}{2\pi}}e^{-N\frac{u^{2}}{2}}\mathbb{E}\left\{ \det\left(\mathbf{X}-\sqrt{N}\bar{u}I\right)\right\} ,\label{eq:C(u)}
\end{align}
where $\mu^{*}$ denotes the semicircle law (\ref{eq:semicirc}).
We note that for $u<-E_{\infty}(p)$, with $w_{p}=\sqrt{\frac{p-1}{p}}$,
\begin{align}
\frac{d}{du}\Theta_{p}\left(u\right)=-\left(\mathfrak{S}(u)+u\right) & =-\int_{-2}^{2}\frac{1}{w_{p}\lambda-u}d\mu^{*}(\lambda)-u=\int_{0}^{2}\frac{2u(w_{p}\lambda)^{2}}{u^{2}-(w_{p}\lambda)^{2}}d\mu^{*}(\lambda)>0.\label{eq:<}
\end{align}

Below we use the standard big- and little-O notation to describe asymptotic
behavior as $N\to\infty$. Often, equations will contain several $o(a_{N}^{(i)})$
terms and will be said to hold uniformly in some variable (or more
than one), say $x\in B_{N}$. To avoid confusion, we remark that such
statements are to be understood as follows. The equation holds as
an equality with each of the $o(a_{N}^{(i)})$ terms replaced by a
function $h_{N}^{(i)}(x)$ satisfying $\sup_{x\in B_{N}}|h_{N}^{(i)}(x)|/|a_{N}^{(i)}|\to0$
as $N\to\infty$. 
\begin{lem}
\label{lem:n1}Let $u<-E_{\infty}\left(p\right)$ and suppose $J_{N}=\left(a_{N},b_{N}\right)$
is an interval such that $a_{N},\, b_{N}\to u$ as $N\to\infty$.
Then, as $N\to\infty$, 
\begin{equation}
\mathbb{E}\left\{ {\rm Crt}_{N}\left(J_{N}\right)\right\} =(1+o(1))\mathfrak{C}_{N}(b_{N})\int_{J_{N}}\exp\left\{ -N\left(u+\mathfrak{S}(u)\right)(v-b_{N})\right\} dv.\label{eq:12}
\end{equation}

\end{lem}
For brevity, we shall use the notation $\left[\mbox{Crt}_{N}\left(B\right)\right]_{2}^{\rho}\triangleq\left[\mbox{Crt}_{N}\left(B,\left(-\rho,\rho\right)\right)\right]_{2}$
in the sequel.
\begin{lem}
\label{lem:n3}Let $u<-E_{\infty}\left(p\right)$ and suppose $J_{N}=\left(a_{N},b_{N}\right)$
is an interval such that $a_{N},\, b_{N}\to u$ as $N\to\infty$.
Let $0<\rho_{N}$ be a sequence such that $\rho_{N}\to0$ as $N\to\infty$.
Then, as $N\to\infty$, 

\begin{equation}
\mathbb{E}\left\{ \left[{\rm Crt}_{N}\left(J_{N}\right)\right]_{2}^{\rho_{N}}\right\} \leq(1+o(1))\left(\mathfrak{C}_{N}(b_{N})\int_{J_{N}}\exp\left\{ -N\left(u+\mathfrak{S}(u)\right)(v-b_{N})\right\} dv\right)^{2}.\label{eq:12-1}
\end{equation}

\end{lem}

\begin{lem}
\label{lem:21}Let $u\in\left(-E_{0}\left(p\right),-E_{\infty}\left(p\right)\right)$,
$\rho\in\left(0,1\right)$ and $\epsilon>0$. Then
\[
\lim_{N\to\infty}\mathbb{E}\left\{ \left[{\rm Crt}_{N}\left(u-\epsilon,u\right)\right]_{2}^{\rho}\right\} /\mathbb{E}\left\{ \left({\rm Crt}_{N}\left(\left(-\infty,u\right)\right)\right)^{2}\right\} =1.
\]

\end{lem}
In Section \ref{sub:pfthm1} we prove Theorem \ref{thm:Var-E2} assuming
Lemmas \ref{lem:n1}, \ref{lem:n3} and \ref{lem:21}. The proof of
Lemma \ref{lem:21} only requires bounds on the exponential scale
we have already proved and will be given in Section \ref{sub:l3}.
Lemmas \ref{lem:n1} and \ref{lem:n3} will be proved in Sections
\ref{sub:l1} and \ref{sub:l2} after we prove several auxiliary results
in Section \ref{sub:Auxiliar}.

\subsection{\label{sub:pfthm1}Proof of Theorem \ref{thm:Var-E2} assuming Lemmas
\ref{lem:n1}, \ref{lem:n3} and \ref{lem:21}}

From Theorem \ref{thm:A-BA-C}, Lemma \ref{lem:21} and the fact that
$\Theta_{p}\left(u\right)$ is strictly increasing for $u<-E_{\infty}(p)$
(see (\ref{eq:<})), there exist positive sequences $\epsilon_{N}$,
$\rho_{N}$ such that as $N\to\infty$, $\epsilon_{N},\,\rho_{N}\to0$
and 
\[
\lim_{N\to\infty}\frac{\mathbb{E}\left\{ \left[{\rm Crt}_{N}\left(u-\epsilon_{N},u\right)\right]_{2}^{\rho_{N}}\right\} }{\mathbb{E}\left\{ \left({\rm Crt}_{N}\left(\left(-\infty,u\right)\right)\right)^{2}\right\} }=\lim_{N\to\infty}\frac{\mathbb{E}\left\{ {\rm Crt}_{N}\left(\left(u-\epsilon_{N},u\right)\right)\right\} }{\mathbb{E}\left\{ {\rm Crt}_{N}\left(\left(-\infty,u\right)\right)\right\} }=1.
\]
By Lemmas \ref{lem:n1} and \ref{lem:n3},
\[
\lim_{N\to\infty}\frac{\mathbb{E}\left\{ \left[{\rm Crt}_{N}\left(u-\epsilon_{N},u\right)\right]_{2}^{\rho_{N}}\right\} }{\left(\mathbb{E}\left\{ {\rm Crt}_{N}\left(\left(u-\epsilon_{N},u\right)\right)\right\} \right)^{2}}\leq1.
\]
For any $N$,
\[
\frac{\mathbb{E}\big\{\left({\rm Crt}_{N}\left(\left(-\infty,u\right)\right)\right)^{2}\big\}}{\big(\mathbb{E}\left\{ {\rm Crt}_{N}\left(\left(-\infty,u\right)\right)\right\} \big)^{2}}\geq1.
\]
Theorem \ref{thm:Var-E2} follows from the above.\qed

\subsection{\label{sub:l3}Proof of Lemma \ref{lem:21}}

Note that
\begin{align*}
\left(\mbox{Crt}_{N}\left(\left(-\infty,u\right)\right)\right)^{2}-\left(\mbox{Crt}_{N}\left(\left(u-\epsilon,u\right)\right)\right)^{2} & =\left(\mbox{Crt}_{N}\left(\left(-\infty,u-\epsilon\right]\right)\right)^{2}\\
 & +2\mbox{Crt}_{N}\left(\left(-\infty,u-\epsilon\right]\right)\mbox{Crt}_{N}\left(\left(u-\epsilon,u\right)\right).
\end{align*}
By Theorem \ref{thm:Var-E2-log} and the Cauchy-Schwarz inequality,
\begin{align*}
\lim_{N\to\infty}\frac{1}{N}\log\left(\mathbb{E}\left\{ \left(\mbox{Crt}_{N}\left(\left(-\infty,u\right)\right)\right)^{2}\right\} \right) & =2\Theta_{p}\left(u\right),\\
\lim_{N\to\infty}\frac{1}{N}\log\left(\mathbb{E}\left\{ \left(\mbox{Crt}_{N}\left(\left(-\infty,u-\epsilon\right]\right)\right)^{2}\right\} \right) & =2\Theta_{p}\left(u-\epsilon\right),\\
\limsup_{N\to\infty}\frac{1}{N}\log\left(\mathbb{E}\left\{ 2\mbox{Crt}_{N}\left(\left(-\infty,u-\epsilon\right]\right)\mbox{Crt}_{N}\left(\left(u-\epsilon,u\right)\right)\right\} \right) & =\Theta_{p}\left(u\right)+\Theta_{p}\left(u-\epsilon\right).
\end{align*}

For any $u<-E_{\infty}\left(p\right)$, by (\ref{eq:<}), $\Theta_{p}\left(u\right)>0$
and therefore the expressions in the last two lines above are strictly
less than $2\Theta_{p}\left(u\right)$. It follows that
\[
\lim_{N\to\infty}\mathbb{E}\left\{ \left(\mbox{Crt}_{N}\left(\left(-\infty,u\right)\right)\right)^{2}\right\} /\mathbb{E}\left\{ \left(\mbox{Crt}_{N}\left(\left(u-\epsilon,u\right)\right)\right)^{2}\right\} =1.
\]

By Remark \ref{rem:r=00003D1} and the fact that $u>-E_{0}\left(p\right)$,
also
\[
\lim_{N\to\infty}\mathbb{E}\left\{ \left[\mbox{Crt}_{N}\left(u-\epsilon,u\right)\right]_{2}^{1}\right\} /\mathbb{E}\left\{ \left(\mbox{Crt}_{N}\left(\left(u-\epsilon,u\right)\right)\right)^{2}\right\} =1.
\]

Since 
\[
\left[{\rm Crt}_{N}\left(\left(-\infty,u\right),\left(-1,1\right)\setminus\left(-\rho,\rho\right)\right)\right]_{2}\geq\left[{\rm Crt}_{N}\left(\left(u-\epsilon,u\right),\left(-1,1\right)\setminus\left(-\rho,\rho\right)\right)\right]_{2},
\]
Corollary \ref{cor:lowoverlap} implies that
\[
\lim_{N\to\infty}\mathbb{E}\left\{ \left[{\rm Crt}_{N}\left(\left(u-\epsilon,u\right),\left(-1,1\right)\setminus\left(-\rho,\rho\right)\right)\right]_{2}\right\} /\mathbb{E}\left\{ \left[\mbox{Crt}_{N}\left(u-\epsilon,u\right)\right]_{2}^{1}\right\} =0,
\]
and completes the proof.\qed

\subsection{\label{sub:Auxiliar}Auxiliary results}

The expectations in Lemmas \ref{lem:n1} and \ref{lem:n3} are expressed
by the integral formulas of Lemmas \ref{lem:1stmom} and \ref{lem:2ndKR-refinedFrom},
which by further conditioning on the value of $U$ and $U_{1}\left(r\right),U_{2}\left(r\right)$,
respectively, can be written as integrals over $J_{N}$ and $J_{N}\times J_{N}$.
In this section we prove several auxiliary results that are concerned
with the corresponding integrands. 

We now discuss elements in the proofs related to the more involved
Lemma \ref{lem:n3}. We note that the random matrices $\mathcal{\mathbf{M}}_{N-1}^{\left(i\right)}\left(r,\sqrt{N}u_{1},\sqrt{N}u_{2}\right)$
which appear in Lemma \ref{lem:2ndKR-refinedFrom} satisfy, in distribution,
\[
\left(\begin{array}{c}
\mathbf{M}_{N-1}^{\left(1\right)}\left(r,\sqrt{N}u_{1},\sqrt{N}u_{2}\right)\\
\mathbf{M}_{N-1}^{\left(2\right)}\left(r,\sqrt{N}u_{1},\sqrt{N}u_{2}\right)
\end{array}\right)=\left(\begin{array}{c}
\mathbf{X}_{N-1}^{\left(1\right)}(r)-\sqrt{N}\bar{u}_{1}I+\mathbf{E}_{N-1}^{\left(1\right)}\vphantom{\mathbf{M}_{N-1}^{\left(1\right)}\left(r,\sqrt{N}u_{1},\sqrt{N}u_{2}\right)}\\
\mathbf{X}_{N-1}^{\left(2\right)}(r)-\sqrt{N}\bar{u}_{2}I+\mathbf{E}_{N-1}^{\left(2\right)}\vphantom{\mathbf{M}_{N-1}^{\left(1\right)}\left(r,\sqrt{N}u_{1},\sqrt{N}u_{2}\right)}
\end{array}\right),
\]
where $\bar{u}_{i}=\sqrt{\frac{1}{N-1}\frac{p}{p-1}}u_{i}$, $\mathbf{X}_{N-1}^{\left(i\right)}(r)$
are correlated GOE matrices and $\mathbf{E}_{N-1}^{\left(i\right)}:=\mathbf{E}_{N-1}^{\left(i\right)}\left(r,\sqrt{N}u_{1},\sqrt{N}u_{2}\right)$
are random matrices of rank $2$, viewed as perturbations. We are
interested in values of $u_{1}$ and $u_{2}$ that are approximately
equal to some fixed $u$ and values of $r$ which are close to $0$.
In order to prove Lemma \ref{lem:n3} we will need to compute the
asymptotics of the ratio of 
\begin{equation}
\mathbb{E}\left\{ \prod_{i=1,2}\left|\det\left(\mathcal{\mathbf{M}}_{N-1}^{\left(i\right)}\left(r,\sqrt{N}u_{1},\sqrt{N}u_{2}\right)\right)\right|\right\} \mbox{\,\,\,\ and\,\,\,\,\,\,}\left(\mathbb{E}\left\{ \det\left(\mathbf{X}_{N-1}-\sqrt{N}\bar{b}_{N}I\right)\right\} \right)^{2},\label{eq:q97}
\end{equation}
where $\bar{b}_{N}=\sqrt{\frac{p}{p-1}\frac{1}{N-1}}b_{N}$ and $\mathbf{X}_{N-1}$
is a GOE matrix. This will be done in three steps: 1. we will show
that the perturbations $\mathbf{E}_{N-1}^{\left(i\right)}$ are negligible
- i.e., the expectation on the left-hand side of (\ref{eq:q97}) is
asymptotically equivalent to $\mathbb{E}\prod_{i=1,2}\left|\det\left(\mathbf{X}_{N-1}^{\left(i\right)}(r)-\sqrt{N}\bar{u}_{i}I\right)\right|$;
2. relate the latter expectation to the same without the absolute
value and with $\bar{u}_{i}=\bar{b}_{N}$; and 3. prove that taking
$\mathbf{X}_{N-1}^{\left(i\right)}(r)$ to be independent in the expectation
with $\bar{u}_{i}=\bar{b}_{N}$ asymptotically does not affect the
expectation. 

The first step is dealt with in Lemma \ref{lem:n7} where we bound
the  Hilbert-Schmidt norms of the perturbations $\mathbf{E}_{N-1}^{\left(i\right)}$
and relate them to the ratio of the perturbed and unperturbed determinants.
The importance of the assumption in Lemma \ref{lem:n3} that $u<-E_{\infty}(p)$,
is that for large $N$, we have that $-\sqrt{N}\bar{u}>2$, as in
the setting of Lemma \ref{lem:n2} below. The fact that the shifts
are greater than $2$, and thus the corresponding spectra of the shifted
GOE matrices are strictly positive, is crucial to the proof of Lemma
\ref{lem:n2} since it allows us to use concentration results for
linear statistics of the eigenvalues. The latter will be applied to
(uniformly) control the fluctuation of the corresponding determinants
and their derivatives in the shifts ($v_{i}$ in Lemma \ref{lem:n2},
which correspond to $-\sqrt{N}\bar{u}_{i}$ above). Other arguments
in the proof of Lemma \ref{lem:n2} are related to large deviations
and similar to ones we already used, e.g. in the proof of Lemma \ref{lem:bound-h}.
Once the bound on the fluctuations of the derivative in $\bar{u}_{i}$
is obtained step 2 above can be completed. Finally, in Lemma \ref{lem:n4}
we shall exploit certain Gaussian identities to analyze the expectation
of a product related to two shifted GOE matrices, assuming a certain
correlation structure. In the case where the product is of the determinants
of the two matrices, the lemma asserts that the corresponding expectation
is convex in a parameter controlling the correlation. This allows
us to relate the situation of low correlation to that where the matrices
are completely independent and complete step 3 above. We now proceed
to state and prove the auxiliary results.

With $\mathbf{X}_{i}=\mathbf{X}_{i,N-1}$, $i\leq k$, being random
$N-1\times N-1$ matrices, denote by $\mathcal{L}_{k,N-1}^{{\scriptstyle GOE}}$
the space of probability measures on $(\mathbb{R}^{N-1\times N-1})^{k}$
such that
\[
\mathbb{P}\left\{ \left(\mathbf{X}_{i}\right)_{i\leq k}\in\,\cdot\,\right\} \in\mathcal{L}_{k,N-1}^{{\scriptstyle GOE}}\Longleftrightarrow\forall i\leq k,\,\mbox{under }\mathbb{P}\left\{ \mathbf{X}_{i}\in\,\cdot\,\right\} \mbox{ is a GOE matrix}.
\]
That is, the collection of probability laws such that marginally each
$\mathbf{X}_{i}$ is a GOE matrix, but with no further assumptions
on the joint law. For a measure $\nu\in\mathcal{L}_{k,N-1}^{{\scriptstyle GOE}}$,
we will use $(\mathbf{X}_{i})_{i\leq k}\sim\nu$ to denote $\mathbb{P}\left\{ \left(\mathbf{X}_{i}\right)_{i\leq k}\in\,\cdot\,\right\} =\nu(\,\cdot\,)$. 
\begin{lem}
\label{lem:n2}Assume $(\mathbf{X}_{i})_{i\leq k}\sim\nu$ with $\nu\in\mathcal{L}_{k,N-1}^{{\scriptstyle GOE}}$
and denote by $\lambda_{j}^{(i)}$ the eigenvalues of $\mathbf{X}_{i}:=\mathbf{X}_{i,N-1}$.
Let $t_{2}>t_{1}>2$ be real numbers. Then:
\begin{enumerate}
\item For any $\delta>0$, there exists $c>0$ such that, for large enough
$N$, uniformly in $v_{i}:=v_{i,N}\in\left(-t_{2},-t_{1}\right)$
and $\nu\in\mathcal{L}_{k,N-1}^{{\scriptstyle GOE}}$,
\begin{align}
\mathbb{E}\left\{ \prod_{i=1}^{k}\left|\det\left(\mathbf{X}_{i}-v_{i}I\right)\right|\mathbf{1}\left\{ \min_{i,j}\lambda_{j}^{(i)}\leq-2-\delta\right\} \right\}  & \leq e^{-cN}\mathbb{E}\left\{ \prod_{i=1}^{k}\left|\det\left(\mathbf{X}_{i}-v_{i}I\right)\right|\right\} ,\label{eq:48-1}\\
\mathbb{E}\left\{ \prod_{i=1}^{k}\left|\det\left(\mathbf{X}_{i}-v_{i}I\right)\right|\mathbf{1}\left\{ \max_{i,j}\lambda_{j}^{(i)}\geq2+\delta\right\} \right\}  & \leq e^{-cN}\mathbb{E}\left\{ \prod_{i=1}^{k}\left|\det\left(\mathbf{X}_{i}-v_{i}I\right)\right|\right\} .\label{eq:48-2}
\end{align}

\item With $\mu^{*}$ denoting the semicircle law (\ref{eq:semicirc}),
as $N\to\infty$, uniformly in $v_{i}:=v_{i,N}\in\left(-t_{2},-t_{1}\right)$
and $\nu:=\nu_{N}\in\mathcal{L}_{k,N-1}^{{\scriptstyle GOE}}$, 
\begin{equation}
\frac{d}{dv_{1}}\log\left(\mathbb{E}\left\{ \prod_{i=1}^{k}\det\left(\mathbf{X}_{i}-v_{i}I\right)\right\} \right)=-(1+o(1))N\int\frac{1}{\lambda-v_{1}}d\mu^{*}(\lambda).\label{eq:z5}
\end{equation}

\end{enumerate}
\end{lem}
\begin{proof}
All the equalities, inequalities and limits in the proof should be
understood to hold uniformly in $v_{i}\in\left(-t_{2},-t_{1}\right)$
and $\nu\in\mathcal{L}_{k,N-1}^{{\scriptstyle GOE}}$. First we show
that 
\begin{equation}
\limsup_{N\to\infty}\frac{1}{N}\log\left(\mathbb{E}\left\{ \prod_{i=1}^{k}\left|\det\left(\mathbf{X}_{i}-v_{i}I\right)\right|\right\} \right)\leq\sum_{i=1}^{k}\Omega(v_{i}).\label{eq:z4}
\end{equation}
Recall the definition (\ref{eq:h_trunc}) of the truncation functions
$h_{\epsilon}^{\kappa}\left(x\right)$ and $h_{\kappa}^{\infty}\left(x\right)$.
Fix some $\bar{\kappa}>\bar{\epsilon}>0$. By the Cauchy-Schwarz inequality,
\begin{align*}
\mathbb{E}\left\{ \prod_{i=1}^{k}\left|\det\left(\mathbf{X}_{i}-v_{i}I\right)\right|\right\}  & \leq\left(\mathbb{E}\left\{ \prod_{i=1}^{k}\prod_{j=1}^{N-1}\left(h_{\bar{\epsilon}}^{\bar{\kappa}}\left(\left|\lambda_{j}^{(i)}-v_{i}\right|\right)\right)^{2}\right\} \right)^{1/2}\\
 & \times\left(\mathbb{E}\left\{ \prod_{i=1}^{k}\prod_{j=1}^{N-1}\left(h_{\bar{\kappa}}^{\infty}\left(x\right)\left(\left|\lambda_{j}^{(i)}-v_{i}\right|\right)\right)^{2}\right\} \right)^{1/2}.
\end{align*}
Similarly to part (\ref{enu:p2 bound-h}) of Lemma \ref{lem:bound-h},
using Lemma \ref{lem:maxEigBd} and a union bound (over $i\leq k$)
one can show that the second expectation above is smaller than $2$,
assuming $\bar{\kappa}$ is larger than some appropriate constant
$\bar{\kappa}_{0}$. From the LDP for the empirical measure of eigenvalues
of Theorem \ref{lem:GOELD} (similarly to the proof of part (\ref{enu:p1 bound-h})
of Lemma \ref{lem:bound-h}), we therefore have that%
\footnote{\label{fn:uni}We remark that uniformity in $v_{i}$ relies on the
fact that the LDP for the empirical measure of the eigenvalues is
phrased in terms of the Lipschitz bounded metric and we use the functions
$\log_{\bar{\epsilon}}^{\bar{\kappa}}\left(\left|\cdot-v_{i}\right|\right)$
which have the same bound and Lipschitz constant for all $v_{i}$.%
} 
\begin{equation}
\limsup_{N\to\infty}\frac{1}{N}\log\left(\mathbb{E}\left\{ \prod_{i=1}^{k}\left|\det\left(\mathbf{X}_{i}-v_{i}I\right)\right|\right\} \right)\leq\sum_{i=1}^{k}\int\log_{\bar{\epsilon}}^{\bar{\kappa}}\left(\left|\lambda-v_{i}\right|\right)d\mu^{*}\left(\lambda\right),\label{eq:z21}
\end{equation}
where $\log_{\bar{\epsilon}}^{\bar{\kappa}}(x)=\log(h_{\bar{\epsilon}}^{\bar{\kappa}}(x))$.
By choosing small enough $\bar{\epsilon}$ and large enough $\bar{\kappa}$
so that $\bar{\epsilon}<t_{1}-2<-v_{i}-2$ and $\bar{\kappa}>t_{2}+2>-v_{i}+2$,
(\ref{eq:z4}) follows.

Suppose $\delta,\,\epsilon,\,\kappa>0$ satisfy $0<\epsilon<t_{1}-2-\delta$
and $\kappa>t_{2}+2+\delta$. Then on the event 
\begin{equation}
A(\delta)=\left\{ -2-\delta<\min_{i,j}\lambda_{j}^{(i)}\leq\max_{i,j}\lambda_{j}^{(i)}<2+\delta\right\} \label{eq:Adelta}
\end{equation}
all the eigenvalues of $\mathbf{X}_{i}-v_{i}I$, $i\leq k$, are in
$(\epsilon,\kappa)$ and 
\[
\prod_{i=1}^{k}\det\left(\mathbf{X}_{i}-v_{i}I\right)=e^{V_{N}}\,\,\,\mbox{and}\,\,\frac{1}{N}\sum_{j=1}^{N-1}\frac{1}{\lambda_{j}^{(1)}-v_{1}}=V_{N}^{\prime},
\]
with 
\begin{equation}
V_{N}\triangleq\sum_{i=1}^{k}\sum_{j=1}^{N-1}\log\left(h_{\epsilon}^{\kappa}(\lambda_{j}^{(i)}-v_{i})\right),\,\,\, V_{N}^{\prime}\triangleq\frac{1}{N}\sum_{j=1}^{N-1}\frac{1}{h_{\epsilon}^{\kappa}\left(\lambda_{j}^{(1)}-v_{1}\right)}.\label{eq:z24}
\end{equation}
From the LDP of Theorem \ref{lem:GOELD}, as $N\to\infty$,%
\footnote{See Footnote \ref{fn:uni}.%
} 
\begin{equation}
\mathbb{E}\left\{ V_{N}^{\prime}\right\} \to\int\frac{1}{\lambda-v_{1}}d\mu^{*}(\lambda)\mbox{\,\,\,\ and\,\,}\frac{1}{N}\log\left(\mathbb{E}\left\{ e^{V_{N}}\right\} \right)\to\sum_{i=1}^{k}\Omega\left(v_{i}\right),\label{eq:z6}
\end{equation}
where we used the fact that for $\lambda$ in the support of $\mu^{*}$,
$\lambda-v_{i}\in(\epsilon,\kappa)$.

For large enough $L=L(\epsilon,\kappa)>0$, $\log\left(h_{\epsilon}^{\kappa}(x)\right)$
and $\frac{1}{N}(h_{\epsilon}^{\kappa}\left(x\right))^{-1}$ are Lipschitz
continuous with Lipschitz constant $L$ and $\frac{1}{N}L$, respectively.
Thus, by the concentration of linear statistics of Wigner matrices
as in \cite[Theorem 2.3.5]{Matrices} and the union bound, we have
that
\begin{equation}
\mathbb{P}\left\{ \left|V_{N}-\mathbb{E}V_{N}\right|>s\right\} \leq2ke^{-Cs^{2}},\,\,\mathbb{P}\left\{ \left|V_{N}^{\prime}-\mathbb{E}V_{N}^{\prime}\right|>s\right\} \leq2e^{-N^{2}Cs^{2}},\label{eq:conc}
\end{equation}
for some constant $C>0$. By the LDP for the maximal (and by symmetry,
minimal) eigenvalue of $\mathbf{X}_{i}$ (see Theorem \ref{thm:maxeigLDP}),
\begin{equation}
\limsup_{N\to\infty}\frac{1}{N}\log\left(\mathbb{P}\left\{ (A(\delta))^{c}\right\} \right)<0.\label{eq:z22}
\end{equation}
Therefore, using (\ref{eq:conc}) and the Cauchy-Schwarz inequality
we have that, as $N\to\infty$,
\begin{equation}
\mathbb{E}\left\{ V_{N}^{\prime}e^{V_{N}}\right\} =\mathbb{E}\left\{ V_{N}^{\prime}\right\} \mathbb{E}\left\{ e^{V_{N}}\right\} (1+o(1)),\label{eq:z7}
\end{equation}
\[
\mathbb{E}\left\{ V_{N}^{\prime}e^{V_{N}}\mathbf{1}_{(A(\delta))^{c}}\right\} \leq\left(\mathbb{E}\left\{ \left(V_{N}^{\prime}e^{V_{N}}\right)^{2}\right\} \mathbb{P}\left\{ (A(\delta))^{c}\right\} \right)^{1/2}=o\left(\mathbb{E}\left\{ V_{N}^{\prime}e^{V_{N}}\right\} \right)
\]
and similarly 
\begin{equation}
\mathbb{E}\left\{ V_{N}^{\prime}\mathbf{1}_{(A(\delta))^{c}}\right\} =o\left(\mathbb{E}\left\{ V_{N}^{\prime}\right\} \right),\,\,\,\,\mathbb{E}\left\{ e^{V_{N}}\mathbf{1}_{(A(\delta))^{c}}\right\} =o\left(\mathbb{E}\left\{ e^{V_{N}}\right\} \right).\label{eq:z9}
\end{equation}

Since $\prod_{i=1}^{k}\left|\det\left(\mathbf{X}_{i}-v_{i}I\right)\right|\geq e^{V_{N}}\mathbf{1}_{A(\delta)}$,
from (\ref{eq:z4}), (\ref{eq:z6}) and (\ref{eq:z9}) we have
\begin{equation}
\lim_{N\to\infty}\frac{1}{N}\log\left(\mathbb{E}\left\{ \prod_{i=1}^{k}\left|\det\left(\mathbf{X}_{i}-v_{i}I\right)\right|\right\} \right)=\sum_{i=1}^{k}\Omega(v_{i}).\label{eq:z4-1}
\end{equation}
Since $k\geq1$ was general, by taking two copies of each of the matrices
in (\ref{eq:z4-1}), we also have
\begin{equation}
\lim_{N\to\infty}\frac{1}{N}\log\left(\mathbb{E}\left\{ \prod_{i=1}^{k}\left|\det\left(\mathbf{X}_{i}-v_{i}I\right)\right|^{2}\right\} \right)=2\sum_{i=1}^{k}\Omega(v_{i}),\label{eq:z22-1}
\end{equation}
and by (\ref{eq:z22}) and the Cauchy-Schwarz inequality, the first
part of Lemma \ref{lem:n2} follows.

Since $\det\left(\mathbf{X}_{i}-v_{i}I\right)$ is a polynomial function
of the Gaussian entries of $\mathbf{X}_{i}-v_{i}I$, the left-hand
side of (\ref{eq:z5}) is equal to 
\[
\frac{d}{dv_{1}}\log\left(\mathbb{E}\left\{ Y_{N}\right\} \right)=\frac{\mathbb{E}\left\{ \frac{d}{dv_{1}}Y_{N}\right\} }{\mathbb{E}\left\{ Y_{N}\right\} }=-\frac{N\mathbb{E}\left\{ Y_{N}Z_{N}\right\} }{\mathbb{E}\left\{ Y_{N}\right\} },
\]
where we denote 
\[
Y_{N}=\prod_{i=1}^{k}\det\left(\mathbf{X}_{i}-v_{i}I\right)\mbox{\,\,\,\ and\,\,}Z_{N}=\frac{1}{N}\sum_{j=1}^{N-1}\frac{1}{\lambda_{j}^{(1)}-v_{1}}.
\]

By (\ref{eq:z7}), (\ref{eq:z9}) and (\ref{eq:z6}), as $N\to\infty$,
\begin{align*}
\frac{\mathbb{E}\left\{ Y_{N}Z_{N}\mathbf{1}_{A(\delta)}\right\} }{\mathbb{E}\left\{ Y_{N}\mathbf{1}_{A(\delta)}\right\} } & =\frac{\mathbb{E}\left\{ V_{N}^{\prime}e^{V_{N}}\mathbf{1}_{A(\delta)}\right\} }{\mathbb{E}\left\{ e^{V_{N}}\mathbf{1}_{A(\delta)}\right\} }\\
 & =(1+o(1))\mathbb{E}\left\{ V_{N}^{\prime}\right\} =(1+o(1))\int\frac{1}{\lambda-v_{1}}d\mu^{*}(\lambda),
\end{align*}
where the first equality follows since on $A(\delta)$ all the eigenvalues
of $\mathbf{X}_{i}-v_{i}I$, $i\leq k$, are in $(\epsilon,\kappa)$.
By the first part of Lemma \ref{lem:n2}, since $Y_{N}=|Y_{N}|$ on
$A(\delta)$, 
\begin{equation}
\frac{\mathbb{E}\left\{ |Y_{N}|\mathbf{1}_{(A(\delta))^{c}}\right\} }{\mathbb{E}\left\{ |Y_{N}|\right\} }\overset{{\scriptstyle N\to\infty}}{\longrightarrow}0\mbox{\,\,\,\ and\,\,\,\,\,\,}\frac{\mathbb{E}\left\{ Y_{N}\mathbf{1}_{A(\delta)}\right\} }{\mathbb{E}\left\{ Y_{N}\right\} }\overset{{\scriptstyle N\to\infty}}{\longrightarrow}1.\label{eq:wq1}
\end{equation}
What remains to show in order to complete the proof of (\ref{eq:z5})
is that 
\begin{equation}
\frac{\mathbb{E}\left\{ Y_{N}Z_{N}\mathbf{1}_{A(\delta)}\right\} }{\mathbb{E}\left\{ Y_{N}Z_{N}\right\} }\overset{{\scriptstyle N\to\infty}}{\longrightarrow}1.\label{eq:z23}
\end{equation}

Note that for any $\bar{\epsilon}>0$, 
\[
\left|Y_{N}Z_{N}\right|\leq\frac{1}{\bar{\epsilon}}\prod_{i=1}^{k}\prod_{j=1}^{N-1}h_{\bar{\epsilon}}\left(\left|\lambda_{j}^{(i)}-v_{i}\right|\right)
\]
and similarly to the proof of (\ref{eq:z4}) and (\ref{eq:z21}),
by letting $\bar{\epsilon}\to0$, it can be shown that
\[
\limsup_{N\to\infty}\frac{1}{2N}\log\left(\mathbb{E}\left\{ \left(Y_{N}Z_{N}\right)^{2}\right\} \right)\leq\sum_{i=1}^{k}\Omega(v_{i}).
\]
On $A(\delta)$, $Z_{N}\in(c_{1},c_{2})$ for appropriate constants
$0<c_{1}<c_{2}$. Thus, from (\ref{eq:wq1}) and (\ref{eq:z4-1}),
\[
\lim_{N\to\infty}\frac{1}{N}\log\left(\mathbb{E}\left\{ Y_{N}Z_{N}\mathbf{1}_{A(\delta)}\right\} \right)=\lim_{N\to\infty}\frac{1}{N}\log\left(\mathbb{E}\left\{ Y_{N}\mathbf{1}_{A(\delta)}\right\} \right)=\sum_{i=1}^{k}\Omega(v_{i}).
\]
From the Cauchy-Schwarz inequality and (\ref{eq:z22}),
\[
\frac{\mathbb{E}\left\{ \left|Y_{N}Z_{N}\right|\mathbf{1}_{\left(A(\delta)\right)^{c}}\right\} }{\mathbb{E}\left\{ Y_{N}Z_{N}\mathbf{1}_{A(\delta)}\right\} }\overset{{\scriptstyle N\to\infty}}{\longrightarrow}0.
\]
This implies (\ref{eq:z23}) and the proof is completed.\end{proof}
\begin{cor}
\label{cor:n6}Let $u<-E_{\infty}\left(p\right)$ and suppose $J_{N}=\left(a_{N},b_{N}\right)$
is an interval such that $a_{N},\, b_{N}\to u$ as $N\to\infty$.
Assume $(\mathbf{X}_{i})_{i\leq k}\sim\nu$ with $\nu\in\mathcal{L}_{k,N-1}^{{\scriptstyle GOE}}$.
Then, uniformly in $u_{i}:=u_{i,N}\in J_{N}$ and $\nu:=\nu_{N}\in\mathcal{L}_{k,N-1}^{{\scriptstyle GOE}}$,
as $N\to\infty$,
\begin{align}
 & \negthickspace\negthickspace\negthickspace\log\left(\mathbb{E}\left\{ \prod_{i=1}^{k}\left|\det\left(\mathbf{X}_{i}-\sqrt{N}\bar{u}_{i}I\right)\right|\right\} \right)=\log\left(\mathbb{E}\left\{ \prod_{i=1}^{k}\det\left(\mathbf{X}_{i}-\sqrt{N}\bar{u}_{i}I\right)\right\} \right)+o(1)\nonumber \\
 & =\log\left(\mathbb{E}\left\{ \prod_{i=1}^{k}\det\left(\mathbf{X}_{i}-\sqrt{N}\bar{b}_{N}I\right)\right\} \right)+o(1)+N\mathfrak{S}(u)\sum_{i=1}^{k}(1+o(1))(b_{N}-u_{i}),\label{eq:z19}
\end{align}
 where $\bar{b}_{N}=\sqrt{\frac{p}{p-1}\frac{1}{N-1}}b_{N}$, $\bar{u}_{i}=\sqrt{\frac{p}{p-1}\frac{1}{N-1}}u_{i}$
and $\mathfrak{S}(u)$ is given by (\ref{eq:S(u)}).\end{cor}
\begin{proof}
From our assumption on $u$, for some $t_{2}>t_{1}>2$,  for large
$N$, $\sqrt{N}\bar{b}_{N},\,\sqrt{N}\bar{u}_{i}\in\left(-t_{2},-t_{1}\right)$
for any $u_{i}\in J_{N}$. On the event $A(\delta)$ defined in (\ref{eq:Adelta}),
for small enough $\delta$,
\[
\prod_{i=1}^{k}\left|\det\left(\mathbf{X}_{i}-\sqrt{N}\bar{u}_{i}I\right)\right|=\prod_{i=1}^{k}\det\left(\mathbf{X}_{i}-\sqrt{N}\bar{u}_{i}I\right).
\]
Therefore, the first equality in (\ref{eq:z19}) follows from the
first part of Lemma \ref{lem:n2} which asserts that, as $N\to\infty$,
\[
\mathbb{E}\left\{ \prod_{i=1}^{k}\left|\det\left(\mathbf{X}_{i}-\sqrt{N}\bar{u}_{i}I\right)\right|\mathbf{1}_{(A(\delta))^{c}}\right\} =o(1)\mathbb{E}\left\{ \prod_{i=1}^{k}\left|\det\left(\mathbf{X}_{i}-\sqrt{N}\bar{u}_{i}I\right)\right|\right\} .
\]
From the second part of Lemma \ref{lem:n2}, 
\begin{align*}
\log\left(\mathbb{E}\left\{ \prod_{i=1}^{k}\det\left(\mathbf{X}_{i}-\sqrt{N}\bar{u}_{i}I\right)\right\} \right) & =\log\left(\mathbb{E}\left\{ \prod_{i=1}^{k}\det\left(\mathbf{X}_{i}-\sqrt{N}\bar{b}_{N}I\right)\right\} \right)\\
 & +N^{3/2}\int\frac{1}{\lambda-\sqrt{N}\bar{u}}d\mu^{*}(\lambda)\sum_{i=1}^{k}(1+o(1))(\bar{b}_{N}-\bar{u}_{i}),
\end{align*}
as $N\to\infty$, uniformly in $u_{i}\in J_{N}$ and $\nu\in\mathcal{L}_{k,N-1}^{\mbox{GOE}}$.
This completes the proof.\end{proof}
\begin{cor}
\label{cor:n5}Let $u<-E_{\infty}\left(p\right)$ and suppose $J_{N}=\left(a_{N},b_{N}\right)$
is an interval such that $a_{N},\, b_{N}\to u$ as $N\to\infty$.
Assume $(\mathbf{X}_{i})_{i\leq k}\sim\nu$ with $\nu\in\mathcal{L}_{k,N-1}^{{\scriptstyle GOE}}$.
Then, uniformly in $u_{i}:=u_{i,N}\in J_{N}$ and $\nu:=\nu_{N}\in\mathcal{L}_{k,N-1}^{{\scriptstyle GOE}}$,
 
\begin{equation}
\mathbb{E}\left\{ \prod_{i=1}^{k}\det\left(\mathbf{X}_{i}-\sqrt{N}\bar{u}_{i}I\right)\right\} \leq c_{k}\prod_{i=1}^{k}\mathbb{E}\left\{ \det\left(\mathbf{X}_{i}-\sqrt{N}\bar{u}_{i}I\right)\right\} ,\label{eq:z18}
\end{equation}
for appropriate constants $c_{k}>0$ independent of $N$, where $\bar{u}_{i}=\sqrt{\frac{p}{p-1}\frac{1}{N-1}}u_{i}$.\end{cor}
\begin{proof}
From our assumption on $u$ for some $t_{2}>t_{1}>2$, for large $N$,
$\sqrt{N}\bar{u},\,\sqrt{N}\bar{u}_{i}\in\left(-t_{2},-t_{1}\right)$
for any $u_{i}\in J_{N}$. Let $\lambda_{j}^{(i)}$ denote the eigenvalues
of $\mathbf{X}_{i}$ and recall the definition of $A(\delta)$ given
in (\ref{eq:Adelta}). From the first part of Lemma \ref{lem:n2}
for small $\delta>0$, uniformly in $u_{i}:=u_{i,N}\in J_{N}$ and
$\nu:=\nu_{N}\in\mathcal{L}_{k,N-1}^{{\scriptstyle GOE}}$, as $N\to\infty$,
\[
\mathbb{E}\left\{ \prod_{i=1}^{k}\det\left(\mathbf{X}_{i}-\sqrt{N}\bar{u}_{i}I\right)\right\} =(1+o(1))\mathbb{E}\left\{ \prod_{i=1}^{k}\det\left(\mathbf{X}_{i}-\sqrt{N}\bar{u}_{i}I\right)\mathbf{1}_{A(\delta)}\right\} .
\]
For small enough $\delta,\,\epsilon>0$ and large enough $\kappa>0$,
on $A(\delta)$ we have that $\prod_{i=1}^{k}\det\left(\mathbf{X}_{i}-\sqrt{N}\bar{u}_{i}I\right)=e^{\bar{V}_{N}}$,
where 
\[
\bar{V}_{N}\triangleq\sum_{i=1}^{k}\sum_{j=1}^{N-1}\log\left(h_{\epsilon}^{\kappa}(\lambda_{j}^{(i)}-\sqrt{N}\bar{u}_{i})\right)
\]
is defined similarly to $V_{N}$ (see (\ref{eq:z24})). Similarly
to (\ref{eq:conc}), by the concentration of linear statistics of
Wigner matrices as in \cite[Theorem 2.3.5]{Matrices}, defining $\bar{V}_{N,i}=\sum_{j=1}^{N-1}\log\left(h_{\epsilon}^{\kappa}(\lambda_{j}^{(i)}-\sqrt{N}\bar{u}_{i})\right)$,
we have for all $i\leq k$,
\begin{equation}
\mathbb{P}\left\{ \left|\bar{V}_{N,i}-\mathbb{E}\bar{V}_{N,i}\right|>s\right\} \leq2ke^{-Cs^{2}},\label{eq:conc-1}
\end{equation}
with some constant $C=C(\epsilon,\kappa)>0$ that depends on the Lipschitz
constant of $\log\left(h_{\epsilon}^{\kappa}(x)\right)$. From the
above, (\ref{eq:z18}) follows.\end{proof}
\begin{lem}
\label{lem:n7}Let $u<-E_{\infty}\left(p\right)$ and suppose $J_{N}=\left(a_{N},b_{N}\right)$
is an interval such that $a_{N},\, b_{N}\to u$ as $N\to\infty$.
Let $0<\rho_{N}=o(1)$, let $\mathbf{X}_{iid}^{\left(i\right)}=\mathbf{X}_{iid,N-1}^{\left(i\right)}$,
$i=1,2$, and $\mathbf{X}_{iid}=\mathbf{X}_{iid,N-1}$ be three i.i.d
GOE matrices of dimension $N-1$, and set 
\begin{equation}
\mathbf{X}_{N-1}^{\left(i\right)}\left(r\right)=\sqrt{1-\left|r\right|^{p-2}}\mathbf{X}_{iid}^{\left(i\right)}+\left(\mbox{sgn}\left(r\right)\right)^{ip}\sqrt{\left|r\right|^{p-2}}\mathbf{X}_{iid}.\label{eq:z14}
\end{equation}
Let $\mathcal{\mathbf{M}}_{N-1}^{\left(i\right)}\left(r,u_{1},u_{2}\right)$
be as defined in Lemma \ref{lem:Hess_struct_2} and set $\bar{u}_{i}=\sqrt{\frac{p}{p-1}\frac{1}{N-1}}u_{i}$.
Then, as $N\to\infty$, uniformly in $u_{i}:=u_{i.N}\in J_{N}$ and
$r:=r_{N}\in(-\rho_{N},\rho_{N})$, 
\[
\mathbb{E}\left\{ \prod_{i=1,2}\left|\det\left(\mathcal{\mathbf{M}}_{N-1}^{\left(i\right)}\left(r,\sqrt{N}u_{1},\sqrt{N}u_{2}\right)\right)\right|\right\} \leq(1+o(1))\mathbb{E}\left\{ \prod_{i=1,2}\left|\det\left(\mathbf{X}_{N-1}^{\left(i\right)}\left(r\right)-\sqrt{N}\bar{u}_{i}I\right)\right|\right\} .
\]
\end{lem}
\begin{proof}
We start from the representation of Lemma \ref{lem:Hess_struct_2}.
Conditional on $f\left(\mathbf{n}\right)=\sqrt{N}u_{1},\, f\left(\boldsymbol{\sigma}\left(r\right)\right)=\sqrt{N}u_{2}$
and $\nabla f\left(\mathbf{n}\right)=\nabla f\left(\boldsymbol{\sigma}\left(r\right)\right)=0$
we have that, in distribution,
\[
\left(\begin{array}{c}
\frac{\nabla^{2}f\left(\mathbf{n}\right)}{\sqrt{\left(N-1\right)p\left(p-1\right)}}\\
\frac{\nabla^{2}f\left(\boldsymbol{\sigma}\left(r\right)\right)}{\sqrt{\left(N-1\right)p\left(p-1\right)}}
\end{array}\right)=\left(\begin{array}{c}
\mathbf{M}_{N-1}^{\left(1\right)}\left(r,\sqrt{N}u_{1},\sqrt{N}u_{2}\right)\\
\mathbf{M}_{N-1}^{\left(2\right)}\left(r,\sqrt{N}u_{1},\sqrt{N}u_{2}\right)
\end{array}\right),
\]
with
\begin{align*}
\mathbf{M}_{N-1}^{\left(i\right)}\left(r,u_{1},u_{2}\right) & =\hat{\mathbf{M}}_{N-1}^{\left(i\right)}\left(r\right)-\sqrt{N}\bar{u}_{i}I+\frac{m_{i}\left(r,\sqrt{N}u_{1},\sqrt{N}u_{2}\right)}{\sqrt{\left(N-1\right)p\left(p-1\right)}}e_{N-1,N-1},\\
\hat{\mathbf{M}}_{N-1}^{\left(i\right)}\left(r\right) & =\left(\begin{array}{cc}
\hat{\mathbf{G}}_{N-2}^{\left(i\right)}\left(r\right) & Z^{\left(i\right)}\left(r\right)\\
\left(Z^{\left(i\right)}\left(r\right)\right)^{T} & Q^{\left(i\right)}\left(r\right)
\end{array}\right),\\
\hat{\mathbf{G}}^{\left(i\right)} & =\sqrt{1-\left|r\right|^{p-2}}\bar{\mathbf{G}}^{\left(i\right)}+\left(\mbox{sgn}\left(r\right)\right)^{ip}\sqrt{\left|r\right|^{p-2}}\bar{\mathbf{G}},
\end{align*}
where all the variables are as described in Lemma \ref{lem:Hess_struct_2}.

Denote by $\tilde{\mathbf{X}}_{N-1}^{\left(i\right)}\left(r\right)$
the matrix obtained from $\mathbf{X}_{N-1}^{\left(i\right)}\left(r\right)$
(defined in (\ref{eq:z14})) by replacing every element not in the
last row or column by $0$ and denote by $\bar{\mathbf{X}}_{N-2}^{\left(i\right)}\left(r\right)$
the upper-left $N-2\times N-2$ submatrix of $\mathbf{X}_{N-1}^{\left(i\right)}\left(r\right)$.
Couple the variables so that, almost surely, 
\begin{equation}
\bar{\mathbf{X}}_{N-2}^{\left(i\right)}\left(r\right)=\hat{\mathbf{G}}_{N-2}^{\left(i\right)}\left(r\right),\label{eq:17}
\end{equation}
and, denoting by $(\mathbf{A})_{i,j}$ the $i,j$ element of a general
matrix $\mathbf{A}$, 
\begin{align}
Z_{j}^{\left(i\right)}\left(r\right) & =\sqrt{\frac{\Sigma_{Z,11}\left(r\right)-\left|\Sigma_{Z,12}\left(r\right)\right|}{p(p-1)}}\left(\mathbf{X}_{iid}^{\left(i\right)}\right)_{j,N-1}+\left(\mbox{sgn}\left(\Sigma_{Z,12}\left(r\right)\right)\right)^{i}\sqrt{\frac{\left|\Sigma_{Z,12}\left(r\right)\right|}{p(p-1)}}\left(\mathbf{X}_{iid}\right)_{j,N-1},\nonumber \\
Q_{i}\left(r\right) & =\sqrt{\frac{\Sigma_{Q,11}\left(r\right)-\left|\Sigma_{Q,12}\left(r\right)\right|}{2p(p-1)}}\left(\mathbf{X}_{iid}^{\left(i\right)}\right)_{N-1,N-1}+\left(\mbox{sgn}\left(\Sigma_{Q,12}\left(r\right)\right)\right)^{i}\sqrt{\frac{\left|\Sigma_{Q,12}\left(r\right)\right|}{2p(p-1)}}\left(\mathbf{X}_{iid}\right)_{N-1,N-1}.\label{eq:ZQGcoupling}
\end{align}

Define 
\[
\mathbf{T}_{N-1}^{\left(i\right)}\left(r\right)\triangleq\left(\begin{array}{cc}
0 & Z^{\left(i\right)}\left(r\right)\\
\left(Z^{\left(i\right)}\left(r\right)\right)^{T} & Q_{i}\left(r\right)
\end{array}\right)-\tilde{\mathbf{X}}_{N-1}^{\left(i\right)}\left(r\right),
\]
and note that 
\[
\hat{\mathbf{M}}_{N-1}^{\left(i\right)}\left(r\right)=\mathbf{X}_{N-1}^{\left(i\right)}\left(r\right)+\mathbf{T}_{N-1}^{\left(i\right)}\left(r\right).
\]

For a general matrix $\mathbf{A}$ with eigenvalues $\lambda_{i}\left(\mathbf{A}\right)$,
denote $\lambda_{*}\left(\mathbf{A}\right)=\max_{i}\left|\lambda_{i}\left(\mathbf{A}\right)\right|$.
Define the event 
\[
E_{N}(\delta)=\cap_{r\in(-\rho_{N},\rho_{N})}\left(\cap_{i=1,2}\left\{ \lambda_{*}\left(\mathbf{X}_{N-1}^{\left(i\right)}\left(r\right)\right)<2+\eta\right\} \cap\left\{ \lambda_{*}\left(\mathbf{T}_{N-1}^{\left(i\right)}\left(r\right)\right)<\delta\right\} \right),
\]
where $\eta>0$, which will be fixed from now on, is such that 
\[
\lambda_{*}\left(\mathbf{X}_{N-1}^{\left(i\right)}\left(r\right)\right)<2+\eta\Longrightarrow\min_{j}\lambda_{j}\left(\mathbf{X}_{N-1}^{\left(i\right)}-\sqrt{N}\bar{u}_{i}I\right)>\eta,
\]
for large $N$, uniformly in $u_{i}\in J_{N}$ (which is possible
to choose since $u<-E_{\infty}(p)$). Note that
\[
\mathbf{M}_{N-1}^{\left(i\right)}\left(r,u_{1},u_{2}\right)=\underset{\triangleq\mathbf{D}_{N-1}^{\left(i\right)}\left(r,\sqrt{N}u_{1},\sqrt{N}u_{2}\right)}{\underbrace{\vphantom{{\frac{m_{i}\left(r,u_{1},u_{2}\right)}{\sqrt{Np\left(p-1\right)}}}}\mathbf{X}_{N-1}^{\left(i\right)}(r)-\sqrt{N}\bar{u}_{i}I}}+\underset{\triangleq\mathbf{E}_{N-1}^{\left(i\right)}\left(r,\sqrt{N}u_{1},\sqrt{N}u_{2}\right)}{\underbrace{\mathbf{T}_{N-1}^{\left(i\right)}\left(r\right)+\frac{m_{i}\left(r,\sqrt{N}u_{1},\sqrt{N}u_{2}\right)}{\sqrt{(N-1)p\left(p-1\right)}}e_{N-1,N-1}}.}
\]

The rank of $\mathbf{E}_{N-1}^{\left(i\right)}\left(r,\sqrt{N}u_{1},\sqrt{N}u_{2}\right)$,
and therefore the number of non-zero eigenvalues, is $2$ at most.
On $E_{N}(\delta)$, the eigenvalues $\mathbf{E}_{N-1}^{\left(i\right)}\left(r,\sqrt{N}u_{1},\sqrt{N}u_{2}\right)$
are bounded in absolute value by 
\[
\delta+2\frac{\sup_{u_{i}\in J_{N}}\left|m_{i}\left(r,\sqrt{N}u_{1},\sqrt{N}u_{2}\right)\right|}{\sqrt{p\left(p-1\right)}}\overset{N\to\infty}{\longrightarrow}\delta,
\]
uniformly in $u_{i}\in J_{N}$ and $r\in(-\rho_{N},\rho_{N})$. From
the bound (\ref{eq:miroslav2}) of Corollary \ref{cor:Miroslav} with
$\mathbf{C}_{1}=\mathbf{D}_{N-1}^{\left(i\right)}\left(r,\sqrt{N}u_{1},\sqrt{N}u_{2}\right)$
and $\mathbf{C}_{2}=\mathbf{E}_{N-1}^{\left(i\right)}\left(r,\sqrt{N}u_{1},\sqrt{N}u_{2}\right)$
we obtain that on $E_{N}(\delta)$, for large enough $N$, for any
$u_{i}\in J_{N}$ and $r\in(-\rho_{N},\rho_{N})$, 
\begin{equation}
\left|\det\left(\mathbf{M}_{N-1}^{\left(i\right)}\left(r,\sqrt{N}u_{1},\sqrt{N}u_{2}\right)\right)\right|\leq\left|\det\left(\mathbf{X}_{N-1}^{\left(i\right)}(r)-\sqrt{N}\bar{u}_{i}I\right)\right|\cdot\left(1+2\frac{\delta}{\eta}\right)^{2}.\label{eq:81}
\end{equation}
In order to conclude the proof of Lemma \ref{lem:n7}, it will be
enough to show that for any $\delta$,  uniformly in $u_{i}\in J_{N}$
and $r\in(-\rho_{N},\rho_{N})$,
\begin{equation}
\lim_{N\to\infty}\frac{\mathbb{E}\left\{ \prod_{i=1,2}\left|\det\left(\mathcal{\mathbf{M}}_{N-1}^{\left(i\right)}\left(r,\sqrt{N}u_{1},\sqrt{N}u_{2}\right)\right)\right|\mathbf{1}\left\{ \left(E_{N}(\delta)\right)^{c}\right\} \right\} }{\mathbb{E}\left\{ \prod_{i=1,2}\left|\det\left(\mathbf{X}_{N-1}^{\left(i\right)}(r)-\sqrt{N}\bar{u}_{i}I\right)\right|\mathbf{1}\left\{ E_{N}(\delta)\right\} \right\} }=0.\label{eq:z16}
\end{equation}

By (\ref{eq:z4-1}) (which holds uniformly in $v_{i}\in(-t_{2},-t_{1})$
as in the statement of Lemma \ref{lem:n2}), uniformly in $u_{i}\in J_{N}$
and $r\in(-\rho_{N},\rho_{N})$, 
\begin{align}
\lim_{N\to\infty}\frac{1}{N}\log\left(\mathbb{E}\left\{ \prod_{i=1,2}\left|\det\left(\mathbf{X}_{N-1}^{\left(i\right)}(r)-\sqrt{N}\bar{u}_{i}I\right)\right|\right\} \right) & =\label{eq:q12}\\
\frac{1}{2}\cdot\lim_{N\to\infty}\frac{1}{N}\log\left(\mathbb{E}\left\{ \prod_{i=1,2}\left|\det\left(\mathbf{X}_{N-1}^{\left(i\right)}(r)-\sqrt{N}\bar{u}_{i}I\right)\right|^{2}\right\} \right) & =\sum_{i=1,2}\Omega\left(\sqrt{\frac{p}{p-1}}u_{i}\right).\nonumber 
\end{align}
By Lemmas \ref{cor:r2pert} and \ref{lem:W bound} and the Cauchy-Schwarz
inequality, for any $\epsilon>0$, uniformly in $u_{i}\in J_{N}$
and $r\in(-\rho_{N},\rho_{N})$, 
\begin{align}
 & \frac{1}{2}\cdot\limsup_{N\to\infty}\frac{1}{N}\log\left(\mathbb{E}\left\{ \prod_{i=1,2}\left|\det\left(\mathcal{\mathbf{M}}_{N-1}^{\left(i\right)}\left(r,u_{1}\sqrt{N},u_{2}\sqrt{N}\right)\right)\right|^{2}\right\} \right)\label{eq:z51}\\
 & \leq\frac{1}{4}\cdot\limsup_{N\to\infty}\frac{1}{N}\log\left(\mathbb{E}\left\{ \prod_{i=1,2}\prod_{j=1}^{N-2}\left(h_{\epsilon}\left(\left|\lambda_{j}\left(\hat{\mathbf{G}}_{N-2}^{\left(i\right)}\left(r\right)-\sqrt{N}\bar{u}_{i}I\right)\right|\right)\right)^{4}\right\} \right),\nonumber 
\end{align}
where $h_{\epsilon}(x)=\max\{\epsilon,x\}$. By the same arguments
used to derive (\ref{eq:z21}) and by letting $\epsilon\to0$, we
obtain that \ref{eq:z51} is bounded from above by $\sum_{i=1,2}\Omega\left(\sqrt{\frac{p}{p-1}}u_{i}\right)$. 

If we prove that for large $N$, 
\begin{equation}
\mathbb{P}\left\{ \left(E_{N}(\delta)\right)^{c}\right\} <e^{-C_{0}N}\label{eq:z12-1}
\end{equation}
for some $C_{0}=C_{0}(\delta)>0$, then from the above and the Cauchy-Schwarz
inequality we would have that the limit supremum of $\frac{1}{N}\log$
of the numerator of (\ref{eq:z16}) and the limit supremum of $\frac{1}{N}\log$
of 
\[
\mathbb{E}\left\{ \prod_{i=1,2}\left|\det\left(\mathbf{X}_{N-1}^{\left(i\right)}(r)-\sqrt{N}\bar{u}_{i}I\right)\right|\mathbf{1}\left\{ \left(E_{N}(\delta)\right)^{c}\right\} \right\} 
\]
are both asymptotically strictly smaller than $\sum_{i=1,2}\Omega\left(\sqrt{\frac{p}{p-1}}u_{i}\right)$,
which together with (\ref{eq:q12}), would imply (\ref{eq:z16}). 

From the LDP of the maximal eigenvalue of GOE matrices (see Theorem
\ref{thm:maxeigLDP}) and (\ref{eq:z14}), 
\[
\mathbb{P}\left\{ \sup_{r\in(-r_{N},r_{N})}\lambda_{*}\left(\mathbf{X}_{N-1}^{\left(i\right)}\left(r\right)\right)\geq2+\eta\right\} <e^{-C_{1}N}
\]
for some $C_{1}>0$, for large $N$. Thus, in order to prove (\ref{eq:z12-1})
it is enough to show that, for large $N$,
\begin{equation}
\mathbb{P}\left\{ \sup_{r\in(-r_{N},r_{N})}\lambda_{*}\left(\mathbf{T}_{N-1}^{\left(i\right)}\left(r\right)\right)\geq\delta\right\} <e^{-C_{2}N}\label{eq:z13}
\end{equation}
for some $C_{2}=C_{2}(\delta)>0$. From (\ref{eq:ZQGcoupling}) and
the expressions for $\Sigma_{Z}$ and $\Sigma_{Q}$ (\ref{eq:86}),
it follows that any element of $\mathbf{T}_{N-1}^{\left(i\right)}\left(r\right)$
in the last row or column can be written as 
\[
\alpha_{1}\left(r\right)\left(\mathbf{X}_{iid}^{\left(i\right)}\right)_{j,N-1}+\alpha_{2}\left(r\right)\left(\mathbf{X}_{iid}\right)_{j,N-1},
\]
for some $j\leq N-1$, such that $\sup_{i\in\{1,2\},r\in(-r_{N},r_{N})}\left|\alpha_{i}\left(r\right)\right|\to0$,
as $N\to\infty$. The variance of the Gaussian elements of $\mathbf{X}_{iid}^{\left(i\right)}$
and $\mathbf{X}_{iid}$ is bounded from above by $2/(N-1)$. Also,
\[
2\sum_{m=1}^{N-1}\left(\left(\mathbf{T}_{N-1}^{\left(i\right)}\left(r\right)\right)_{N-1,m}\right)^{2}\geq\left(\lambda_{*}\left(\mathbf{T}_{N-1}^{\left(i\right)}\left(r\right)\right)\right)^{2}.
\]
Using, for example, Cram{\'{e}}r's theorem \cite[Theorem 2.2.3]{LDbook},
(\ref{eq:z13}) follows and the proof is completed.\end{proof}
\begin{lem}
\label{lem:n4}For any $\rho\in[-1,1]$, let $\mathbf{W}_{N}^{\left(1\right)}\left(\rho\right)$
and $\mathbf{W}_{N}^{\left(2\right)}\left(\rho\right)$ be $N\times N$
centered jointly Gaussian Wigner matrices with
\begin{equation}
{\rm Cov}\left(\mathbf{W}_{ij}^{\left(m\right)}\left(\rho\right),\mathbf{W}_{kl}^{\left(n\right)}\left(\rho\right)\right)=\delta_{\left\{ i,j\right\} =\left\{ k,l\right\} }\left(1+\delta_{i=j}\right)\left(\rho+\left(1-\rho\right)\delta_{m=n}\right).\label{eq:covrho}
\end{equation}
Let $g:\mathbb{R}^{N\times N}\to\mathbb{R}$ be a smooth function
and assume all its derivatives have a $O(|x|^{n})$ growth rate at
infinity. If we define 
\[
\hat{g}\left(\rho\right)\triangleq\mathbb{E}\left\{ g\left(\mathbf{W}_{N}^{\left(1\right)}\left(\rho\right)\right)g\left(\mathbf{W}_{N}^{\left(2\right)}\left(\rho\right)\right)\right\} ,
\]
then $\frac{d^{k}}{d\rho^{k}}\hat{g}\left(0\right)\geq0$ for all
$k\geq1$. In particular, if $g(\mathbf{A})$ is a polynomial function
of the elements of $\mathbf{A}$, then $\hat{g}:[-1,1]\to\mathbb{R}$
is a polynomial function, it is convex on $[0,1]$, and for any $\rho\in[0,1]$
it satisfies 
\begin{equation}
\left|\hat{g}(-\rho)-\hat{g}(0)\right|\leq\hat{g}(\rho)-\hat{g}(0)\leq\rho\left(\hat{g}(1)-\hat{g}(0)\right).\label{eq:z28}
\end{equation}
\end{lem}
\begin{proof}
In the current proof for any function $h\left(\mathbf{A}\right)$
of a symmetric matrix $\mathbf{A}$, we denote 
\[
\frac{\partial}{\partial\mathbf{A}_{ij}}h\left(\mathbf{A}\right):=\lim_{t\to0}\left(h\left(\mathbf{A}+t\left(e_{ij}+\left(1-\delta_{ij}\right)e_{ji}\right)\right)-h\left(\mathbf{A}\right)\right)/t,
\]
where $e_{ij}$ is the matrix whose only non-zero entry is the $\left(i,j\right)$
entry, which is equal to $1$. We will also use the notation 
\[
\partial_{i_{1},j_{1},...,i_{k},j_{k}}h\left(\mathbf{A}\right)=\frac{\partial}{\partial\mathbf{A}_{i_{1}j_{1}}}\cdots\frac{\partial}{\partial\mathbf{A}_{i_{k}j_{k}}}h\left(\mathbf{A}\right).
\]
Suppose that $X_{\mathbf{C}}\sim N\left(0,\mathbf{C}\right)$ is a
general Gaussian vector of length $k$ with density $\varphi_{\mathbf{C}}\left(x\right)$,
where $\mathbf{C}=\left(\mathbf{C}_{ij}\right)$ is a non-singular
covariance matrix. From integration by parts and the well known fact
that for $i\neq j$, 
\[
\frac{\partial}{\partial\mathbf{C}_{ij}}\varphi_{\mathbf{C}}\left(x\right)=\frac{\partial}{\partial x_{i}}\frac{\partial}{\partial x_{j}}\varphi_{\mathbf{C}}\left(x\right),
\]
one has that, for any function $w:\,\mathbb{R}^{k}\to\mathbb{R}$
with $O(|x|^{n})$ growth rate at infinity,
\[
\frac{\partial}{\partial\mathbf{C}_{ij}}\mathbb{E}\left\{ w\left(X_{\mathbf{C}}\right)\right\} =\int w\left(x\right)\frac{\partial}{\partial\mathbf{C}_{ij}}\varphi_{\mathbf{C}}\left(x\right)dx=\int\left(\frac{\partial}{\partial x_{i}}\frac{\partial}{\partial x_{j}}w\left(x\right)\right)\varphi_{\mathbf{C}}\left(x\right)dx.
\]

Therefore, by applying the above with the function $(\mathbf{A},\mathbf{B})\mapsto g(\mathbf{A})g(\mathbf{B})$
and $\left(\mathbf{W}_{N}^{\left(1\right)}\left(\rho\right),\mathbf{W}_{N}^{\left(2\right)}\left(\rho\right)\right)$,
treated as a vector of the on-and-above elements, we obtain 
\begin{align}
\frac{d^{k}}{d\rho^{k}}\hat{g}\left(\rho\right) & =\sum_{\forall l\leq k:\,1\leq i_{l}\leq j_{l}\leq N}\prod_{l=1}^{k}\left(1+\delta_{i_{l}=j_{l}}\right)\label{eq:020216}\\
 & \times\mathbb{E}\left\{ \left(\partial_{i_{1},j_{1},...,i_{k},j_{k}}g\left(\mathbf{W}_{N}^{\left(1\right)}\left(\rho\right)\right)\right)\left(\partial_{i_{1},j_{1},...,i_{k},j_{k}}g\left(\mathbf{W}_{N}^{\left(2\right)}\left(\rho\right)\right)\right)\right\} .\nonumber 
\end{align}
For $\rho=0$, $\mathbf{W}_{N}^{\left(1\right)}\left(0\right)$ and
$\mathbf{W}_{N}^{\left(2\right)}\left(0\right)$ are i.i.d and the
expectation in (\ref{eq:020216}) is equal to 
\[
\left(\mathbb{E}\left\{ \left(\partial_{i_{1},j_{1},...,i_{k},j_{k}}g\left(\mathbf{W}_{N}^{\left(1\right)}\left(0\right)\right)\right)\right\} \right)^{2},
\]
which proves that $\frac{d^{k}}{d\rho^{k}}\hat{g}\left(0\right)\geq0$.
Lastly, the fact that $\hat{g}\left(\rho\right)$ is a polynomial
function whenever $g$ is, follows from the fact that $\mathbf{W}_{N}^{\left(i\right)}\left(\rho\right)$
are jointly Gaussian and (\ref{eq:covrho}). Convexity on $[0,1]$
and (\ref{eq:z28}) are direct consequences since the coefficients
of the polynomial function are equal to $\frac{d^{k}}{d\rho^{k}}\hat{g}\left(0\right)/k!$.
\end{proof}

\subsection{\label{sub:l1}Proof of Lemma \ref{lem:n1}}

Lemma \ref{lem:1stmom} expresses $\mathbb{E}\left\{ {\rm Crt}_{N}\left(J_{N}\right)\right\} $.
By further conditioning on $U$, substituting (\ref{eq:z19}), and
using the fact that uniformly in $v\in J_{N}$, as $N\to\infty$,
\[
v^{2}=b_{N}^{2}+2u(v-b_{N})(1+o(1))
\]
(where $u$ and $b_{N}$ are related to $J_{N}$ as in the statement
of Lemma \ref{lem:n1}), we obtain that, as $N\to\infty$, 
\[
\mathbb{E}\left\{ {\rm Crt}_{N}\left(J_{N}\right)\right\} =\omega_{N}\left(\frac{p-1}{2\pi}(N-1)\right)^{\frac{N-1}{2}}\sqrt{\frac{N}{2\pi}}e^{-\frac{Nb_{N}^{2}}{2}}\mathbb{E}\left\{ \det\left(\mathbf{X}_{N-1}-\sqrt{N}\bar{b}_{N}I\right)\right\} \int_{J_{N}}g(v)dv,
\]
where uniformly in $v\in J_{N}$,
\[
g(v)=\exp\left\{ -(1+o(1))N\left(\mathfrak{S}(u)+u\right)(v-b_{N})+o(1)\right\} .
\]
Recall that $\mathfrak{S}(u)+u<0$ (see (\ref{eq:<})). Thus,
\[
\int_{J_{N}}g(v)dv=(1+o(1))\int_{J_{N}}\exp\left\{ -N\left(\mathfrak{S}(u)+u\right)(v-b_{N})\right\} dv,
\]
which completes the proof.\qed

\subsection{\label{sub:l2}Proof of Lemma \ref{lem:n3}}

By Lemma \ref{lem:2ndKR-refinedFrom} and with the definitions in
its statement, by conditioning on $U_{1}\left(r\right)$, $U_{2}\left(r\right)$,
\begin{align*}
\mathbb{E}\left\{ \left[\mbox{Crt}_{N}\left(J_{N}\right)\right]_{2}^{\rho_{N}}\right\}  & =C_{N}N\int_{-\rho_{N}}^{\rho_{N}}dr\cdot\left(\mathcal{G}\left(r\right)\right)^{N}\mathcal{F}\left(r\right)\int_{J_{N}\times J_{N}}du_{1}du_{2}\\
 & \varphi_{\Sigma_{U}\left(r\right)}(\sqrt{N}u_{1},\sqrt{N}u_{2})\mathbb{E}\left\{ \prod_{i=1,2}\left|\det\left(\mathcal{\mathbf{M}}_{N-1}^{\left(i\right)}\left(r,\sqrt{N}u_{1},\sqrt{N}u_{2}\right)\right)\right|\right\} ,
\end{align*}
where by straightforward analysis, as $r\to0$,
\begin{align*}
\varphi_{\Sigma_{U}\left(r\right)}(u_{1},u_{2}) & \triangleq\frac{1}{2\pi}\left(\det\left(\Sigma_{U}\left(r\right)\right)\right)^{-1/2}\exp\left\{ -\frac{1}{2}\left(u_{1},u_{2}\right)\left(\Sigma_{U}\left(r\right)\right)^{-1}\left(u_{1},u_{2}\right)^{T}\right\} \\
 & =\left(1+O\left(r^{p}\right)\right)\frac{1}{2\pi}\exp\left\{ -\frac{1}{2}\left(u_{1}^{2}+u_{2}^{2}\right)+\left(u_{1}+u_{2}\right)^{2}O(r^{p})\right\} ,
\end{align*}
$\mathcal{F}\left(r\right)=1+O(r)$ and $\mathcal{G}\left(r\right)=e^{-\frac{1}{2}r^{2}+O(r^{4})}$.
Also note that 
\[
\frac{\omega_{N-1}}{\omega_{N}}/\sqrt{\frac{N}{2\pi}}\overset{{\scriptstyle N\to\infty}}{\longrightarrow}1
\]
and that, as $N\to\infty$, uniformly in $u_{i}\in J_{N}$ (with $u$
and $b_{N}$ related to $J_{N}$ as in the statement of Lemma \ref{lem:n3}),
\[
u_{i}^{2}=b_{N}^{2}+2u(u_{i}-b_{N})(1+o(1)).
\]

Combining all of the above, we arrive at
\begin{align*}
\mathbb{E}\left\{ \left[\mbox{Crt}_{N}\left(J_{N}\right)\right]_{2}^{\rho_{N}}\right\}  & =\left(\mathfrak{C}_{N}(b_{N})\right)^{2}\sqrt{\frac{N}{2\pi}}\int_{-\rho_{N}}^{\rho_{N}}dr\cdot e^{-\frac{1}{2}Nr^{2}+N\cdot O(r^{4})}\\
 & \int_{J_{N}\times J_{N}}du_{1}du_{2}g(u_{1},u_{2})\frac{\mathbb{E}\left\{ \prod_{i=1,2}\left|\det\left(\mathcal{\mathbf{M}}_{N-1}^{\left(i\right)}\left(r,\sqrt{N}u_{1},\sqrt{N}u_{2}\right)\right)\right|\right\} }{\left(\mathbb{E}\left\{ \det\left(\mathbf{X}_{N-1}-\sqrt{N}\bar{b}_{N}I\right)\right\} \right)^{2}},
\end{align*}
where $\mathfrak{C}_{N}(x)$ is defined in (\ref{eq:C(u)}), $\mathbf{X}_{N-1}$
is a GOE matrix and as $N\to\infty,$ uniformly in $u_{i}\in J_{N}$,
\[
g(u_{1},u_{2})=(1+o(1))\exp\left\{ -N\sum_{i=1}^{2}u(u_{i}-b_{N})(1+o(1))\right\} .
\]
Note that from our assumption that $\rho_{N}\to0$ as $N\to\infty$,
\[
\lim_{N\to\infty}\sqrt{\frac{N}{2\pi}}\int_{-\rho_{N}}^{\rho_{N}}dr\cdot e^{-\frac{1}{2}Nr^{2}+N\cdot O(r^{4})}\leq1.
\]
Therefore, since $\mathfrak{S}(u)+u<0$ (see (\ref{eq:<})), Lemma
\ref{lem:n3} follows if we can show that as $N\to\infty$, uniformly
in $u_{i}\in J_{N}$ and $r\in(-\rho_{N},\rho_{N})$,
\begin{align}
 & \mathbb{E}\left\{ \prod_{i=1,2}\left|\det\left(\mathcal{\mathbf{M}}_{N-1}^{\left(i\right)}\left(r,\sqrt{N}u_{1},\sqrt{N}u_{2}\right)\right)\right|\right\} \nonumber \\
 & \leq(1+o(1))\left(\mathbb{E}\left\{ \det\left(\mathbf{X}_{N-1}-\sqrt{N}\bar{b}_{N}I\right)\right\} \right)^{2}\exp\left\{ N\mathfrak{S}(u)\sum_{i=1}^{2}(1+o(1))(b_{N}-u_{i})\right\} .\label{eq:n271}
\end{align}

By Lemma \ref{lem:n7} and Corollary \ref{cor:n6}, as $N\to\infty$,
uniformly in $u_{i}\in J_{N}$ and $r\in(-\rho_{N},\rho_{N})$,
\begin{align}
 & \mathbb{E}\left\{ \prod_{i=1,2}\left|\det\left(\mathcal{\mathbf{M}}_{N-1}^{\left(i\right)}\left(r,\sqrt{N}u_{1},\sqrt{N}u_{2}\right)\right)\right|\right\} \nonumber \\
 & \leq(1+o(1))\mathbb{E}\left\{ \prod_{i=1}^{2}\det\left(\mathbf{X}_{N-1}^{\left(i\right)}\left(r\right)-\sqrt{N}\bar{b}_{N}I\right)\right\} \exp\left\{ N\mathfrak{S}(u)\sum_{i=1}^{2}(1+o(1))(b_{N}-u_{i})\right\} ,\label{eq:z27}
\end{align}
with $\mathbf{X}_{N-1}^{\left(i\right)}\left(r\right)$ as defined
in Lemma \ref{lem:n7}. 

Since for $r=0$, $\mathbf{X}_{N-1}^{\left(1\right)}\left(0\right)$
and $\mathbf{X}_{N-1}^{\left(2\right)}\left(0\right)$ are i.i.d,
defining 
\[
\Phi_{X}\left(r\right):=\Phi_{X,N}\left(r\right)=\mathbb{E}\left\{ \prod_{i=1}^{2}\det\left(\mathbf{X}_{N-1}^{\left(i\right)}\left(r\right)-\sqrt{N}\bar{b}_{N}I\right)\right\} ,
\]
what remains to show is that $\Phi_{X}\left(r\right)=(1+o(1))\Phi_{X}\left(0\right)$
as $N\to\infty$, uniformly in $r\in(-\rho_{N},\rho_{N})$. We show
this by appealing to Lemma \ref{lem:n4}. First, suppose that $\mathbf{W}_{N-1}^{\left(i\right)}\left(r\right)$
are defined as in this lemma and set
\[
\Phi_{W}\left(r\right):=\Phi_{W,N}\left(r\right)=\mathbb{E}\left\{ \prod_{i=1}^{2}\det\left(\frac{1}{\sqrt{N-1}}\mathbf{W}_{N-1}^{\left(i\right)}\left(r\right)-\sqrt{N}\bar{b}_{N}I\right)\right\} .
\]

Since, in distribution, 
\[
\left(\mathbf{X}_{N-1}^{\left(1\right)}\left(r\right),\mathbf{X}_{N-1}^{\left(2\right)}\left(r\right)\right)=\frac{1}{\sqrt{N-1}}\left(\mathbf{W}_{N-1}^{\left(1\right)}\left(s\left(r\right)\right),\mathbf{W}_{N-1}^{\left(2\right)}\left(s\left(r\right)\right)\right),
\]
with $s\left(r\right)=\left(\mbox{sgn}\left(r\right)\right)^{p}\sqrt{\left|r\right|^{p-2}}$,
it follows that 
\[
\Phi_{X}\left(r\right)=\Phi_{W}\left(s\left(r\right)\right).
\]
Thus, it is enough to show that for any $\rho_{N}^{\prime}>0$ such
that $\rho_{N}^{\prime}\to0$, as $N\to\infty$, 
\begin{equation}
\Phi_{W}\left(r\right)=(1+o(1))\Phi_{W}\left(0\right),\mbox{\,\,\,\,\ uniformly in }r\in(-\rho_{N}^{\prime},\rho_{N}^{\prime}).\label{eq:qq2}
\end{equation}

Assume $\rho_{N}^{\prime}>0$ is such an arbitrary sequence. By Corollary
\ref{cor:n5}, 
\begin{align}
\Phi_{W}\left(1\right) & \leq C\Phi_{W}\left(0\right)=C\left(\mathbb{E}\left\{ \det\left(\mathbf{X}_{N-1}-\sqrt{N}\bar{b}_{N}I\right)\right\} \right)^{2},\label{eq:z26}
\end{align}
where $\mathbf{X}_{N-1}$ is a GOE matrix of dimension $N-1$ and
$C>0$ is an appropriate constant.

In the notation of Lemma \ref{lem:n4}, $\hat{g}(r)=\Phi_{W}\left(r\right)$
where $g(\mathbf{A})=\det\left(\frac{1}{\sqrt{N-1}}\mathbf{A}-\sqrt{N}\bar{b}_{N}I\right)$
is a polynomial function of the elements of the matrix $\mathbf{A}$.
Thus by Lemma \ref{lem:n4}, uniformly in $r\in(-\rho_{N}^{\prime},\rho_{N}^{\prime})$,
as $N\to\infty$,
\begin{align*}
\left|\Phi_{W}\left(r\right)-\Phi_{W}\left(0\right)\right| & \leq\rho_{N}^{\prime}\left(\Phi_{W}\left(1\right)-\Phi_{W}\left(0\right)\right)\\
 & \leq\rho_{N}^{\prime}(C-1)\Phi_{W}\left(0\right)=o(1)\Phi_{W}\left(0\right),
\end{align*}
and therefore (\ref{eq:qq2}) follows. This completes the proof of
Lemma \ref{lem:n3}.\qed

\section{Appendix I: Eigenvalues}

Let $\lambda_{i}=\lambda_{i}^{N}$, $i=1,...,N$ denote the eigenvalues
of an $N\times N$ GOE matrix and denote by 
\begin{equation}
L_{N}=\frac{1}{N}\sum_{i=1}^{N}\delta_{\lambda_{i}^{N}}\label{eq:emp_meas}
\end{equation}
the empirical measure of eigenvalues. The following two bounds on
the maximal eigenvalue, both proved in \cite{BDG}, are useful to
us. 
\begin{lem}
\label{lem:maxEigBd}\cite[Lemma 6.3]{BDG} For large enough $m$
and all $N$,
\[
\mathbb{P}\left\{ \max_{i=1}^{N}\left|\lambda_{i}\right|\geq m\right\} \leq e^{-Nm^{2}/9}.
\]
\end{lem}
\begin{thm}
\label{thm:maxeigLDP}\cite[Theorem 6.2]{BDG} The maximal eigenvalues
$\lambda_{+}^{N}=\max_{i=1}^{N}\lambda_{i}^{N}$ satisfy the large
deviation principle in $\mathbb{R}$ with speed $N$ and the good
rate function 
\[
I^{+}(x)=\begin{cases}
\int_{2}^{x}\sqrt{(z/2)^{2}-1}dz, & x\geq2,\\
\infty, & otherwise.
\end{cases}
\]

\end{thm}
Next, we state the LDP satisfied by $L_{N}$ proved in \cite{BAG97}.
Let $M_{1}\left(\mathbb{R}\right)$ be the space of Borel probability
measures on $\mathbb{R},$ and endow it with the weak topology, which
is compatible with the Lipschitz bounded metric $d_{LU}\left(\cdot,\cdot\right)$,
defined by 
\[
d_{LU}\left(\mu,\mu'\right)=\sup_{f\in\mathcal{F}_{LU}}\left|\int_{\mathbb{R}}fd\mu-\int_{\mathbb{R}}fd\mu'\right|,
\]
where $\mathcal{F}_{LU}$ is the class of Lipschitz continuous functions
$f:\,\mathbb{R}\to\mathbb{R}$, with Lipschitz constant $1$ and uniform
bound $1$. The specific form of the rate function in the LDP is of
no importance to us and will therefore not be included in the statement
below. 
\begin{thm}
\label{lem:GOELD}\cite[Theorem 2.1.1]{BAG97}There exists a good
rate function $J\left(\mu\right)$, for which $J\left(\mu\right)=0$
if and only if $\mu=\mu^{*}$, where $\mu^{*}$ is the semicircle
law (see (\ref{eq:semicirc})), and such that the empirical measure
$L_{N}$ satisfies the large deviation principle on $M_{1}\left(\mathbb{R}\right)$
with speed $N^{2}$ and the rate function $J\left(\mu\right)$.
\end{thm}
We finish with a corollary of the main theorem of \cite{Miroslav}.
\begin{cor}
\cite{Miroslav}\label{cor:Miroslav} Let $\mathbf{C}_{1}$, $\mathbf{C}_{2}$
be two (deterministic) real, symmetric $N\times N$ matrices and let
$\lambda_{j}\left(\mathbf{C}_{i}\right)$ denote the eigenvalues of
$\mathbf{C}_{i}$, ordered with non-decreasing absolute value. Suppose
that the number of non-zero eigenvalues of $\mathbf{C}_{2}$ is $d$
at most. Then,
\[
\left|\det\left(\mathbf{C}_{1}+\mathbf{C}_{2}\right)\right|\leq\prod_{i=1}^{N}\left(\left|\lambda_{i}\left(\mathbf{C}_{1}\right)\right|+\left|\lambda_{i}\left(\mathbf{C}_{2}\right)\right|\right),
\]
and if $\left|\lambda_{1}\left(\mathbf{C}_{2}\right)\right|>0$,
\begin{equation}
\left|\det\left(\mathbf{C}_{1}+\mathbf{C}_{2}\right)\right|\leq\left|\det\left(\mathbf{C}_{1}\right)\right|\left(1+\frac{\left|\lambda_{N}\left(\mathbf{C}_{2}\right)\right|}{\left|\lambda_{1}\left(\mathbf{C}_{1}\right)\right|}\right)^{d}.\label{eq:miroslav2}
\end{equation}

\end{cor}

\section{\label{sec:app-cov}Appendix II: Covariances, densities, and conditional
laws}

In this Appendix we study the covariance structure of 
\[
\left\{ f\left(\mathbf{n}\right),\nabla f\left(\mathbf{n}\right),\nabla^{2}f\left(\mathbf{n}\right),f\left(\boldsymbol{\sigma}\left(r\right)\right),\nabla f\left(\boldsymbol{\sigma}\left(r\right)\right),\nabla^{2}f\left(\boldsymbol{\sigma}\left(r\right)\right)\right\} ,
\]
where
\[
\boldsymbol{\sigma}\left(r\right)=\left(0,...,0,\sqrt{1-r^{2}},r\right),
\]
and prove Lemmas \ref{lem:dens} and \ref{lem:Hess_struct_2}. 

With the standard notation 
\[
\delta_{ij}=\begin{cases}
1 & \mbox{ if }i=j,\\
0 & \mbox{ otherwise},
\end{cases}
\]
in the lemma below we denote $\delta_{i=j}=\delta_{ij}$, $\delta_{i=j=k}=\delta_{ij}\delta_{jk}$,
$\delta_{i=j\neq k}=\delta_{ij}\left(1-\delta_{jk}\right)$, etc.
\begin{lem}
\label{lem:cov}For any $r\in\left[-1,1\right]$ there exists an orthonormal
frame field $E=\left(E_{i}\right)$ such that
\begin{align*}
\mbox{\ensuremath{\mathbb{E}}}\left\{ f\left(\mathbf{n}\right)f\left(\boldsymbol{\sigma}\left(r\right)\right)\right\}  & =r^{p},\\
\mbox{\ensuremath{\mathbb{E}}}\left\{ f\left(\mathbf{n}\right)E_{l}f\left(\boldsymbol{\sigma}\left(r\right)\right)\right\}  & =-\mbox{\ensuremath{\mathbb{E}}}\left\{ E_{l}f\left(\mathbf{n}\right)f\left(\boldsymbol{\sigma}\left(r\right)\right)\right\} =-pr^{p-1}\left(1-r^{2}\right)^{1/2}\delta_{l=N-1},\\
\mbox{\ensuremath{\mathbb{E}}}\left\{ f\left(\mathbf{n}\right)E_{k}E_{l}f\left(\boldsymbol{\sigma}\left(r\right)\right)\right\}  & =\mbox{\ensuremath{\mathbb{E}}}\left\{ E_{k}E_{l}f\left(\mathbf{n}\right)f\left(\boldsymbol{\sigma}\left(r\right)\right)\right\} =p\left(p-1\right)r^{p-2}\left(1-r^{2}\right)\delta_{l=k=N-1}-pr^{p}\delta_{k=l},\\
\mbox{\ensuremath{\mathbb{E}}}\left\{ E_{j}f\left(\mathbf{n}\right)E_{l}f\left(\boldsymbol{\sigma}\left(r\right)\right)\right\}  & =\left[pr^{p}-p\left(p-1\right)r^{p-2}\left(1-r^{2}\right)\right]\delta_{l=j=N-1}+pr^{p-1}\delta_{l=j\neq N-1},\\
\mbox{\ensuremath{\mathbb{E}}}\left\{ E_{j}f\left(\mathbf{n}\right)E_{k}E_{l}f\left(\boldsymbol{\sigma}\left(r\right)\right)\right\}  & =-\mbox{\ensuremath{\mathbb{E}}}\left\{ E_{k}E_{l}f\left(\mathbf{n}\right)E_{j}f\left(\boldsymbol{\sigma}\left(r\right)\right)\right\} =p\left(p-1\right)\left(p-2\right)r^{p-3}\left(1-r^{2}\right)^{3/2}\delta_{j=k=l=N-1}\\
 & -p\left(p-1\right)r^{p-2}\left(1-r^{2}\right)^{1/2}\times\\
 & \left[\left(\delta_{j=k\neq N-1}+r\delta_{j=k=N-1}\right)\delta_{l=N-1}+\left(\delta_{j=l<N-1}+r\delta_{j=l=N-1}\right)\delta_{k=N-1}\right]\\
 & -p^{2}r^{p-1}\left(1-r^{2}\right)^{1/2}\delta_{k=l}\delta_{j=N-1},\\
\mbox{\ensuremath{\mathbb{E}}}\left\{ E_{i}E_{j}f\left(\mathbf{n}\right)E_{k}E_{l}f\left(\boldsymbol{\sigma}\left(r\right)\right)\right\}  & =p\left(p-1\right)\left(p-2\right)\left(p-3\right)r^{p-4}\left(1-r^{2}\right)^{2}\delta_{i=j=k=l=N-1}\\
 & -p\left(p-1\right)\left(p-2\right)r^{p-3}\left(1-r^{2}\right)\left[4r\delta_{i=j=k=l=N-1}\right.\\
 & +r\delta_{i=j}\delta_{k=l=N-1}+r\delta_{i=j=N-1}\delta_{k=l}+\delta_{j=l=N-1}\delta_{i=k\neq N-1}\\
 & \left.+\delta_{i=k=N-1}\delta_{j=l\neq N-1}+\delta_{i=l=N-1}\delta_{j=k\neq N-1}+\delta_{j=k=N-1}\delta_{i=l\neq N-1}\right]\\
 & +p\left(p-1\right)r^{p-2}\\
 & \times\left[-2\left(1-r^{2}\right)\delta_{i=j=N-1}\delta_{k=l}+\left(\delta_{j=l\neq N-1}+r\delta_{j=l=N-1}\right)\left(\delta_{i=k\neq N-1}+r\delta_{i=k=N-1}\right)\right.\\
 & \left.+\left(\delta_{i=l\neq N-1}+r\delta_{i=l=N-1}\right)\left(\delta_{j=k\neq N-1}+r\delta_{j=k=N-1}\right)\right]\\
 & +p\left(p-1\right)r^{p-2}\left[-\left(1-r^{2}\right)\delta_{i=j}\delta_{l=k=N-1}+r^{2}\delta_{i=j}\delta_{k=l}\right]\\
 & -p\left(p-1\right)r^{p-2}\left(1-r^{2}\right)\delta_{i=j}\delta_{k=l=N-1}+pr^{p}\delta_{i=j}\delta_{k=l}.
\end{align*}

\end{lem}
Note that $r=1$ corresponds to the case $\boldsymbol{\sigma}\left(r\right)=\mathbf{n}$.
(This is the case considered in \cite[Lemma 3.2]{A-BA-C}.)
\begin{proof}
We begin by defining the orthonormal frame field $E$. Let $r\in[-1,1]$
and let $P_{\mathbf{n}}:\,\mathbb{S}^{N-1}\to\mathbb{R}^{N-1}$ be
the projection to $\mathbb{R}^{N-1}$,
\[
P_{\mathbf{n}}\left(x_{1},...,x_{N}\right)=\left(x_{1},...,x_{N-1}\right),
\]
set $\theta\in\left[-\pi/2,\pi/2\right]$ to be the angle such that
$\sin\theta=r$, and let $R_{\theta}$ be the rotation mapping 
\[
R_{\theta}\left(x_{1},...,x_{N}\right)=\left(x_{1},...,x_{N-2},\sin\theta\cdot x_{N-1}+\cos\theta\cdot x_{N},-\cos\theta\cdot x_{N-1}+\sin\theta\cdot x_{N}\right).
\]
Let $U$ and $V$ be neighborhoods of $\mathbf{n}$ and $\boldsymbol{\sigma}\left(r\right)$,
respectively. Assuming $U$ and $V$ are small enough, the restrictions
of $P_{\mathbf{n}}$ and $P_{\mathbf{n}}\circ R_{-\theta}$ to $U$
and $V$, respectively, are coordinate systems. 

On $\mbox{Im}\left(P_{\mathbf{n}}\right)$ and $\mbox{Im}\left(P_{\mathbf{n}}\circ R_{-\theta}\right)$,
the images of the charts above, define 
\[
\bar{f}_{1}=f\circ P_{\mathbf{n}}^{-1}\mbox{\,\,\,\ and\,\,\,}\bar{f}_{2}=f\circ\left(P_{\mathbf{n}}\circ R_{-\theta}\right)^{-1}.
\]
We let $E=(E_{i})$ be an orthonormal frame field on the sphere such
that (under the notation (\ref{eq:grad_Hess}))%
\footnote{The fact that such frame field exists can be seen from the following.
If we let $\left\{ \frac{\partial}{\partial x_{i}}\right\} _{i=1}^{N-1}$
be the pull-back of $\left\{ \frac{d}{dx_{i}}\right\} _{i=1}^{N-1}$
by $P_{\mathbf{n}}$, then $\left\{ \frac{\partial}{\partial x_{i}}(\mathbf{n})\right\} _{i=1}^{N-1}$
is an orthonormal frame at the north pole. For any point in $U$ we
can define an orthonormal frame as the parallel transport of $\left\{ \frac{\partial}{\partial x_{i}}(\mathbf{n})\right\} _{i=1}^{N-1}$
along a geodesic from $\mathbf{n}$ to that point. This yields an
orthonormal frame field on $U$, say $E_{i}(\boldsymbol{\sigma})=\sum_{j=1}^{N-1}a_{ij}(\boldsymbol{\sigma})\frac{\partial}{\partial x_{j}}(\boldsymbol{\sigma})$,
$i=1,...,N-1$. Working with the coordinate system $P_{\mathbf{n}}$
one can verify that at $x=0$ the Christoffel symbols $\Gamma_{ij}^{k}$
are equal to $0$, and therefore (see e.g. \cite[Eq. (2), P. 53]{DoCarmo})
the derivatives $\frac{d}{dx_{k}}a_{ij}(P_{\mathbf{n}}^{-1}(x))$
at $x=0$ are also equal to $0$. If $r=1$, i.e., $\boldsymbol{\sigma}\left(r\right)=\mathbf{n}$,
extend the orthonormal frame field $E_{i}(\boldsymbol{\sigma})$ to
the sphere arbitrarily. Otherwise, assume $U$ and $V$ are disjoint
and construct the frame field on $V$ similarly to $U$ and then extend
it to the sphere. %
} 
\begin{align*}
\left\{ f\left(\mathbf{n}\right),\nabla f\left(\mathbf{n}\right),\nabla^{2}f\left(\mathbf{n}\right)\right\}  & =\left\{ \bar{f}_{1}\left(0\right),\nabla\bar{f}_{1}\left(0\right),\nabla^{2}\bar{f}_{1}\left(0\right)\right\} ,\\
\left\{ f\left(\boldsymbol{\sigma}\left(r\right)\right),\nabla f\left(\boldsymbol{\sigma}\left(r\right)\right),\nabla^{2}f\left(\boldsymbol{\sigma}\left(r\right)\right)\right\}  & =\left\{ \bar{f}_{2}\left(0\right),\nabla\bar{f}_{2}\left(0\right),\nabla^{2}\bar{f}_{2}\left(0\right)\right\} ,
\end{align*}
where in $\mathbb{R}^{N-1}$, $\nabla\bar{f}_{i}$ and $\nabla^{2}\bar{f}_{i}$
are the usual gradient and Hessian.

Define $C\left(x,y\right)=\mbox{Cov}\left\{ \bar{f}_{1}\left(x\right),\bar{f}_{2}\left(y\right)\right\} $
on $\mbox{Im}\left(P_{\mathbf{n}}\right)\times\mbox{Im}\left(P_{\mathbf{n}}\circ R_{-\theta}\right)$,
and note that
\begin{align*}
C\left(x,y\right) & =\left(\rho\left(x,y\right)\right)^{p}\triangleq\left\langle P_{\mathbf{n}}^{-1}\left(x\right),\left(P_{\mathbf{n}}\circ R_{-\theta}\right)^{-1}\left(y\right)\right\rangle ^{p}\\
 & =\Big(\sum_{i=1}^{N-2}x_{i}y_{i}+rx_{N-1}y_{N-1}+\sqrt{1-r^{2}}x_{N-1}\sqrt{1-\left\langle y,y\right\rangle }\\
 & \,\,\,+r\sqrt{1-\left\langle x,x\right\rangle }\sqrt{1-\left\langle y,y\right\rangle }-\sqrt{1-r^{2}}y_{N-1}\sqrt{1-\left\langle x,x\right\rangle }\Big)^{p}.
\end{align*}
The lemma follows by a (straightforward, but long) computation of
the corresponding derivatives, using the well-known formula (cf. \cite[eq. (5.5.4)]{RFG}),
\[
\mbox{Cov}\left\{ \frac{d^{k}}{dx_{i_{1}}\cdots dx_{i_{k}}}\bar{f}_{1}\left(x\right),\frac{d^{l}}{dy_{i_{1}}\cdots dy_{i_{l}}}\bar{f}_{2}\left(y\right)\right\} =\frac{d^{k}}{dx_{i_{1}}\cdots dx_{i_{k}}}\frac{d^{l}}{dy_{i_{1}}\cdots dy_{i_{l}}}C\left(x,y\right).
\]

\end{proof}
The variables in Lemma \ref{lem:cov} are jointly Gaussian. Now that
we have their covariances, the required conditional laws can be computed
using the well-known formulas for the Gaussian conditional distribution
(see \cite[p. 10-11]{RFG}). We shall need the following notation.

Define, for any $r\in\left(-1,1\right)$,
\[
\begin{array}{ll}
a_{1}\left(r\right)=\frac{1}{p\left(1-r^{2p-2}\right)}, & a_{2}\left(r\right)=\frac{1}{p\left[1-\left(r^{p}-\left(p-1\right)r^{p-2}\left(1-r^{2}\right)\right)^{2}\right]},\\
a_{3}\left(r\right)=\frac{-r^{p-1}}{p\left(1-r^{2p-2}\right)}, & a_{4}\left(r\right)=\frac{-r^{p}+\left(p-1\right)r^{p-2}\left(1-r^{2}\right)}{p\left[1-\left(r^{p}-\left(p-1\right)r^{p-2}\left(1-r^{2}\right)\right)^{2}\right]},\\
b_{1}\left(r\right)=-p & b_{2}\left(r\right)=-pr^{p}\\
\,\,+a_{2}\left(r\right)p^{3}r^{2p-2}\left(1-r^{2}\right), & \,\,-a_{4}\left(r\right)p^{3}r^{2p-2}\left(1-r^{2}\right),\\
b_{3}\left(r\right)= & b_{4}\left(r\right)=p\left(p-1\right)r^{p-2}\left(1-r^{2}\right)\\
\,\, a_{2}\left(r\right)p^{2}\left(p-1\right)r^{2p-4}\left(1-r^{2}\right)\left[-\left(p-2\right)+pr^{2}\right], & \,\,-a_{4}\left(r\right)p^{2}\left(p-1\right)r^{2p-4}\left(1-r^{2}\right)\left[-\left(p-2\right)+pr^{2}\right].
\end{array}
\]

Define $\Sigma_{U}\left(r\right)=\left(\Sigma_{U,ij}\left(r\right)\right)_{i,j=1}^{2,2}$
by 
\begin{equation}
\Sigma_{U}\left(r\right)=-\frac{1}{p}\left(\begin{array}{cc}
b_{1}\left(r\right) & b_{2}\left(r\right)\\
b_{2}\left(r\right) & b_{1}\left(r\right)
\end{array}\right),\label{eq:26}
\end{equation}
and define $\Sigma_{Z}\left(r\right)=\left(\Sigma_{Z,ij}\left(r\right)\right)_{i,j=1}^{2,2}$
and $\Sigma_{Q}\left(r\right)=\left(\Sigma_{Q,ij}\left(r\right)\right)_{i,j=1}^{2,2}$
by 
\begin{align}
\Sigma_{Z,11}\left(r\right) & =\Sigma_{Z,22}\left(r\right)=p\left(p-1\right)-a_{1}\left(r\right)p^{2}\left(p-1\right)^{2}r^{2p-4}\left(1-r^{2}\right),\nonumber \\
\Sigma_{Z,12}\left(r\right) & =\Sigma_{Z,21}\left(r\right)=p\left(p-1\right)^{2}r^{p-1}-p\left(p-1\right)\left(p-2\right)r^{p-3}+a_{3}\left(r\right)p^{2}\left(p-1\right)^{2}r^{2p-4}\left(1-r^{2}\right),\nonumber \\
\Sigma_{Q,11}\left(r\right) & =\Sigma_{Q,22}\left(r\right)=2p\left(p-1\right)-a_{2}\left(r\right)\left(1-r^{2}\right)\left[p\left(p-1\right)r^{p-3}\left(pr^{2}-\left(p-2\right)\right)\right]^{2}\nonumber \\
 & -\left(b_{3}\left(r\right),\, b_{4}\left(r\right)\right)\left(\Sigma_{U}\left(r\right)\right)^{-1}\left(\begin{array}{c}
b_{3}\left(r\right)\\
b_{4}\left(r\right)
\end{array}\right).\nonumber \\
\Sigma_{Q,12}\left(r\right) & =\Sigma_{Q,21}\left(r\right)=p^{4}r^{p}-2p\left(p-1\right)\left(p^{2}-2p+2\right)r^{p-2}+p\left(p-1\right)\left(p-2\right)\left(p-3\right)r^{p-4}\nonumber \\
 & +a_{4}\left(r\right)p^{2}r^{2p-6}\left(1-r^{2}\right)\left(p^{2}r^{2}-\left(p-1\right)\left(p-2\right)\right)^{2}\nonumber \\
 & -\left(b_{1}\left(r\right)+b_{3}\left(r\right),\, b_{2}\left(r\right)+b_{4}\left(r\right)\right)\left(\Sigma_{U}\left(r\right)\right)^{-1}\left(\begin{array}{c}
b_{2}\left(r\right)+b_{4}\left(r\right)\\
b_{1}\left(r\right)+b_{3}\left(r\right)
\end{array}\right).\label{eq:86}
\end{align}
Lastly, define 
\begin{align}
m_{1}\left(r,u_{1},u_{2}\right) & =\left(b_{3}\left(r\right),b_{4}\left(r\right)\right)\left(\Sigma_{U}\left(r\right)\right)^{-1}\left(u_{1},u_{2}\right)^{T},\label{eq:m_i}\\
m_{2}\left(r,u_{1},u_{2}\right) & =m_{1}\left(r,u_{2},u_{1}\right).\nonumber 
\end{align}

\begin{rem}
\label{rem:a}By standard analysis $1\pm\left(pr^{p}-\left(p-1\right)r^{p-2}\right)$,
and thus the denominators of $a_{i}\left(r\right)$ above, are positive
for any $r\in\left(-1,1\right)$. It is straightforward to verify
that
\begin{align*}
 & \left(\Sigma_{U,11}\left(r\right)\pm\Sigma_{U,12}\left(r\right)\right)\left(1\mp\left(pr^{p}-\left(p-1\right)r^{p-2}\right)\right)\\
 & =\left(1-r^{2}\right)(p-1)\left[\frac{1+r^{2}+\cdots+r^{2p-4}}{p-1}\pm r^{p-2}\right].
\end{align*}
Thus, from (\ref{eq:bd1}), $\Sigma_{U,11}\left(r\right)\pm\Sigma_{U,12}\left(r\right)>0$
for any $r\in\left(-1,1\right)$. Since these are the two eigenvalues
of $\Sigma_{U}\left(r\right)$, it is strictly positive definite for
$r\in\left(-1,1\right)$. In Lemma \ref{lem:non-degeneracy} we shall
prove that $\Sigma_{Z}\left(r\right)$ is strictly positive definite
for $r\in\left(-1,1\right)$. In the proof of Lemmas \ref{lem:dens}
and \ref{lem:Hess_struct_2}, we show that $\Sigma_{Q}\left(r\right)$
is semi positive definite.

Finally, we turn to the proof of Lemmas \ref{lem:dens} and \ref{lem:Hess_struct_2}.
\end{rem}

\subsection{\label{sub:Proof-of-Lemmas}Proof of Lemmas \ref{lem:dens} and \ref{lem:Hess_struct_2}}

Fix $r\in\left(-1,1\right)$ and let $E$ be the orthonormal frame
field defined in the proof of Lemma \ref{lem:cov}. We remind the
reader that 
\[
\nabla f_{N}\left(\boldsymbol{\sigma}\right)=\left(E_{i}f_{N}\left(\boldsymbol{\sigma}\right)\right)_{i=1}^{N-1},\,\,\nabla^{2}f_{N}\left(\boldsymbol{\sigma}\right)=\left(E_{i}E_{j}f_{N}\left(\boldsymbol{\sigma}\right)\right)_{i,j=1}^{N-1}.
\]
Assume all vectors in the proof are column vectors and denote the
concatenation of any two vectors $v_{1}$, $v_{2}$ by $\left(v_{1};v_{2}\right)$.
The covariance matrix of the vector $\left(\nabla f\left(\mathbf{n}\right);\nabla f\left(\boldsymbol{\sigma}\left(r\right)\right)\right)$
can be extracted from Lemma \ref{lem:cov}. By standard calculations,
one can prove (\ref{eq:grad_dens}) and show that the inverse of the
covariance matrix is the block matrix
\[
G\left(r\right)=\left(\begin{array}{cc}
a_{1}\left(r\right)I_{N-1}+\left(a_{2}\left(r\right)-a_{1}\left(r\right)\right)e_{N-1,N-1} & a_{3}\left(r\right)I_{N-1}+\left(a_{4}\left(r\right)-a_{3}\left(r\right)\right)e_{N-1,N-1}\\
a_{3}\left(r\right)I_{N-1}+\left(a_{4}\left(r\right)-a_{3}\left(r\right)\right)e_{N-1,N-1} & a_{1}\left(r\right)I_{N-1}+\left(a_{2}\left(r\right)-a_{1}\left(r\right)\right)e_{N-1,N-1}
\end{array}\right),
\]
where $I_{N-1}$ is the $N-1\times N-1$ identity matrix and where
$e_{N-1,N-1}$ is the $N-1\times N-1$ matrix whose $N-1\times N-1$
element is $1$ and all others are $0$.

For any random vector $V$ let $\mathbb{E}V$ denote the corresponding
vector of expectations. From Lemma \ref{lem:cov}, denoting by $e_{i}$
the $1\times\left(2N-2\right)$ vector with the $i$-th entry equal
to $1$ and all others equal to $0$, we obtain
\begin{align*}
 & \mbox{\ensuremath{\mathbb{E}}}\left\{ f\left(\mathbf{n}\right)\cdot\left(\nabla f\left(\mathbf{n}\right);\nabla f\left(\boldsymbol{\sigma}\left(r\right)\right)\right)\right\} =-pr^{p-1}\left(1-r^{2}\right)^{1/2}e_{2N-2},\\
 & \mbox{\ensuremath{\mathbb{E}}}\left\{ f\left(\boldsymbol{\sigma}\left(r\right)\right)\cdot\left(\nabla f\left(\mathbf{n}\right);\nabla f\left(\boldsymbol{\sigma}\left(r\right)\right)\right)\right\} =pr^{p-1}\left(1-r^{2}\right)^{1/2}e_{N-1},\\
 & \mbox{\ensuremath{\mathbb{E}}}\left\{ E_{i}E_{j}f\left(\mathbf{n}\right)\cdot\left(\nabla f\left(\mathbf{n}\right);\nabla f\left(\boldsymbol{\sigma}\left(r\right)\right)\right)\right\} \\
 & \quad=\begin{cases}
0 & ,\left|\left\{ i,j,N-1\right\} \right|=3\\
p^{2}r^{p-1}\left(1-r^{2}\right)^{1/2}e_{2N-2} & ,i=j\neq N-1\\
p\left(p-1\right)r^{p-2}\left(1-r^{2}\right)^{1/2}e_{N-1+i} & ,i\neq j=N-1\\
p\left(p-1\right)r^{p-2}\left(1-r^{2}\right)^{1/2}e_{N-1+j} & ,j\neq i=N-1\\
\left(1-r^{2}\right)^{1/2}\left(p^{3}r^{p-1}-p\left(p-1\right)\left(p-2\right)r^{p-3}\right)e_{2N-2} & ,i=j=N-1,
\end{cases}\\
 & \mbox{\ensuremath{\mathbb{E}}}\left\{ E_{i}E_{j}f\left(\boldsymbol{\sigma}\left(r\right)\right)\cdot\left(\nabla f\left(\mathbf{n}\right);\nabla f\left(\boldsymbol{\sigma}\left(r\right)\right)\right)\right\} \\
 & \quad=\begin{cases}
0 & ,\left|\left\{ i,j,N-1\right\} \right|=3\\
-p^{2}r^{p-1}\left(1-r^{2}\right)^{1/2}e_{N-1} & ,i=j\neq N-1\\
-p\left(p-1\right)r^{p-2}\left(1-r^{2}\right)^{1/2}e_{i} & ,i\neq j=N-1\\
-p\left(p-1\right)r^{p-2}\left(1-r^{2}\right)^{1/2}e_{j} & ,j\neq i=N-1\\
-\left(1-r^{2}\right)^{1/2}\left(p^{3}r^{p-1}-p\left(p-1\right)\left(p-2\right)r^{p-3}\right)e_{N-1} & ,i=j=N-1.
\end{cases}
\end{align*}

Denoting by $\mbox{Cov}_{\nabla f}\left\{ X,Y\right\} $ the covariance
of two random variables $X$, $Y$ conditional on $\nabla f\left(\mathbf{n}\right)=\nabla f\left(\boldsymbol{\sigma}\left(r\right)\right)=0$
(and the covariance with no conditioning by $\mbox{Cov}\left\{ X,Y\right\} $),
we have (cf. \cite[p. 10-11]{RFG})
\[
\mbox{Cov}_{\nabla f}\left\{ X,Y\right\} =\mbox{Cov}\left\{ X,Y\right\} -\left(\mbox{\ensuremath{\mathbb{E}}}\left\{ X\cdot\left(\nabla f\left(\mathbf{n}\right);\nabla f\left(\boldsymbol{\sigma}\left(r\right)\right)\right)\right\} \right)^{T}G\left(r\right)\mbox{\ensuremath{\mathbb{E}}}\left\{ Y\cdot\left(\nabla f\left(\mathbf{n}\right);\nabla f\left(\boldsymbol{\sigma}\left(r\right)\right)\right)\right\} .
\]

Thus, under the conditioning, $f\left(\mathbf{n}\right)$, $f\left(\boldsymbol{\sigma}\left(r\right)\right)$,
$\nabla^{2}f\left(\mathbf{n}\right)$, and $\nabla^{2}f\left(\boldsymbol{\sigma}\left(r\right)\right)$
are jointly Gaussian and centered, and, by straightforward calculations,
\begin{align*}
\mbox{Cov}_{\nabla f}\left\{ f\left(\mathbf{n}\right),f\left(\mathbf{n}\right)\right\}  & =\mbox{Cov}_{\nabla f}\left\{ f\left(\boldsymbol{\sigma}\left(r\right)\right),f\left(\boldsymbol{\sigma}\left(r\right)\right)\right\} =\Sigma_{U,11}\left(r\right),\\
\mbox{Cov}_{\nabla f}\left\{ f\left(\mathbf{n}\right),f\left(\boldsymbol{\sigma}\left(r\right)\right)\right\}  & =\Sigma_{U,12}\left(r\right),\\
\mbox{Cov}_{\nabla f}\left\{ f\left(\mathbf{n}\right),E_{i}E_{j}f\left(\mathbf{n}\right)\right\}  & =\mbox{Cov}_{\nabla f}\left\{ f\left(\boldsymbol{\sigma}\left(r\right)\right),E_{i}E_{j}f\left(\boldsymbol{\sigma}\left(r\right)\right)\right\} =\delta_{ij}\left(b_{1}\left(r\right)+\delta_{i,N-1}b_{3}\left(r\right)\right),\\
\mbox{Cov}_{\nabla f}\left\{ f\left(\mathbf{n}\right),E_{i}E_{j}f\left(\boldsymbol{\sigma}\left(r\right)\right)\right\}  & =\mbox{Cov}_{\nabla f}\left\{ f\left(\boldsymbol{\sigma}\left(r\right)\right),E_{i}E_{j}f\left(\mathbf{n}\right)\right\} =\delta_{ij}\left(b_{2}\left(r\right)+\delta_{i,N-1}b_{4}\left(r\right)\right),
\end{align*}
\begin{align}
 & \mbox{Cov}_{\nabla f}\left\{ E_{i}E_{j}f\left(\mathbf{n}\right),E_{k}E_{l}f\left(\mathbf{n}\right)\right\} =\mbox{Cov}_{\nabla f}\left\{ E_{i}E_{j}f\left(\boldsymbol{\sigma}\left(r\right)\right),E_{k}E_{l}f\left(\boldsymbol{\sigma}\left(r\right)\right)\right\} \label{eq:cov1}\\
 & =\begin{cases}
2\delta_{ik}p\left(p-1\right)-pb_{1}\left(r\right)-pb_{3}\left(r\right)\left(\delta_{i,N-1}+\delta_{k,N-1}\right)\\
-\delta_{i,N-1}\delta_{k,N-1}a_{2}\left(r\right)\left(1-r^{2}\right)\left[p\left(p-1\right)r^{p-3}\left(pr^{2}-\left(p-2\right)\right)\right]^{2} & ,\, i=j,\, k=l,\\
p\left(p-1\right) & ,\, i=k\neq j=l,\, N-1\notin\left\{ i,j\right\} ,\\
\Sigma_{Z,11}\left(r\right) & ,\, i=k\neq j=l,\, N-1\in\left\{ i,j\right\} ,\\
0 & ,\,\mbox{if }\left|\left\{ i,j,k,l\right\} \right|\geq3,
\end{cases}\nonumber \\
 & \mbox{Cov}_{\nabla f}\left\{ E_{i}E_{i}f\left(\mathbf{n}\right),E_{j}E_{j}f\left(\boldsymbol{\sigma}\left(r\right)\right)\right\} \nonumber \\
 & =\begin{cases}
-pb_{2}\left(r\right)-pb_{4}\left(r\right)\left(\delta_{i,N-1}+\delta_{j,N-1}\right) & ,\, i\neq j\\
-pb_{2}\left(r\right)+2p\left(p-1\right)r^{p-2} & ,\, i=j\neq N-1\\
p^{4}r^{p}-2p\left(p-1\right)\left(p^{2}-2p+2\right)r^{p-2}+p\left(p-1\right)\left(p-2\right)\left(p-3\right)r^{p-4}\\
+a_{4}\left(r\right)p^{2}r^{2p-6}\left(1-r^{2}\right)\left(p^{2}r^{2}-\left(p-1\right)\left(p-2\right)\right)^{2}, & ,\, i=j=N-1,
\end{cases}\nonumber \\
 & \mbox{Cov}_{\nabla f}\left\{ E_{i}E_{j}f\left(\mathbf{n}\right),E_{i}E_{j}f\left(\boldsymbol{\sigma}\left(r\right)\right)\right\} \nonumber \\
 & =\begin{cases}
p\left(p-1\right)r^{p-2} & ,\,\left|\left\{ i,j,N-1\right\} \right|=3\\
\Sigma_{Z,12}\left(r\right) & ,\,\left|\left\{ i,j,N-1\right\} \right|=2,\, i\neq j,
\end{cases}\nonumber \\
 & \mbox{Cov}_{\nabla f}\left\{ E_{i}E_{j}f\left(\mathbf{n}\right),E_{k}E_{l}f\left(\boldsymbol{\sigma}\left(r\right)\right)\right\} =0,\mbox{ \,\,\ if }\left|\left\{ i,j,k,l\right\} \right|\geq3.\nonumber 
\end{align}
Note that, in particular, this shows that the law of $\left(f\left(\mathbf{n}\right),f\left(\boldsymbol{\sigma}\left(r\right)\right)\right)$
under the conditioning is as stated in the lemma. Also, from the above
it follows that $\Sigma_{Z}\left(r\right)$ is positive definite for
any $r\in\left(-1,1\right)$.

Let $\mbox{Cov}_{f,\nabla f}\left\{ X,Y\right\} $ denote the covariance
of two random variables $X$, $Y$ conditional on 
\begin{equation}
\nabla f\left(\mathbf{n}\right)=\nabla f\left(\boldsymbol{\sigma}\left(r\right)\right)=0,\, f\left(\mathbf{n}\right)=u_{1},\, f\left(\boldsymbol{\sigma}\left(r\right)\right)=u_{2}.\label{eq:cond}
\end{equation}
 (which is independent of the values $u_{i}$) Note that
\begin{align*}
 & \mbox{Cov}_{f,\nabla f}\left\{ X,Y\right\} =\mbox{Cov}_{\nabla f}\left\{ X,Y\right\} \\
 & -\left(\mbox{Cov}_{\nabla f}\left\{ X,f\left(\mathbf{n}\right)\right\} ,\mbox{Cov}_{\nabla f}\left\{ X,f\left(\boldsymbol{\sigma}\left(r\right)\right)\right\} \right)\left(\Sigma_{U}\left(r\right)\right)^{-1}\mbox{\ensuremath{\mathbb{E}}}\left(\mbox{Cov}_{\nabla f}\left\{ X,f\left(\mathbf{n}\right)\right\} ,\mbox{Cov}_{\nabla f}\left\{ X,f\left(\boldsymbol{\sigma}\left(r\right)\right)\right\} \right)^{T}.
\end{align*}

Clearly,
\begin{align*}
\left(b_{1}\left(r\right),\, b_{2}\left(r\right)\right)\left(\Sigma_{U}\left(r\right)\right)^{-1} & =-p\left(1,0\right),\\
\left(b_{2}\left(r\right),\, b_{1}\left(r\right)\right)\left(\Sigma_{U}\left(r\right)\right)^{-1} & =-p\left(0,1\right).
\end{align*}
Thus,
\begin{align*}
 & \mbox{Cov}_{f,\nabla f}\left\{ E_{i}E_{j}f\left(\mathbf{n}\right),E_{k}E_{l}f\left(\mathbf{n}\right)\right\} -\mbox{Cov}_{\nabla f}\left\{ E_{i}E_{j}f\left(\mathbf{n}\right),E_{k}E_{l}f\left(\mathbf{n}\right)\right\} \\
 & =\mbox{Cov}_{f,\nabla f}\left\{ E_{i}E_{j}f\left(\boldsymbol{\sigma}\left(r\right)\right),E_{k}E_{l}f\left(\boldsymbol{\sigma}\left(r\right)\right)\right\} -\mbox{Cov}_{\nabla f}\left\{ E_{i}E_{j}f\left(\boldsymbol{\sigma}\left(r\right)\right),E_{k}E_{l}f\left(\boldsymbol{\sigma}\left(r\right)\right)\right\} \\
 & =-\delta_{ij}\delta_{kl}\left(b_{1}\left(r\right)+\delta_{i,N-1}b_{3}\left(r\right),\, b_{2}\left(r\right)+\delta_{i,N-1}b_{4}\left(r\right)\right)\left(\Sigma_{U}\left(r\right)\right)^{-1}\left(\begin{array}{c}
b_{1}\left(r\right)+\delta_{k,N-1}b_{3}\left(r\right)\\
b_{2}\left(r\right)+\delta_{k,N-1}b_{4}\left(r\right)
\end{array}\right)\\
 & =\delta_{ij}\delta_{kl}\cdot p\left[b_{1}\left(r\right)+\left(\delta_{i,N-1}+\delta_{k,N-1}\right)b_{3}\left(r\right)\right]\\
 & \quad-\delta_{ij}\delta_{kl}\delta_{i,N-1}\delta_{k,N-1}\left(b_{3}\left(r\right),\, b_{4}\left(r\right)\right)\left(\Sigma_{U}\left(r\right)\right)^{-1}\left(\begin{array}{c}
b_{3}\left(r\right)\\
b_{4}\left(r\right)
\end{array}\right),\\
 & \mbox{Cov}_{f,\nabla f}\left\{ E_{i}E_{j}f\left(\mathbf{n}\right),E_{k}E_{l}f\left(\boldsymbol{\sigma}\left(r\right)\right)\right\} -\mbox{Cov}_{\nabla f}\left\{ E_{i}E_{j}f\left(\mathbf{n}\right),E_{k}E_{l}f\left(\boldsymbol{\sigma}\left(r\right)\right)\right\} \\
 & =-\delta_{ij}\delta_{kl}\left(b_{1}\left(r\right)+\delta_{i,N-1}b_{3}\left(r\right),\, b_{2}\left(r\right)+\delta_{i,N-1}b_{4}\left(r\right)\right)\left(\Sigma_{U}\left(r\right)\right)^{-1}\left(\begin{array}{c}
b_{2}\left(r\right)+\delta_{k,N-1}b_{4}\left(r\right)\\
b_{1}\left(r\right)+\delta_{k,N-1}b_{3}\left(r\right)
\end{array}\right)\\
 & =\delta_{ij}\delta_{kl}\cdot p\left[b_{2}\left(r\right)+\left(\delta_{i,N-1}+\delta_{k,N-1}\right)b_{4}\left(r\right)\right]\\
 & \quad-\delta_{ij}\delta_{kl}\delta_{i,N-1}\delta_{k,N-1}\left(b_{3}\left(r\right),\, b_{4}\left(r\right)\right)\left(\Sigma_{U}\left(r\right)\right)^{-1}\left(\begin{array}{c}
b_{4}\left(r\right)\\
b_{3}\left(r\right)
\end{array}\right).
\end{align*}

Combining the previous calculations, we arrive at
\begin{align*}
 & \mbox{Cov}_{f,\nabla f}\left\{ E_{i}E_{i}f\left(\mathbf{n}\right),E_{j}E_{j}f\left(\mathbf{n}\right)\right\} =\mbox{Cov}_{f,\nabla f}\left\{ E_{i}E_{i}f\left(\boldsymbol{\sigma}\left(r\right)\right),E_{j}E_{j}f\left(\boldsymbol{\sigma}\left(r\right)\right)\right\} \\
 & =\begin{cases}
0 & ,i\neq j\\
2p\left(p-1\right) & ,i=j\neq N-1\\
\Sigma_{Q,11}\left(r\right) & ,i=j=N-1,
\end{cases}\\
 & \mbox{Cov}_{f,\nabla f}\left\{ E_{i}E_{i}f\left(\mathbf{n}\right),E_{j}E_{j}f\left(\boldsymbol{\sigma}\left(r\right)\right)\right\} \\
 & =\begin{cases}
0 & ,i\neq j\\
2p\left(p-1\right)r^{p-2} & ,i=j\neq N-1\\
\Sigma_{Q,12}\left(r\right) & ,i=j=N-1.
\end{cases}
\end{align*}
For the cases of indices that do not appear above we have 
\begin{align*}
\mbox{Cov}_{f,\nabla f}\left\{ E_{i}E_{j}f\left(\mathbf{n}\right),E_{k}E_{l}f\left(\mathbf{n}\right)\right\}  & =\mbox{Cov}_{\nabla f}\left\{ E_{i}E_{j}f\left(\mathbf{n}\right),E_{k}E_{l}f\left(\mathbf{n}\right)\right\} ,\\
\mbox{Cov}_{f,\nabla f}\left\{ E_{i}E_{j}f\left(\boldsymbol{\sigma}\left(r\right)\right),E_{k}E_{l}f\left(\boldsymbol{\sigma}\left(r\right)\right)\right\}  & =\mbox{Cov}_{\nabla f}\left\{ E_{i}E_{j}f\left(\boldsymbol{\sigma}\left(r\right)\right),E_{k}E_{l}f\left(\boldsymbol{\sigma}\left(r\right)\right)\right\} ,\\
\mbox{Cov}_{f,\nabla f}\left\{ E_{i}E_{j}f\left(\mathbf{n}\right),E_{k}E_{l}f\left(\boldsymbol{\sigma}\left(r\right)\right)\right\}  & =\mbox{Cov}_{\nabla f}\left\{ E_{i}E_{j}f\left(\mathbf{n}\right),E_{k}E_{l}f\left(\boldsymbol{\sigma}\left(r\right)\right)\right\} .
\end{align*}
From the above it follows that $\Sigma_{Q}\left(r\right)$ is semi
positive definite for any $r\in\left(-1,1\right)$.

It is now easy to compare covariances and see that, conditional on
(\ref{eq:cond}), the law of 
\[
\left(\frac{\nabla^{2}f\left(\mathbf{n}\right)-\mathbb{E}\left\{ \nabla^{2}f\left(\mathbf{n}\right)\right\} }{\sqrt{Np\left(p-1\right)}},\frac{\nabla^{2}f\left(\boldsymbol{\sigma}\left(r\right)\right)-\mathbb{E}\left\{ \nabla^{2}f\left(\boldsymbol{\sigma}\left(r\right)\right)\right\} }{\sqrt{Np\left(p-1\right)}}\right)
\]
 is the same as that of 
\[
\left(\hat{\mathbf{M}}_{N-1}^{\left(1\right)}\left(r\right),\,\hat{\mathbf{M}}_{N-1}^{\left(2\right)}\left(r\right)\right).
\]

What remains is to show that the conditional expectation of $\nabla^{2}f\left(\mathbf{n}\right)$
and $\nabla^{2}f\left(\boldsymbol{\sigma}\left(r\right)\right)$ under
(\ref{eq:cond}) are equal to 
\begin{equation}
-pu_{1}I+m_{1}\left(r,u_{1},u_{2}\right)e_{N-1,N-1}\mbox{\,\,\ and\,\,}-pu_{2}I+m_{2}\left(r,u_{1},u_{2}\right)e_{N-1,N-1},\label{eq:104}
\end{equation}
respectively. Denoting expectation conditional on (\ref{eq:cond})
by $\mathbb{E}_{f,\nabla f}^{u_{1},u_{2}}\left\{ \cdot\right\} $,
\begin{align*}
\mathbb{E}_{f,\nabla f}^{u_{1},u_{2}}\left\{ E_{i}E_{j}f\left(\mathbf{n}\right)\right\}  & =\left(\mbox{Cov}_{\nabla f}\left\{ E_{i}E_{j}f\left(\mathbf{n}\right),f\left(\mathbf{n}\right)\right\} ,\mbox{Cov}_{\nabla f}\left\{ E_{i}E_{j}f\left(\mathbf{n}\right),f\left(\boldsymbol{\sigma}\left(r\right)\right)\right\} \right)\left(\Sigma_{U}\left(r\right)\right)^{-1}\left(u_{1},u_{2}\right)^{T}\\
 & =\delta_{ij}\left(b_{1}\left(r\right)+\delta_{i,N-1}b_{3}\left(r\right),b_{2}\left(r\right)+\delta_{i,N-1}b_{4}\left(r\right)\right)\left(\Sigma_{U}\left(r\right)\right)^{-1}\left(u_{1},u_{2}\right)^{T}\\
 & =-\delta_{ij}pu_{1}+\delta_{ij}\delta_{i,N-1}\left(b_{3}\left(r\right),b_{4}\left(r\right)\right)\left(\Sigma_{U}\left(r\right)\right)^{-1}\left(u_{1},u_{2}\right)^{T}.
\end{align*}
Similarly,
\begin{align*}
\mathbb{E}_{f,\nabla f}^{u_{1},u_{2}}\left\{ E_{i}E_{j}f\left(\boldsymbol{\sigma}\left(r\right)\right)\right\}  & =-\delta_{ij}pu_{2}+\delta_{ij}\delta_{i,N-1}\left(b_{3}\left(r\right),b_{4}\left(r\right)\right)\left(\Sigma_{U}\left(r\right)\right)^{-1}\left(u_{2},u_{1}\right)^{T}.
\end{align*}
Which gives the required expectation (\ref{eq:104}). This completes
the proof.\qed

\section{\label{sec:app-nondeg}Appendix III: Regularity conditions for the
K-R formula}

In Section \ref{sec:ExactFormula:KR+Hessians} we needed to apply
the K-R Theorem to `count' pairs of different points $\left(\boldsymbol{\sigma},\boldsymbol{\sigma}'\right)\in\mathbb{S}^{N-1}\times\mathbb{S}^{N-1}$
at which $\nabla f_{N}\left(\boldsymbol{\sigma}\right)=\nabla f_{N}\left(\boldsymbol{\sigma}'\right)=0$
and $f_{N}\left(\boldsymbol{\sigma}\right),\, f_{N}\left(\boldsymbol{\sigma}'\right)\in\sqrt{N}B$.
The variant of the K-R Theorem we used is \cite[Theorem 12.1.1]{RFG}
which in particular accounts for the case where the parameter space
is a (Riemannian) manifold. It requires a long list of technical conditions
to be met (conditions (a)-(g) in the statement of the theorem) which
we discuss in this section. We start by relating our notation to that
of \cite[Theorem 12.1.1]{RFG}. 

In \cite[Theorem 12.1.1]{RFG}, $f\left(t\right)=\left(f^{1}\left(t\right),...,f^{N}\left(t\right)\right)$
is a random field on an $N$-dimensional manifold $M$ taking values
in $\mathbb{R}^{N}$, $\nabla f\left(t\right)=\left(E_{j}f^{i}\left(t\right)\right)_{i,j=1}^{N}$
is its Jacobian matrix (where $E$ is a fixed orthonormal frame field),
and $h\left(t\right)=\left(h^{1}\left(t\right),...,h^{K}\left(t\right)\right)$
is an additional random field from $M$ to $\mathbb{R}^{K}$. Those
$f$, $\nabla f$, and $h$ correspond to our $\left(\nabla f_{N}\left(\boldsymbol{\sigma}\right),\nabla f_{N}\left(\boldsymbol{\sigma}'\right)\right)$,
$J\left(\boldsymbol{\sigma},\boldsymbol{\sigma}'\right)$, and $\left(f_{N}\left(\boldsymbol{\sigma}\right),f_{N}\left(\boldsymbol{\sigma}'\right)\right)$,
respectively, where $J\left(\boldsymbol{\sigma},\boldsymbol{\sigma}'\right)$
is defined as the Jacobian matrix of $\left(\nabla f_{N}\left(\boldsymbol{\sigma}\right),\nabla f_{N}\left(\boldsymbol{\sigma}'\right)\right)$
with respect to the orthonormal frame field $E$. That is, if $E_{i}\left(\boldsymbol{\sigma}\right)$
(respectively, $E_{j}\left(\boldsymbol{\sigma}'\right)$) is considered
as a derivation with respect to the first (respectively, second) coordinate
of $f_{N}\left(\boldsymbol{\sigma},\boldsymbol{\sigma}'\right)$,
then $J\left(\boldsymbol{\sigma},\boldsymbol{\sigma}'\right)$ is
the block matrix
\[
J\left(\boldsymbol{\sigma},\boldsymbol{\sigma}'\right)\triangleq\left(E_{i'}\left(\boldsymbol{\sigma}_{i}\right)E_{j'}\left(\boldsymbol{\sigma}_{j}\right)f_{N}\left(\boldsymbol{\sigma},\boldsymbol{\sigma}'\right)\right)_{i,j=1}^{2N-2}=\left(\begin{array}{cc}
\nabla^{2}f_{N}\left(\boldsymbol{\sigma}\right) & 0\\
0 & \nabla^{2}f_{N}\left(\boldsymbol{\sigma}'\right)
\end{array}\right),
\]
where $i'=i\mbox{ mod }N-1$ and similarly for $j'$, and 
\[
\boldsymbol{\sigma}_{i}=\begin{cases}
\boldsymbol{\sigma} & \mbox{ if }i<N-1,\\
\boldsymbol{\sigma}' & \mbox{ if }i\geq N-2.
\end{cases}
\]

The manifold $M$ in our case is $\mathcal{S}_{N}^{2}\left(I_{R}\right)$
of (\ref{eq:S2N}) where $I_{R}$ is an open interval whose closure
is contained in $(-1,1)$.%
\footnote{In \cite[Theorem 12.1.1]{RFG} it is required that $M$ is compact
but going the proof of the theorem it can be seen that since in our
case $M=\mathcal{S}_{N}^{2}\left(I_{R}\right)$ has a finite atlas,
this requirement can be replaced by requiring conditions (a)-(g) to
hold on the closure of $\mathcal{S}_{N}^{2}\left(I_{R}\right)$.%
} Conditions (a), (f) and (g) of \cite[Theorem 12.1.1]{RFG} regarding
the continuity, moduli of continuity and moments of the involved random
fields are trivial consequences of the representation (\ref{eq:Hamiltonian})
of the Hamiltonian $H_{N}\left(\boldsymbol{\sigma}\right)$, Gaussianity
and stationarity. The remaining conditions concern the continuity
of certain conditional densities.%
\footnote{Though this is not explicit in the statement of \cite[Theorem 12.1.1]{RFG},
from its proof it can be seen that the support of the density of $\nabla f$
(which in our setting is $J\left(\boldsymbol{\sigma},\boldsymbol{\sigma}'\right)$)
can be any subspace $L\subset\mathbb{R}^{N^{2}}$ such that is $\det\left(\nabla f\right)$
has density whose support is $\mathbb{R}$. For example, in our case
$J\left(\boldsymbol{\sigma},\boldsymbol{\sigma}'\right)$ has entries
which are identically $0$.%
} Below we will prove the following lemma.
\begin{lem}
\label{lem:non-degeneracy}For any $r\in\left(-1,1\right)$, the Gaussian
array 
\begin{equation}
\left\{ \nabla f\left(\mathbf{n}\right),\nabla f\left(\boldsymbol{\sigma}\left(r\right)\right),\nabla^{2}f\left(\mathbf{n}\right),\nabla^{2}f\left(\boldsymbol{\sigma}\left(r\right)\right)\right\} ,\label{eq:nondeg2}
\end{equation}
is non-degenerate, up to symmetry of the Hessians. That is, if we
replace the Hessians in (\ref{eq:nondeg2}) by only their on-and-above
elements, then the support of the Gaussian density corresponding to
(\ref{eq:nondeg2}) is $\mathbb{R}^{2+(N-1)(N-2)}$.
\end{lem}
We wish to apply the K-R formula with $\sqrt{N}B$, the target set
of $f_{N}\left(\boldsymbol{\sigma}\right),\, f_{N}\left(\boldsymbol{\sigma}'\right)$,
being equal to an open interval or a finite union of such. Suppose
that instead of considering critical points $\boldsymbol{\sigma}$,
$\boldsymbol{\sigma}'$ with $f_{N}\left(\boldsymbol{\sigma}\right),\, f_{N}\left(\boldsymbol{\sigma}'\right)\in\sqrt{N}B$,
we consider critical points such that $f_{N}\left(\boldsymbol{\sigma}\right)+\epsilon g_{N}\left(\boldsymbol{\sigma}\right)$,
$f_{N}\left(\boldsymbol{\sigma}'\right)+\epsilon g_{N}\left(\boldsymbol{\sigma}'\right)\in\sqrt{N}B$
with $g_{N}\left(\boldsymbol{\sigma}\right)$ being a continuous Gaussian
field on $\mathbb{S}^{N-1}$ independent of $f_{N}\left(\boldsymbol{\sigma}\right)$
such that $\left(g_{N}\left(\boldsymbol{\sigma}\right),g_{N}\left(\boldsymbol{\sigma}'\right)\right)$
forms a non-degenerate Gaussian vector for any $\boldsymbol{\sigma}'\neq\pm\boldsymbol{\sigma}$.
In the latter case with $\epsilon>0$, the additional regularity conditions,
conditions (b)-(e) can be verified provided that Lemma \ref{lem:non-degeneracy}
holds. Then, by letting $\epsilon\to0$ we obtain that the K-R formula
holds for case $\epsilon=0$, which is what we wish to prove. Thus,
what remains is to prove the lemma.

\subsection*{Proof of Lemma \ref{lem:non-degeneracy}}

For $r=0$ the lemma can be verified from the covariance computations
of Lemma \ref{lem:cov}. Fix $r\in\left(-1,1\right)\setminus\{0\}$.
It will be enough to show that: 1. $\left(\nabla f\left(\mathbf{n}\right),\nabla f\left(\boldsymbol{\sigma}\left(r\right)\right)\right)$
is non-degenerate and that conditional on $\left(\nabla f\left(\mathbf{n}\right),\nabla f\left(\boldsymbol{\sigma}\left(r\right)\right)\right)=0$,
and 2.$\left(\nabla^{2}f\left(\mathbf{n}\right),\nabla^{2}f\left(\boldsymbol{\sigma}\left(r\right)\right)\right)$
is non-degenerate (in the sense as in the statement of the lemma).
The first of the two follows directly from the covariance computations
of Lemma \ref{lem:cov}. From Lemma \ref{lem:Hess_struct_2} we have
that second condition follows if we are able to show that $\Sigma_{Z}\left(r\right)$is
invertible and that
\[
\left\{ \left(m_{1}\left(r,u_{1},u_{2}\right),m_{2}\left(r,u_{1},u_{2}\right)\right)\,:\, u_{1},\, u_{2}\in\mathbb{R}\right\} =\mathbb{R}^{2}.
\]

It can verified that
\[
\frac{\left(\Sigma_{Z,11}\left(r\right)\pm\Sigma_{Z,12}\left(r\right)\right)\left(1\mp r^{p-1}\right)}{p\left(p-1\right)}=1-r^{2p-4}\pm\left(p-2\right)r^{p-1}\mp\left(p-2\right)r^{p-3}.
\]
If $r\geq0$ or $p$ is odd, then
\[
\varpi\left(r\right)\triangleq1-r^{2p-4}-\left(p-2\right)r^{p-1}+\left(p-2\right)r^{p-3}>0.
\]
If $p$ is even, it can be verified that the derivative of $\varpi\left(r\right)$
has constant sign on $\left(-1,0\right)$, from which it follows,
by the fact that $\varpi\left(0\right)=1$ and $\varpi\left(-1\right)=0$,
that $\varpi\left(r\right)>0$ for any $r\in\left(-1,0\right)$. A
similar analysis shows that 
\[
1-r^{2p-4}+\left(p-2\right)r^{p-1}-\left(p-2\right)r^{p-3}>0.
\]
This proves that $\Sigma_{Z}\left(r\right)$ is strictly positive
definite for $r\in(-1,1)$. 

By definition (see (\ref{eq:m_i})),
\[
\left(\begin{array}{c}
m_{1}\left(r,u_{1},u_{2}\right)\\
m_{2}\left(r,u_{1},u_{2}\right)
\end{array}\right)=\left(\begin{array}{cc}
b_{3}\left(r\right) & b_{4}\left(r\right)\\
b_{4}\left(r\right) & b_{3}\left(r\right)
\end{array}\right)\left(\Sigma_{U}\left(r\right)\right)^{-1}\left(\begin{array}{c}
u_{1}\\
u_{2}
\end{array}\right),
\]
where we recall that $\Sigma_{U}\left(r\right)$ invertible as shown
in Remark \ref{rem:a}. Thus, it is enough to show that $b_{3}\left(r\right)\pm b_{4}\left(r\right)\neq0$
(and therefore the matrix above is invertible). From straightforward
algebra,
\[
b_{3}(r)\pm b_{4}(r)=p(p-1)r^{p-2}(1-r^{2})\frac{r^{p-2}\pm1}{1\mp\left(r^{p}-\left(p-1\right)r^{p-2}\left(1-r^{2}\right)\right)}.
\]
As mentioned in Remark \ref{rem:a}, $1\pm\left(pr^{p}-\left(p-1\right)r^{p-2}\right)>0$
and therefore the denominator above is positive. This completes the
proof.\qed

\section{Appendix IV: upper bound on the ground state from moments equivalence
on exponential scale}

In this appendix we show how Theorem \ref{thm:Var-E2-log} can be
used to prove that 
\begin{equation}
\lim_{N\to\infty}GS^{N}=-E_{0},\,\,\,\mbox{almost surely.}\label{eq:a1}
\end{equation}
The fact that (\ref{eq:a1}) holds was already proved in \cite{A-BA-C}
based on fact that pure models are 1-RSB. The proof below is based
on the equivalence of second and first moment squared only on the
exponential level -- a fact which may be useful when investigating
general mixed models which are not known to exhibit 1-RSB.

The Borell-TIS inequality \cite{Borell,TIS} (see also \cite[Theorem 2.1.1]{RFG})
gives, for $\epsilon>0$,

\begin{equation}
\mathbb{P}\left\{ \left|GS^{N}-\mathbb{E}\left\{ GS^{N}\right\} \right|>\epsilon\right\} \leq\exp\left\{ -\epsilon^{2}N/2\right\} .\label{eq:Borel}
\end{equation}
From the Borel-Cantelli lemma that in order to prove (\ref{eq:a1}),
it is sufficient to show that 

\begin{equation}
\lim_{N\to\infty}\mathbb{E}\left\{ GS^{N}\right\} =-E_{0}.\label{eq:a2}
\end{equation}

Note that 
\begin{equation}
GS^{N}<u\,\,\Longleftrightarrow\,\,{\rm Crt}_{N}\left(\left(-\infty,u\right)\right)\geq1.\label{eq:o2}
\end{equation}
Thus, by Markov's inequality, Theorem \ref{thm:A-BA-C}, and the definition
of $E_{0}$, 
\begin{equation}
\limsup_{N\to\infty}\mathbb{P}\left\{ GS^{N}<-E_{0}-\epsilon\right\} =\limsup_{N\to\infty}\mathbb{P}\left\{ {\rm Crt}_{N}\left(\left(-\infty,-E_{0}-\epsilon\right)\right)\geq1\right\} \leq\lim_{N\to\infty}e^{-NC_{\epsilon}}=0,\label{eq:o1}
\end{equation}
for any $\epsilon>0$, where $C_{\epsilon}>0$ is a constant depending
on $\epsilon$. 

Now, assume towards contradiction that, for some $\delta>0$, $N_{k}\to\infty$,
\[
\liminf_{N\to\infty}\mathbb{E}\left\{ GS^{N}\right\} =\lim_{k\to\infty}\mathbb{E}\left\{ GS^{N_{k}}\right\} \leq-E_{0}-\delta.
\]
Then, from (\ref{eq:Borel}),
\[
\lim_{k\to\infty}\mathbb{P}\left\{ GS^{N_{k}}<-E_{0}-\delta/2\right\} \geq\lim_{k\to\infty}\mathbb{P}\left\{ \left|GS^{N_{k}}-\mathbb{E}\left\{ GS^{N_{k}}\right\} \right|\leq\delta/4\right\} =1,
\]
which contradicts (\ref{eq:o1}).

Next, assume towards contradiction that , for some $\delta>0$, $N_{k}\to\infty$,
\[
\limsup_{N\to\infty}\mathbb{E}\left\{ GS^{N}\right\} =\lim_{k\to\infty}\mathbb{E}\left\{ GS^{N_{k}}\right\} \geq-E_{0}+\delta.
\]
Then, from (\ref{eq:Borel}),
\[
\limsup_{k\to\infty}\frac{1}{N_{k}}\log\left(\mathbb{P}\left\{ GS^{N_{k}}<-E_{0}\left(p\right)+\delta/2\right\} \right)\leq\lim_{k\to\infty}\frac{1}{N_{k}}\log\left(\mathbb{P}\left\{ \left|GS^{N_{k}}-\mathbb{E}\left\{ GS^{N_{k}}\right\} \right|>\delta/4\right\} \right)\leq-\delta^{2}/32.
\]

On the other hand, from the Paley-Zygmund inequality and (\ref{eq:o2}),
\begin{align*}
\liminf_{k\to\infty}\frac{1}{N_{k}}\log\left(\mathbb{P}\left\{ GS^{N_{k}}<-E_{0}\left(p\right)+\delta/2\right\} \right) & =\liminf_{k\to\infty}\frac{1}{N_{k}}\log\left(\mathbb{P}\left\{ {\rm Crt}_{N_{k}}\left(\left(-\infty,-E_{0}+\delta/2\right)\right)\geq1\right\} \right)\\
 & =\liminf_{k\to\infty}\frac{1}{N_{k}}\log\left(\frac{\left(\mathbb{E}\left\{ \mbox{Crt}_{N_{k}}\left(\left(-\infty,-E_{0}\left(p\right)+\delta\right]\right)\right\} \right)^{2}}{\mathbb{E}\left\{ \left(\mbox{Crt}_{N_{k}}\left(\left(-\infty,-E_{0}\left(p\right)+\delta\right]\right)\right)^{2}\right\} }\right)=0,
\end{align*}
which, of course, contradicts the previous inequality. Hence, (\ref{eq:a2})
and therefore (\ref{eq:a1}) follow.

\bibliographystyle{alpha}
\bibliography{master}

\end{document}